\documentclass{article}[11 pt] 
\usepackage{amsmath}
\usepackage{amsfonts,amssymb}
\usepackage{graphicx}

\usepackage{amsthm}

\newtheorem{prop}{Proposition}
       \newtheorem{thm}{Theorem}[section]
              \newtheorem{coro}{Corollary}[section]
       \newtheorem{example}{Example}[section]

       \newtheorem{lem}[thm]{Lemma}

       \theoremstyle{definition}
       \newtheorem{dfn}{Definition}
        
       \theoremstyle{rmk}
       \newtheorem*{rmk}{Remark}

\usepackage[applemac]{inputenc}

\newcommand{\lec}{\preccurlyeq}
\newcommand{\gec}{\succcurlyeq}

\newcommand{\C}{\mathbb{C}}

\newcommand{\Bc}{\mathcal{B}}
\newcommand{\Fc}{\mathcal{F}}

\newcommand{\cov}{\mathrm{cov}}
\newcommand{\Var}{\mathrm{Var}}

\newcommand{\vp}{\varphi}

\newcommand{\Ec}{\mathcal{E}}
\newcommand{\Lc}{\mathcal{L}}
\newcommand{\Ic}{\mathcal{I}}
\newcommand{\Z}{\mathbb{Z}}
\newcommand{\R}{\mathbb{R}}

\newcommand{\N}{\mathbb{N}}

\newcommand{\Sy}{\mathfrak{S}}

\newcommand{\Sc}{\mathcal{S}}
\newcommand{\Sk}{\mathsf{Sk}}
\newcommand{\Ac}{\mathcal{A}}
\newcommand{\Rc}{\mathcal{R}}
\newcommand{\Gc}{\mathcal{G}}
\newcommand{\Mca}{\mathcal{M}}
\newcommand{\Cc}{\mathcal{C}}

\newcommand{\U}{\mathbb{U}}
\newcommand{\ts}{\otimes} 
\newcommand{\Bts}{\bigotimes} 

\newcommand{\Nu}{\mathcal{V}}

\newcommand{\Tc}{\mathcal{T}}

\newcommand{\ep}{\epsilon}

\newcommand{\Gbb}{\mathbb{G}}
\newcommand{\Ebb}{\mathbb{E}}

\newcommand{\Vbb}{\mathbb{V}}
\newcommand{\Fbb}{\mathbb{F}}

\newcommand{\Dd}{\mathrm{D}}
\newcommand{\Un}{\mathrm{U}}
\newcommand{\Pd}{\mathrm{P}}

\newcommand{\Ld}{\mathrm{L}}
\newcommand{\YM}{\mathrm{YM}}

\newcommand{\RL}{\mathrm{RL}}
\newcommand{\RP}{\mathrm{RP}}

\newcommand{\Nc}{\mathcal{N}}

\newcommand{\Pc}{\mathcal{P}}

\newcommand{\Pk}{\mathfrak{P}}

\newcommand{\Jc}{\mathcal{J}}

\newcommand{\Wc}{\mathcal{W}}

\newcommand{\Prob}{\mathbb{P}}
\newcommand{\esp}{\mathbb{E}}

\newcommand{\GL}{\mathrm{GL}}

\newcommand{\uN}{{\mathfrak{u}(N)}}

\newcommand{\End}{\mathrm{End}}
\newcommand{\supp}{\mathrm{supp}}

\newcommand{\Tr}{\mathrm{Tr}}

\newcommand{\Id}{\mathrm{Id}}

\newcommand{\la}{\langle}
\newcommand{\ra}{\rangle}

\renewcommand{\Im}{\mathrm{Im}\,}

\def\build#1_#2^#3{\mathrel{\mathop{\kern 0pt#1}\limits_{#2}^{#3}}}

\usepackage{color}

\usepackage{enumerate}

\usepackage[citecolor=cyan,colorlinks=true,linkcolor=blue]{hyperref}

\title{{Free energies and fluctuations for  the unitary Brownian motion}}

\author{Antoine Dahlqvist\thanks{Statistical Laboratory, Centre for Mathematical Sciences, 
Wilberforce Road, Cambridge, CB3 0WA, United Kingdom, email address: ad814@maths.cam.ac.uk } }

\begin{document}
\maketitle

\begin{abstract}We show that the Laplace transforms of  traces of words in independent unitary Brownian motions converge towards an analytic function on a non trivial disc.  These results allow to study  the asymptotic 
behavior of Wilson loops under the unitary Yang-Mills measure on the plane with a potential. The limiting objects obtained are shown to be characterized  by equations analogue to Schwinger-Dyson's ones, named here after 
Makeenko and Migdal. 
\end{abstract}
\section{Introduction}

The following paper aims at studying traces of non-commutative  polynomials in independent Brownian motions on the  group of unitary matrices $\Un(N),$ as the  size $N$ goes to infinity.   In 
\cite{BianeMBL,Xu,ThierrySW,MF}, it has been shown that for Brownian motions invariant by conjugation, with a proper time-scale, these  traces, properly normalized, converge towards a deterministic limit given by the 
evaluation of the free Brownian motion.  We want here to study the Laplace transform of these random variables with normalization analogue to the one of the mod-$\phi$ convergence (\cite{AFM});  it can be viewed as an 
analogue of the well known Harisch-Chandra-Itzykson-Zuber integrals (often abbreviated as HCIZ, \cite{HC,ZZ}). As a corollary, we obtain the fluctuations of the latter traces around their  limit.  In \cite{ThierryMylene}, the 
fluctuation of traces in polynomials of one marginal were given, this second  point of the present work gives an extension of their result.  In \cite{CK},  G. C\'ebron and T. Kemp have obtained  the existence of  Gaussian 
fluctuations  of  analogue random variables for diffusions on $\GL_N(\C)$. Therein, the  main result is obtained by an exact computation of the moments. In our situation, compactness allows to go beyond the characterization 
of fluctuations and  to answer   analytical questions that are not answered for the HCIZ integrals.   A second motivation of our paper is to study the planar Yang-Mills measure for large unitary groups, as well as planar Yang-
Mills measure with a potential. We are able here to show the convergence to all orders of the Wilson loops and prove that the limiting objects are characterized by analogues of Schwinger-Dyson equations, named after 
Makeenko and Migdal. In particular, we prove  the existence of a Gaussian field indexed by rectifiable loops describing the fluctuations of the convergence towards the master field proved in \cite{MF}.  In a subsequent joint 
work \cite{CDK}, we have obtained bounds on the speed of convergence of moments of a unitary Brownian motion, in order to show the strong convergence of the latter.  We believe that  an extension of the proof therein leads 
to larger lower bounds for the radiuses of convergence obtained here and  a result of strong convergence for holonomies of rectifiable loops. Though, in order to simplify the presentation, we shall not discuss it further here. Let 
us also highlight three new proofs of the Makeenko-Migdal equations in \cite{DHK}, discovered later on,   during the publication process of the present article.   In contrast with the previous ones of  \cite{MF} and of the current 
paper, the arguments are local and some of these proofs apply to any compact surface \cite{DGHK}.

\vspace{0,5 cm}

\textbf{Free energies of matrix models:} In many random matrix models, the asymptotic behavior of $\esp[e^{N\Tr(V)}],$ where $V$ is a fixed non-commutative polynomial in a sequence of random matrices of size $N$,  have  
been extensively studied and have several applications ranging from theoretical physics, through enumerative combinatorics, free probability and representation theory. A case of study is the HCIZ integral 
(\cite{HC,ZZ,CollinsGuionnetSegala,GGNHur}) 
$$H(A,B)=\esp[e^{N\Tr(AUBU^*)}],$$
where $A$ and $B$ are two deterministic Hermitian matrices and $U$ is a random unitary matrix, distributed according to the Haar measure.
 When the non-commutative polynomial plays the role of the potential of  a Gibbs measure, the  normalized logarithm of Laplace transforms is called the free energy and have been  studied in several places, for example in  
 \cite{NovakGuionnet,CollinsGuionnetSegala,BorotGuionnet}.  In the pioneering work  \cite{BIPZ},  formal  expansions have been proposed for several physical models. In \cite{Collins,HOPS}, technics have been developed to 
 study  formal expansions for model of random matrices with properties of invariance by conjugation.  Though, there are  yet few  results about the radius of convergence of these power series in the complex plane. See 
 \cite{Collins,CollinsGuionnetSegala} as well as \cite{GGN},  for a conjecture addressing this question for the Harisch-Chandra-Itzykson-Zuber integrals.     We have  managed  here to give  a converging expansion for the 
 following model.

Let $(U_{1,t_1},\ldots,U_{q,t_q})_{t\in \R_+^q}$ be  $q$ independent Brownian motions invariant by adjunction in $\Un(N)$ (see section 
\ref{section def} for a definition) and denote by $\Tr$ the usual non-normalized trace of matrices.
\begin{thm} \label{Theom Energy Words}For $t\in\R_+^q $  and any non-commutative polynomial $V$ in $2q$ variables, there exists 
$r_V>0$ and analytic functions $\vp_{t,V}, (\psi_{t,V,N})_{N\ge 1}$ and $\psi_{t,V}$ on $D_{r_V}=\{z\in\C: |z|<r_V\},$ such that 
$$ e^{\psi_{t,V,N}(z)}=\esp[e^{zN\Tr(V(U_{i,t_i},U_{i,t_i}^{*}, i=1..q))- N^2\vp_{t,V}(z)}]\longrightarrow e^{\psi_{t,V}(z)},$$
as $N\to\infty,$ where the convergence is uniform on compact subset of $D_{r_V}$.\end{thm}
For any non-commutative polynomial $V$ in $2q$ variables, whose restriction to unitary matrices is Hermitian-valued, we shall define for any integer $N\ge 1,$ a probability measure $\mu_{N,V}$ on $\Un(N)^{q}$ that is 
absolutely continuous with respect to the law of $(U_{i,t_i})_{1\le i \le q}$, with density proportional to  $e^{zN\Tr(V(U_{i,t_i},U_{i,t_i}^{*}, i=1..q))}.$  Then, for any $N\ge 1,$  $(U^V_{N,1}, \ldots,U_{N,q}^V)  $ denotes a random 
variable with law $\mu_{N,V}.$
 \begin{thm}\label{MF Potential}  If $V,W\in \C\la X_i,Y_i\ra_{i=1..q}$ are  non-commutative polynomials  with  small enough coefficients and $V^*=V$, then,  under the probability measure $\mu_{N,V}$, the random variable 
 $\frac{1}{N}\Tr(W(U^V_i,{U^V_i}^*, i=1..q))$ converges in probability towards a constant  $\Phi_{t,V}(W).$ 
\end{thm}


\textbf{Yang-Mills measure  on the plane:} We shall see that this result can be partly extended to the framework of Yang-Mills measure that has been developed in \cite{Driver,SenguptaYMCompact,ASMF,Champsmarkoholo,MF}. Therein, we give a recursive way to compute coefficients of $\vp_{t,V}(z)$,  proving analogues of Schwinger-Dyson equations, called here Makeenko-Migdal equations. The latter equations for the first coefficient in $z$ appeared in \cite{MM}  and were first  proved rigorously in \cite{MF}.  The Yang-Mills measure  encompasses the different models for  all $q\in\N^*$  and $ t\in \R_+^q,$ into one random object, for which the recursive equations has  a  simple  interpretation.   We shall use the approach of \cite{Champsmarkoholo,MF} by considering for any $N\ge 1,$ a process $(H_l)_l$ indexed by the set $\Ld(\R^2)$ of rectifiable loops in the plane, valued in $\Un(N)$, whose law will be denoted by  $\YM_N$.

\vspace{0,5 cm}

\textbf{Planar master field:} The works  \cite{ASMF,MF}  proved that under $\YM_N$, the random field $(\frac{1}{N}\Tr(H_l))_{l\in\Ld(\R^2)}$ converges in probability towards a deterministic field $(\Phi(l))_{l\in \Ld(\R^2)}$. The statement of this result first appeared in the physics literature, in the study of QCD,   with the  works \cite{KazakovMF,KazakovKostov,MM}, and in the mathematical paper \cite{IS}, as a conjecture. The limiting field was named therein  \emph{master field}, following the terminology of  \cite{Hooft}.  This object is  the first coefficient of an analytic function, limit of  Laplace transforms  appearing in a generalization of  \ref{Theom Energy Words}.  The asymptotic of $2D$-Yang-Mills   measure on other compact surfaces has also been investigated in the physics literature  \cite{GrossMat}.  It  won't be discussed in this text but could lead to future works.

\vspace{0,5 cm}

\textbf{Fluctuations:} The study of fluctuations of traces of random elements of a compact group of large dimension started with \cite{Diashah}, where is was investigated, thanks to representation theory tools, for the Haar measure on the classical compact Lie groups. The Theorem \ref{Theom Energy Words} allows us in particular to characterize the fluctuations in the convergence of the non-commutative distribution of a $\Un(N)$-Brownian motion towards the free unitary Brownian motion distribution.  We further prove that  under $\YM_N$ the random field $\left(\Tr(H_l)- \esp[\Tr(H_l)]\right)_{l\in \Ld(\R^2)}$ converges in law towards a  Gaussian field $(\phi_l)_{l\in \Ld(\R^2)}$, characterized by the Makeenko-Migdal equations.  Besides, we observe that when the loops are dilated by a factor $\lambda$, the above  fields have the same Gaussian behavior as $\lambda\to 0$. Our result  extends the work of \cite{ThierryMylene}, which study the Gaussian fluctuations in  the convergence  of the empirical measure of a $\Un(N)$-Brownian motion marginal. The Gaussian field obtained therein can be shown to be a deformation of the one obtained in \cite{Diashah}.   The fluctuation results presented in this text are extracted from the PhD thesis of the author, where the case of the orthogonal and  symplectic groups have also been addressed.  Note also that in \cite{DiaconisEvans,TCLF},  fluctuations with another scaling are considered to study finite blocks of a random matrix.    For the sake of simplicity, we shall restrict here to  the study of traces of words in the unitary case.  \vspace{0,5 cm}

\textbf{Organisation of the paper:} The next section is devoted to the description of the convention we use  for the standard Brownian motion on  $\Un(N)$ and  the choice of scaling we made.   In  sections 3 and 4,  are obtained the main expressions and estimates needed to get our result. In sections 5 and 6, we give their applications to study respectively the unitary Brownian motion and the Yang-Mills measure. In the last section, we show that the limited object obtained in the paper can be characterized by the recursive equations of Makeenko and Migdal. 

\tableofcontents


\section{Unitary Brownian motion and its large \texorpdfstring{$N$}{Lg} limit \label{section def}}


\subsection{Definition and time scale of unitary Brownian motion}  For any integer $N$, we shall write $\Un(N)$ for  the group of unitary matrices of $M_N(\C)$ and $\uN$ for its Lie algebra, that is,  the set of skew-Hermitian matrices. We define a scalar product   $\langle\cdot, \cdot\rangle$  on $\uN$  by setting for any $X,Y\in \uN,$
 $$\langle X,Y \rangle= - N \Tr(XY).$$
Let us write $(K_t)_{t\ge 0}$ the Brownian motion on the Euclidean space $(\uN, \langle \cdot,\cdot  \rangle)$ and recall that it is a Gaussian process such that for any $X,Y\in \uN, t,s\ge0$, $$\esp[ \langle X, K_t \rangle  \langle Y,  K_s\rangle]= \langle X ,Y\rangle \min (t,s).$$ 
Let us define $(U_t)_{t\ge 0}$ as the $M_N(\C)$-valued  solution of the following stochastic differential equation: 
\begin{align*}\label{EDSU}dU_t &=  U_tdK_t -\frac{1}{2}U_t dt   \tag{*}   \\
U_0  &= \Id.    
\end{align*}

\begin{lem}i) Almost surely, for all $t\ge 0,$ $U_t\in \Un(N)$.

ii)  For all $T\ge 0,$  $(U_T^*U_{T+t})_{t\ge 0}$ is independent of the sigma field $\sigma(U_s, s\le T)$ and has the same law as $(U_t)_{t\ge 0}$.

iii) For any $t\ge 0$ and every fixed $U\in \Un(N),$ $UU_tU^{-1}$ has the same law as $U_t.$
\end{lem}
\begin{proof} Let us prove the first point, the two others are left to the Reader. The processes $(i \sqrt{N} (K_t)_{p,p})_{t\ge 0}$ for $1\le p\le N$ and $(\sqrt{N}(K_t)_{i,j})_{t\ge 0}$ for  $1\le i<j\le N$   are $N^2$ independent processes,  the $N$ first have the same law as  standard real Brownian motions, whereas the others are distributed as standard complex Brownian motions, so that $\esp[|(K_1)_{1,2}|^2]=1.$ Let us denote by  $\la\!\langle \cdot  \ra\!\rangle$  the symbol of quadratic variations, so that 
  $$\la\!\la dK_t. dK_t\ra\!\ra =\sum_{1\le i,p,j\le N}\la\!\la  d(K_t)_{i,p} d(K_t)_{p,j}\ra\!\ra E_{i,j}= - dt \Id. $$
Itô's formula then yields
 $$d\left(U_t U_t^*\right)=  U_t (dK_t+dK_t^*)U_t^* +U_t (\la\!\la dK_t. dK_t^*\ra\!\ra- dt \Id )U_t^*=0.  $$\hfill\qed\end{proof}
\noindent We call  this process  the \emph{$\Un(N)$-Brownian motion}\footnote{It can be shown that it is a diffusion on the Riemannian manifold $\Un(N)$ endowed by the left-invariant metric associated to $\la\cdot,\cdot\ra$ and that its generator is the Laplace-Beltrami operator (see \cite{IkedaWatanabe} or \cite{ThierryMylene}, Proposition 2.1., for an elementary proof).} (see \cite{IntBM,MF} for a similar definition on other classical compact groups). For $N=1$, it has the same law as $(e^{iB_t})_{t\ge 0}$, where $(B_t)_{t\ge 0}$ is the standard real Brownian motion. Let us make  remarks on the scaling.  Recall that  the scalar product $\langle \cdot,\cdot \rangle$  on $\uN$ induces a Riemannian metric $d$ on $\Un(N)$. On the one hand,  this choice of metric  yields that the diameter of $\Un(N)$ is $d(\Id,-\Id)=\int_0^1 \|\dot{\gamma}_t\| dt,$ where $\gamma:t\in[0,1]\mapsto \exp(t i\pi \Id_N )$, that is, $ \| i\pi  \Id \| =  N\pi$. On the other hand,  the law of large numbers implies that $\dim( \uN )^{-1} \| K_t \|^2=N^{-2} \| K_t\|^2  $ converges, as $N\to \infty$, towards $ t$. Heuristically, we may infer that, as $N\to\infty$, for any $t>0$, $d(\Id, U_t)$ behaves like $\|K_t\|$  and   $\frac{d(U_t,  \Id)}{d(\Id,-\Id)}\to C_t\in (0,\infty)$. With this scaling, the Brownian motion ''has the time to visit''\footnote{Note that  a good scaling to study the convergence of the distance in total variation $d_{TV}$ between  the law of Brownian motion and the Haar measure, is  faster than ours. Let $U$ be a Haar distributed random variable on $\Un(N)$. It has been  shown in \cite{PLM} that  the  function $t\mapsto d_{TV}(U_{t\log(N)},U)$  admits a cut-off  around the value $t=2$. } $\Un(N).$   Besides,   the stochastic differential equation (\ref{EDSU}) does not depend on $N$ and  such an equation makes sense in the context of free stochastic differential equations (see \cite{BianeMBL}).  Let us add a last comment on the time-scale.  With the above choice, the $\Un(1)$-Brownian motion appears with the same scaling in all $\Un(N)$-Brownian motions.
\begin{lem} For any $N\in \N^*$, let $(U_{t,N})_{t\ge 0}$ be a $\Un(N)$-Brownian motion. Then,  the process $(\det(U_{t,N}))_{t\ge 0}$ has the same distribution as $(\Un_{t,1})_{t\ge 0}$.\label{det}
\end{lem} 
\begin{proof} Observe that for any $N\in\N^*$, $(i\Tr(K_t))_{t\ge 0}$  has the same law as a standard Brownian motion. If $D_2(\det)_M:M_N(\C)^2\to\C$ denotes the second derivative of the determinant at a point $M\in M_N(\C)$, Itô's formula yields that $$d \left(\det(U_t)\right)= \det(U_t) d\Tr(K_t)-\frac{N}{2} \det(\Un_t)dt+\det(U_t)\la\!\la D_2(\det)_{\Id}(dK_t,dK_t)\ra\!\ra.$$
What is more,
\begin{align*}
\la\!\la D_2(\det)_{\Id}(dK_t,dK_t)\ra\!\ra &=\sum_{1\le i<j\le N}\left(\la\!\la (dK_t)_{i,i} ,(dK_t)_{j,j} \ra\!\ra-\la\!\la (dK_t)_{j,i} ,(dK_t)_{i,j} \ra\!\ra\right)\\
&= \frac{N(N-1)}{2N}dt.
\end{align*}
Hence, $(\det(U_t))_{t\ge 0}$ is the unique strong solution to $$d\left(\det(U_t)\right)=\det(U_t)d\left(\Tr(K_t)\right)-\frac{1}{2}\det(U_t)dt$$
and $\det(U_0)=1,$ that is $(\exp( \Tr(K_t)))_{t\ge 0}.$
\hfill\qed\end{proof}
\subsection{Free unitary  Brownian motion\label{LimitdistUBM}}
Let us recall the first result obtained about the behavior of unitary Brownian motion in large dimension.  We shall denote by $(\mu^N_t)_{t\ge 0}$  the family of random measures given by the empirical measure of eigenvalues of $\Un_t$: if $\lambda_1,\ldots,\lambda_N \in \U$ are the eigenvalues of $\Un_t$, $\mu_t^N=\frac{1}{N}\left(\delta_{\lambda_1}+\cdots +\delta_{\lambda_N}\right)$.  Note that for any $P\in\C[X]$,   $\frac{1}{N}\Tr ( P(\Un_t))=\int_{\U} P(z)\mu_t(dz)$.  The following theorem has first been proved in \cite{BianeMBL} using harmonic analysis on the unitary group and by \cite{Rains} using stochastic calculus.
\begin{thm}[\cite{Rains,BianeMBL,Xu,ThierrySW}] \label{theoBiane}The sequence or random measures $(\mu_t^N)_{N\ge 0}$ converges   weakly in probability\footnote{\label{notetopo}We mean here that for any continuous function $f$, the sequence of random  variables $(\int f d\mu^N_t)_{N\ge 1}$ converges in probability to the constant $\int f d\mu_t$.},  towards a deterministic measure $\mu_t$ on $\U$, whose moments are given as follows: $$\mu_{t,n}= \int_{\U} z^n \mu_t(dz) = e^{-\frac{nt}{2}} \sum_{k=0}^{n-1} \frac{(-t)^k}{k!} n^{k-1}{n\choose k+1} .$$
\end{thm} 
Using the property of independence and stationarity satisfied by  multiplicative increments of a $\Un(N)$-Brownian motion,  together with the invariance of their law by  adjunction,   free probability arguments lead to the following Theorem.
\begin{thm}[\cite{BianeMBL}] For any $t_1,t_2,\ldots,t_q\ge 0$ and $V$ any non-commutative polynomial in $2q$-variables, the random variables
$ \frac{1}{N}\Tr(V(U_1,U_1^*,\ldots,U_q,U_q^*)) $ converge in probability towards a constant.
\end{thm}
The limiting object is called the non-commutative distribution of the \emph{free unitary Brownian motion} and can be characterized by  the family of measures $(\mu_t)_{t\ge 0}$ together with the asymptotic freeness of the increments. This last theorem was proved  in another way in \cite{Xu,ThierrySW,MF} showing directly  the convergence for any non-commutative polynomial.  Let us recall how the argument of \cite{ThierrySW,MF} goes to show Theorem \ref{theoBiane}.  Let us denote by $\Sy_n$ the group of permutations of 
$$[n]:=\{1,\ldots,n\}.$$
For any permutation $\sigma\in \Sy_n$ composed of  $\# \sigma$ cycles, we define a function $f_\sigma$ on $\Un(N)$  by setting  for any $U\in\Un(N),$
$$ f_\sigma (U)= N^{-\# \sigma}\Tr \left(\sigma U^{\ts n} \right)$$
and a function on $\Sy_n$, by setting for any $t>0$,
$$\varphi_t^N(\sigma)=\esp[f_\sigma( U_{t})],$$ 
where $U_{t}$ is the marginal of a $\Un(N)$-Brownian motion. Then, the latter family of functions on the symmetric group is shown to satisfy the following  differential system (see \cite{Rains,ThierrySW} or Lemma \ref{EDPTens}).
\begin{lem}[\cite{Rains,ThierrySW}] \label{lemEDPpol} For any permutation $\sigma\in\Sy_n$,  
\begin{align*}\frac{d}{dt} \varphi_t^N(\sigma) &= -\frac{n}{2} \varphi_t^N(\sigma) - \sum_{1\le i<j\le n}N^{\# \sigma (i  \hspace{0,1 cm} j)-\# \sigma-1} \varphi_t^N(\sigma (i \hspace{0,1cm} j)),\\
\varphi_0^N(\sigma)&=1.
\end{align*} 
\end{lem} 

The unique solution of this  system of ordinary differential equations  is a  power series  in $\frac{1}{N}$ that converges, as  $N\to\infty$, to a function $\varphi_t.$ It can further be shown to satisfy for any  $\sigma \in \Sy_n$ with $a_k$ cycles of length $k$, 
$$\varphi_t(\sigma)= \prod_{k=1}^n \varphi_t( ( 1 \cdots k))^{a_k}.$$
Setting for all $t\ge 0, n\ge1$, $\mu_{t,n}=  \varphi_t(( 1 \cdots  n))$, the limit in $N$  of the former equations takes the following form: 
\begin{equation} \frac{d}{dt}\mu_{t,n}= -\frac{n}{2}\mu_{t,n} - \frac{n}{2}\sum_{k=1}^{n-1} \mu_{t,k}\mu_{t,n-k}, \label{cut}
\end{equation} 
with initial condition $\mu_{0,n}=1.$  This system of equations is then shown to have a unique solution given by the expression of Theorem \cite{BianeMBL}. It  follows that for $n\in\N, t\ge 0$, $$\esp\left[\int_\U \omega^n\mu_t^N(d\omega) \right]\to \mu_{t,n}. $$ To conclude and obtain a convergence in probability, one ultimately needs to estimate the covariances of the complex variables $\left(\frac{1}{N}\Tr(U_t^n)\right)_{n\in \N, t\ge 0}$  with their complex conjugate.  This latter point together with the Lemma \ref{lemEDPpol} can be  proved using the following lemma, that allows to study  any polynomial in the entries and their conjugate of a unitary Brownian motion.

\vspace{0,5 cm}

 For any integer $n\in\N^*$, let us recall the left action of $\Sy_n$ on $\C^{\ts n},$ such that  for any permutation $\sigma\in \Sy_n$  and any elementary tensor  ${v_1}\ts v_2\ts \ldots \ts v_n\in (\C^N)^{\ts n},$
$$\sigma. {v_1}\ts v_2\ts \ldots \ts v_n= v_{\sigma^{-1}(1)}\ts \ldots v_{\sigma^{-1}(n)}. $$
The endomorphism of  $(\C^N)^{\ts n}$ associated to a permutation $\sigma$ will be abusively denoted below by the same symbol.
For any pair of distinct integers $i,j\in [n]$, we denote by $ \langle i\hspace{0,1cm} j\rangle$  the endomorphism of $(\C^N)^{\ts n}$  which acts like the  endomorphism $\sum_{1\le  r,s \le N}E_{r,s}\ts E_{r,s}$ on the $i$th  and $j$th  tensors  and trivially on the others.    

\begin{lem}[\cite{IntBM,MF}] \label{EDPTens}Let $U_{t}$ be a  Brownian motion on $\Un(N)$. For any positive integers  $a,b$, which add up to $n$, the following differential equation holds: 

$$\frac{d}{dt}\esp[U_t^{\ts a}\ts  \overline{U}_t^{\ts b }]= -\esp[U_t^{\ts a}\ts \overline{U}_t^{\ts b }]\left(\frac{n}{2} +\frac{1}{N} \sum_{\substack{i<j \le a\\ \text{ or }a<i<j}} ( i\hspace{0,1cm} j) -\frac{1}{N}\sum_{i\le a<j} \langle i \hspace{0,1cm} j \rangle \right).$$
\end{lem}

For any permutation $\sigma\in\Sy_n$, $\varphi^N_t(\sigma)= N^{-\# \sigma}\Tr(\sigma    \esp[U_t^{\ts n}])$   and the Lemma \ref{lemEDPpol} reduces to  this more general one.

\begin{proof} We shall use the stochastic differential equation (\ref{EDSU}) and apply Itô formula. First, writing the $\uN$-valued Brownian motion $(K_t)_{t\ge 0}$  as a sum of  independent real standard Brownian motions yields  that $$\la\!\langle dK_t \ts dK_t\ra\!\rangle =- \frac{1}{N}\sum_{1\le r,s \le N }E_{r,s } \ts  E_{s,r} dt = -\frac{1}{N}( 1 \hspace{0,1cm} 2) dt\in \End((\C^N)^{\ts 2})$$ 
and $$\la\!\langle dK_t \ts d\overline{K}_t \ra\!\rangle= \frac{1}{N}  \sum_{1\le r,s\le N}E_{r,s}\ts E_{r,s}= \frac{1}{N}\langle 1  \,2\rangle dt .$$
We can now use  the Itô formula to get that the variational-bounded part of the variation of the semi-martingales $U_t^{\ts 2}$ and $U_t\ts \overline{U_t}$ are respectively $ U_t \ts U_t  .(-dt+\la\!\langle dK_t \ts dK_t\ra\!\rangle)=- U_t^{\ts 2}(1+ \frac{1}{N}( 1\hspace{0,1cm} 2)) dt$ and  $U_t\ts \overline{U_t}(-dt+ \la\!\langle dK_t \ts d\overline{K}_t \ra\!\rangle )= -U_t\ts \overline{U_t}(1-\frac{1}{N}\langle 1 \, 2\rangle) dt$.  The same analysis yields that the variational-bounded part of the variation of the semi-martingale $U_t^{\ts a}\ts  \overline{U}_t^{\ts b }$ is
 $$U_t^{\ts a}\ts  \overline{U}_t^{\ts b }\left(-\frac{n}{2} -\frac{1}{N} \sum_{i<j \le a,\text{ or }a<i<j} ( i\hspace{0,1cm} j) +\frac{1}{N}\sum_{i\le a<j} \langle i \hspace{0,1cm} j \rangle \right)dt.$$\hfill\qed\end{proof}

\section{Free energy, words in unitary Brownian motions}

The convergence of the above  paragraph can be considered as a law of large numbers for the traces of words in  unitary Brownian motion.  We aim at studying their Laplace transform and at deriving from this study a central limit theorem. Note that if $(K_t)_{t\ge 0}$ is a $\uN$ Brownian motion as defined above, for any $t\ge 0,$  $N^{-2}\log\esp[e^{N\Tr(K_t)}]= N^{-2}\log\esp[e^{i \la K_t, i\Id_N\ra}]=-t$ and $N^{-2}\log\esp[e^{N\Tr(K_tK_t^*)}]=N^{-2}\log\esp[e^{-\|K_t\|^2}]=\log\esp [e^{-B_t^2}]$, where $B_t$ is the marginal of a standard real Brownian motion.  These two naive examples suggest that the scaling chosen in Theorem \ref{Theom Energy Words}  is the good one (see \cite{Zvonkin} for examples of Hermitian matrix models, where this scaling is used).  We shall prove it in the following by  estimating cumulants.
\subsection{Laplace transforms and cumulants of traces} 
\subsubsection{Scaling of cumulants\label{Scaling cumulants}} For any bounded random variable $X$,  the function $\log \esp[e^{zX} ]$   is analytic on a neighborhood of $0$. We denote its analytic expansion 
$$\log \esp[e^{zX} ]=\sum_{n\ge 1}\frac{C_n(X)}{n!}z^n ,$$
the coefficients $(C_n(X))_{n\ge 1}$ are called the \emph{classical cumulants} of the random variable $X$. 
 We are interested here  in the behavior in $N$ of $N^{-2}\log\esp[e^{N\Tr(A_N)}]$, hence of the rescaled cumulant $$N^{n-2}C_n(\Tr(A_N)),$$ where the $A_N$ are    random matrices of $M_N(\C)$, uniformly bounded in norm.  
 
 \subsubsection{Cumulants of several random variables}These coefficients are related to the moments of $X$ via a M\"obius inversion formula (\cite{Rota}, \cite[2.3]{Collins}) in a lattice of partitions.     For any $n\in\N^*,$ the set $\Pc_n$ of partitions of $[n]$  is endowed with a partial  order $\lec$, such that for $\pi,\nu\in\Pc_n$,  $\pi\lec \nu $ if the blocks of $\pi$ are included in the ones of $\nu$. It has a maximum and a minimum that we denote respectively by $1_n$ and $0_n$. For any pair $\pi,\nu\in \Pc_n$, $\pi\wedge\nu$  (respectively $\pi\vee \nu$) denotes the biggest (smallest) partition smaller (bigger) than $\pi$ and $\nu.$ Each partition $\pi$ has $\#\pi$  blocks. For any sequence of complex numbers $(\alpha_A)_{A\subset [n]},$ let us set  for any  partition $\pi\in\Pc_n,$    $$\alpha_{\pi}=\prod_{A\in \pi} \alpha_A.$$  
Then, for any $\pi\in \Pc_n,$ there exists a unique sequence $(\beta_{\pi,\nu}(\alpha))_{\pi\lec \nu}$ such that for any $\nu\gec\pi $, 
\begin{equation}
\alpha_{\nu}=\sum_{\pi \lec\pi'\lec \nu} \beta_{\pi, \pi'}(\alpha). \label{caract cumulants}
\end{equation}
For any $\pi,\nu\in \Pc_n$ with $\displaystyle \nu \lec \pi$ and $\pi\not=\nu$, we set $\beta_{\pi,\nu}(\alpha)=0.$
If $X_1\ldots, X_n$ are  bounded complex  random variables and for any $A\subset[n]$, $\alpha_A= \esp[\prod_{i \in A} X_{i}]$, let us set  $C_n(X_1,\ldots,X_n)= \beta_{0_n,1_n}(\alpha)$ and similarly, for any  pair  $(\pi,\nu)$ of partitions,  $C_{\pi,\nu}(X_1,\ldots, X_n)= \beta_{\pi,\nu}(\alpha)$.  Then, the following expansion holds for any $z\in \C^n$  in a neighborhood of $0$, 
$$\log\esp[e^{z_1X_1+\ldots+z_nX_n}]=\sum_{\substack{k\ge 1\\1\le i_1,\ldots,i_k\le n}}C_k(X_{i_1},\ldots,X_{i_k})\frac{z_{i_1}\ldots z_{i_k}}{k!}.$$
Differentiating the previous expression leads to two useful formulas: if $Y$ and $Z$ are bounded random variables coupled with $X$ and  if $z\in \C$ is in a neighborhood of $0,$  then
\begin{equation}
\frac{\esp[Ye^{zX}]}{\esp[e^{zX}]}=\sum_{k\ge 0}\frac{C_{k+1}(Y,X,\ldots,X)}{k!}z^k \label{Mean Pot}
\end{equation}
and 
\begin{equation}\label{VarPot}
\frac{\esp[Y Z e^{zX}]}{\esp[e^{zX}]}-\frac{\esp[Y e^{z X}]\esp[ Z e^{z X}]}{\esp[e^{zX}]^2}=\sum_{k\ge 0}\frac{C_{k+2}(Y,Z,X,\ldots,X)}{k!}z^k.
\end{equation}

\noindent The coefficient $C_n(X_1,\ldots,X_n)$  is called a \emph{classical cumulant} and is symmetric in the variables $X_1,\ldots ,X_n$.    The coefficients $(C_{\pi,\nu}(X_1,\ldots,X_n))_{\pi\lec\nu}$ are called \emph{relative cumulants}.   Using the characterizing formula (\ref{caract cumulants}), one can express relative cumulants in terms of classical cumulants.  For any pair $\pi\lec \nu$,  if  for any $A\in \nu,$ $A_\pi$  denotes the set of blocks of $\pi$ included in $A$, then
\begin{align}
C_{\pi,\nu}(X_1,\ldots,X_n)&= \prod_{A\in \nu} C_{\#A_\pi}\left(\prod_{i\in B} X_{i}, B\in A _\pi\right)\nonumber \\
&= \sum_{\mu\in \Pc_n: \pi\vee \mu=\nu}C_{\mu}(X_1,\ldots, X_n).\label{LeoS}
\end{align}
The second formula is   named after Leonov and  Shiryaev   \cite{LeonovS}.

\subsubsection{Tensor valued cumulants}  We shall here recall  the  notion of moments and cumulants  for vector valued random variables (see \cite{CollinsSniady}, where this notion is considered in the broader setting of non-commutative probability spaces). This framework will  later on (section  \ref{Sharper Bounds}) allow us to obtain representation formulas for classical cumulants in terms of matrices, which are the core of the argument to  bound  from below  the  radius  of convergence of Laplace transforms.  

For any  finite dimensional vector space $V$ and any finite set $A$, let us denote by $V^{\ts A} $ the vector space of multilinear map on $\left({V^*}\right)^A$ and for any $n\in\N^*$ identify  $V^{\ts [n]}$ with  $V^{\ts n}.$    Any function  $X: [n]\to V$ defines an elementary element of $V^{\ts A}$ that we denote by $\bigotimes_{i\in A} X_i $.   Any partition $\pi \in \Pc_n $ defines  a multilinear map $\prod_{A\in\pi} V^{\ts A}  \to V^{\ts n}, (\alpha_A)_{A\in\pi}\mapsto \bigotimes_{A\in \pi}\alpha_A ,$ such that for any $X\in V^n$, 
$\bigotimes_{A\in \pi} \bigotimes_{i\in A} X_i = \bigotimes_{i \in [n]} X_i.$  For instance, for $v_1,\ldots,v_4\in V$,  if $\alpha_{\{1,3\}}= v_1\ts v_2\in V^{\ts 2}\simeq V^{\ts\{1,3\}}$ and $\alpha_{\{2,3\}}=v_3\ts v_4\in V^{\ts 2}\simeq V^{\ts\{2,4\}}$,
$$\bigotimes_{A\in \{\{1,3\},\{2,4\}\}} \alpha_A= v_1\ts v_3\ts v_2\ts v_4.$$ 
For any sequence $(\alpha_A)_{A\subset [n]}$ such that for any $A\subset [n]$, $\alpha_A\in V^{\ts A}$, let us set for any partition $\pi\in\Pc_n$,
\begin{equation}
\alpha_\pi=\bigotimes_{A\in\pi}\alpha_A\in V^{\ts n}.\label{partitioned tensor}
\end{equation}
Then, for any $\pi\in \Pc_n,$ there exists a unique sequence $(\beta_{\pi,\nu}(\alpha))_{\pi\lec \nu}$ such that for any $\nu\gec\pi $, 
\begin{equation}
\alpha_{\nu}=\sum_{\pi \lec\pi'\lec \nu} \beta_{\pi, \pi'}(\alpha)\in V^{\ts n}.\label{Characterization Cumulants tens}
\end{equation}
When $X_1\ldots, X_n$ are  bounded random variables valued in $V$  on a probability space $(\Omega,\Bc,\Prob)$ and for any $A\subset [n]$, $\alpha_A= \esp[\Bts_{i \in A} X_{i}]$, note that for any  pair of partitions $\pi, \nu,$  $C_{\pi,\nu}(X_1, \ldots,  X_n)=\beta_{\pi,\nu}(\alpha)$ is  $n$-linear as a function on the space $\Ld^\infty(\Omega,\Bc,\Prob)\ts V$.  We then define linear functions  on $\Ld^\infty(\Omega,\Bc )\ts V^{\ts n}$  by setting for any random variables  $X_1,\ldots,X_n\in \Ld^{\infty}(\Omega,\Bc,\Prob)\ts V$, $C_{\pi,\nu}(X_1\ts \ldots\ts  X_n)= C_{\pi,\nu}(X_1, \ldots , X_n)$ and  $C_n(X_1\ts \ldots\ts X_n)=C_{0_n,1_n}(X_1, \ldots, X_n)$. For example, if $A$ and $B$ are two bounded random vectors of $V,$
$$C_2(A\ts B)= \esp[A\ts B]-\esp[A]\ts \esp[B]\in V^{\ts 2},$$
$$C_{\{\{1,2\},\{3\}\},1_3}(A\ts B \ts C)= \esp[A\ts B\ts C]-\esp[A\ts B]\ts \esp[C]\in V^{\ts 3}.$$
If $A_1,\ldots, A_n$ are random matrices in $M_n(\C^N)$  with bounded operator norms    and $\pi,\nu\in\Pc_n$ is a pair of partitions,  then 
$$C_{\pi,\nu}(\Tr(A_1),\Tr(A_2),\ldots, \Tr(A_n))= \Tr_{(\C^N)^{\ts n}}(C_{\pi,\nu}(A_1\ts\ldots\ts A_n)). $$
If $\sigma\in\Sy_n $ is a permutation whose orbits are included in blocks of the partition $\nu$, then 
\begin{equation}\label{Tensor Cumulants and usual Cumulants}
C_{\#\nu}(\prod_{\substack{(i_1\cdots i_k)\text{ cycle of }\sigma\\\text{included in }B}}\Tr(A_{i_1}\cdots A_{i_k}), B\in \nu)=\Tr( C_{\nu, 1_n}( \sigma A_1 \ts A_2\ts\ldots\ts A_n)).
\end{equation}

\subsubsection{Cumulant of  exponential tensors} For any $n\in \N^*$, let us denote by $\Pk_n$ the set of subsets of $[n].$ For any $A,B\in \Pk_n $, with $B\subset A$, let us write $\overline{B}^A$  for the smallest partition of $A$ for $\lec$, containing   $B$ as a block, and set for any endomorphism $T\in \End(V^{\ts B})$, $\overline{T}^A= \Bts_{S\in\{B,A\setminus B \}}\alpha_S \in \End(V^{\ts A}),$ where $\alpha_B= T$ and $\alpha_{A\setminus B}=\Id_{V^{\ts A\setminus B}}.$   If $A=[n]$, we shall write respectively $\overline{B}$ and $\overline{T}$ for  $\overline{B}^A$ and  $\overline{T}^A.$ For any $d\ge 1,$ let us  denote by $\Delta^d$ the simplex   $\{s\in[0,1]^{d+1}: s_0+\ldots+s_d=1 \}$ of dimension $d$ and by $ds$ the Lebesgue measure on $\Delta^d$. 
\begin{lem} \label{Lemma exp cumulants}Let $(T_A)_{A\in \Pk_n, \#A \ge 2}$ be a family of endomorphisms such that for any $A\in\Pk_n$, $T_A\in \End(V^{\ts A})$ and   $$\alpha_A= \exp[\sum_{B\subset A, \#B\ge 2}\overline{T_B}^A].$$  
Then,   for any pair of partitions $\pi\lec\nu $ in $\Pc_n,$ 
\begin{equation}
\beta_{\pi,\nu}(\alpha)=  \sum   \int_{\Delta^{d}} \alpha^{s_0}_\pi \overline{T}_{A_1}\alpha^{s_1}_{\pi\vee A_1}  \overline{T}_{A_2}  \ldots\overline{T}_{A_d}\alpha^{s_d}_\nu ds,\label{DuhamelTens}
\end{equation}
where the sum is over sequences $A_1,\ldots,A_d,$ such that $(\pi\vee ( \vee_{i=1}^{k} A_i))_{0\le k\le d}$ is strictly increasing, with $\pi \vee A_1\vee \ldots \vee A_d= \nu.$ Moreover, if  $\mu\in \Pc_n$ and  for any $t\in \R^{\mu},$
$$T_A(t)= t_C T_A , $$
whenever $A\subset C, $ with $C\in\mu$, and $0$ otherwise,  then  for any $C\in \mu$, 
\begin{equation}
\frac{d}{dt_C}\beta_{\pi,\nu}(\alpha)= \sum_{A\subset C: \#A\ge 2}   \overline{T}_A \beta_{\pi\vee A,\nu}( \alpha(t)).\label{diff cumulant exp tens}
\end{equation}
\end{lem}
Though we shall not use it, it is also enlightening  and easy to show that for  any  partitions $\pi\lec \nu$  of $[n],$  $$\beta_{\pi,\nu}(\alpha)= \sum_{k\ge0 }\frac{1}{k!} \sum  \overline{T_{A_1}}\ldots \overline{T_{A_k}}\in \End(V^{\ts n}),$$
where the second sum is over sequences $A_1,\ldots, A_k\in\Pk_n$, with $ \#A_i\ge 2$ for any $1\le i\le k$ and $\pi\vee A_1\vee A_2\vee \ldots \vee A_k= \nu$.

\begin{proof} Applying  the uniqueness of (\ref{Characterization Cumulants tens}) to the integral of the right-hand-side of  (\ref{diff cumulant exp tens}) implies the latter.  Then,  (\ref{diff cumulant exp tens})   for $\mu=1_n,$ implies that both sides of (\ref{DuhamelTens})  satisfies the same differential equation.  Moreover, both sides have the same initial condition:    $\beta_{\pi,\nu}(\alpha(0))=1,$ if $\pi=\nu$ and $0$ otherwise. Hence, by induction on $\#\pi-\#\nu,$ the two terms agree.
\hfill\qed\end{proof}

\subsection{Words in independent unitary Brownian motions and their traces} 

We obtain here a differential system for the normalized cumulants in traces of   words of unitary Brownian motions and show that the latter converge as $N\to\infty.$

\subsubsection{Partitioned  words\label{Section Partial words}}
For each positive integer $q$,  $W_q$  denotes  the monoid of words   in the alphabet made of  $2q$ symbols $x_1,\ldots,x_q ,x_1^{-1}, 
\ldots, x_q^{-1}$.    An element $w$ of $W_q$ writes down uniquely $x_{i_1}^{\epsilon_1}\ldots x_{i_n}^{\epsilon_n}$, with $
\epsilon_1,\ldots,\ep_n\in \{-1,1\}$. We call $n$  the length of $w$ and denote it by $\ell(w)$.  Its $p^{\text{th}}$ letter $x_{i_p}^
{\epsilon_p}$ is denoted by  $X_p(w)$ and for any $k\in [q]$ and $A\subset [n],$ we set $n^{\pm}_{w,A}(k) =\#\{r\in A: X_r(w)=x_{k}
^{\pm}  \}$, 
\begin{equation}
n_{w,A}(k)=n_{w,A}^+(k)-n_{w,A}^-(k)\, \text{ and  }\,\bar{n}_{w,A}(k)=n_{w,A}^+(k)+n_{w,A}^-(k),\label{index word}
\end{equation}
where we shall drop the second index when $A=[n].$   We call \emph{partitioned word} every couple $(S, \pi)$, where  $S$ is a tuple $
(w_1,\ldots,w_m)$ of   words  in $W_q$  and $\pi \in\Pc_m$.  Given such a couple, we set  $\# S= m$, $w(S)=w_1w_2\ldots w_m$ and $\ell(S)=
\ell
(w(S,\pi))$. We denote the set of partitioned words by $\Pc W_q$.    For any $m\in \N^*$ and $w\in W_q,$   $w^{[m]}$ denotes the tuple composed of $m$ copies of $w.$ For any function $\vp \in \C^{\Pc W_q},$ and $S\in W_q^m,$ we shall write $\vp(S)$ for  $\vp((S,0_m)).$


\subsubsection{Operations on  partitioned  words\label{Section CutJoin}}   For    any partitioned word $(S,\pi)\in \Pc W_q$ with $S=(w_1,\ldots, w_m)$ and $w=w(S,\pi)$, let us introduce two transformations of  $(S,\pi)$. For any pair  of positive integers $i,j$ such that  $1\le i<j \le \ell(S)$ and  $X_i(w)=X_j(w)^{\pm}=a$,  let us define  $\Tc^{\pm}_{i,j}((S,\pi))= (\Tc^{\pm}_{i,j}(S),\pi')\in \Pc W_q$  according to two cases. 

\vspace{0.2 cm}

\noindent 1. If the  $i^{\text{th}}$ and $j^{\text{th}}$  letters of $w$ belong to the same word $w_k=  \lambda X_i(w)\mu X_j(w) \nu $,  then let us  set  $$\Tc^{+}_{i,j}(S)=(w_1,\ldots, w_{k-1}, \lambda a \nu,a \mu ,w_{k+1},\ldots,w_m),$$
if $X_i(w)=X_j(w)$,  
$$\Tc^{-}_{i,j}(S)=(w_1,\ldots, w_{k-1}, \lambda \nu,a\mu a^{-1} ,w_{k+1},\ldots,w_m),$$
if $X_i(w)=X_j(w)^{-1}$
and  in both cases $\pi'\in \Pc_{m+1}$ the partition  obtained from $\pi $ by substituting $l $ with $l+1$ for $l>k$ and adding $k+1$ to any block of $\pi$ including $k$. 

\vspace{0.1 cm}

\noindent 2. If the  $i^{\text{th}}$ and $j^{\text{th}}$  letters of $w$ belong to two words $w_p=  \lambda X_i(w)\mu$ and $w_q=\nu  X_j(w) \chi $,  then  let us set   
$$\Tc^{+}_{i,j}(S)=(w_1,\ldots, \hat{w_p},\lambda a\chi\nu a\mu,\ldots,\hat{w_{q}},\ldots,w_m),$$
if $X_i(w)=X_j(w)$,  
$$\Tc^{-}_{i,j}(S)=(w_1,\ldots, \hat{w_p},\lambda \chi\nu a^{-1}a\mu,\ldots,\hat{w_{q}},\ldots,w_m),$$
if $X_i(w)=X_j(w)^{-1}$ and in both cases we let $\pi'\in \Pc_{m-1}$ be the image of the restriction of  $\pi $ to $[m] \setminus \{q\}$ by the increasing bijection $[m]\setminus\{q\} \to [m-1].$ 

\vspace{0,2 cm}

Let us make two remarks. In the first case the number of blocks of the partitioned words are constant and the number of words is increased by $1$. In the second one, the number of words is decreased by $1$, whereas the number of blocks of $\#\pi'$  is equal to $\#\pi$, if $p$ and $q$ belong to the same block of $\pi$ and $\#\pi-1$ otherwise.  For any $f\in [q],$ we define the sets 
$$\Nc^\pm_{w}(f)=\{(i,j)\in [\ell(w)]^2: i<j, X_i(f)=X_j^\pm(f)\in \{x_f,x_f^{-1}\}\},$$
$$\Nc^{0,\pm}_{S,\pi}(f)=\{(i,j)\in \Nc^\pm_{w}(f) : \#\pi'=\#\pi-1 \text{ or }\# \Tc^{\pm}_{i,j}(S) =\# S+1 \},$$
$$\Nc^{2,\pm}_{S,\pi}(f)=\{(i,j)\in  \Nc^\pm_{w}(f): \#\pi'=\#\pi \text{ and }\# \Tc^{\pm}_{i,j}(S) =\# S-1 \},$$
\begin{equation}
\Sc_{S,\pi}(f)=\{(i,j)\in \Nc^\pm_{w}(f): \#\pi'=\#\pi\} \label{non fusion}
\end{equation}
and $$\Nc_{w}=\bigcup_{f=1}^q\Nc_{w}^+(f)\cup\Nc_{w}^-(f).$$

\begin{figure}\label{Cut Operation}\centering \includegraphics[width=84 mm,height=44 mm]{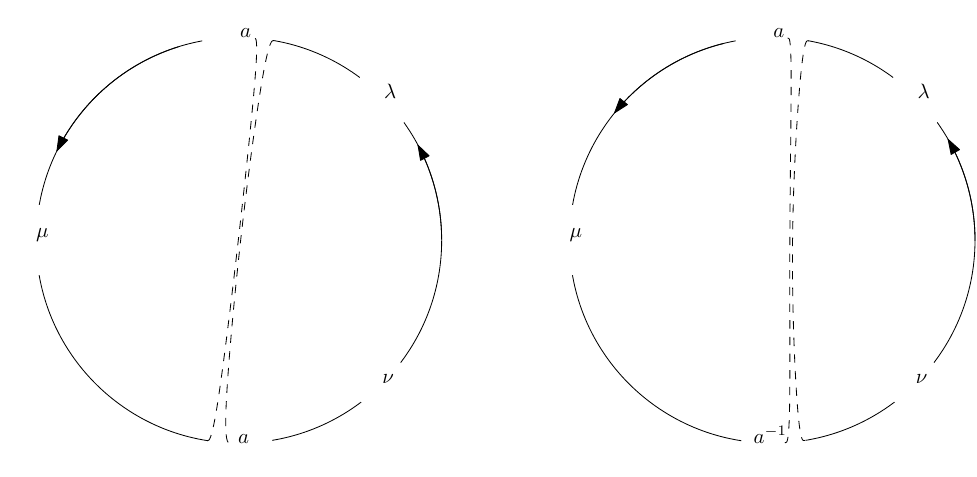}\caption{Cut  transformations}
\end{figure}
\begin{figure}\label{Join Operation}\centering \includegraphics[width=120 mm,height=1.5 in]{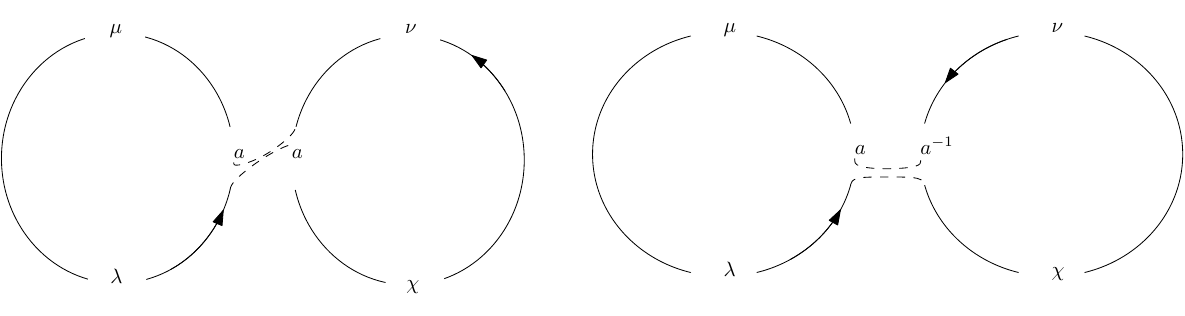}\caption{Join  transformations}
\end{figure}

\subsubsection{A differential system\label{Section differential system}} 
Let us fix $q$ independent and identically distributed $\Un(N)$-valued  Brownian motions $(U_{1,s})_{s\ge 0},\ldots,(U_{q,s})_{s\ge 0}.$  We would like to show that Lemma \ref{lemEDPpol} admits a  natural generalization for cumulants  instead of moments of traces in this multiple times setting. This will be achieved in Proposition \ref{PropoSystDiff Cut and Join}.  For any word decomposed as $w=x_{i_1}^{\epsilon_1}\ldots x_{i_n}^{\epsilon_n}$, with $
\epsilon_1,\ldots,\ep_n\in \{-1,1\}$  and $U_1,\ldots ,U_q\in \Un(N),$ let us set 
$$w(U_1,\ldots, U_q)=U_{i_1}^{\ep_1}\ldots U_{i_n}^{\ep_n}$$
and for any vector $t\in \R_+^q,$  
\begin{equation}
\label{mots Browniens ind}w_t^N = w(U_{1,t_1},\ldots,U_{q,t_q}). 
\end{equation}
 For any partitioned word $(S,\nu)\in \Pc W_q,$ with $S=(w_1,\ldots,w_m),$ we shall consider  for any $t\in\R_+^q$, 
 $$K_t(S,\nu)=C_{\#\nu}( \prod_{i\in A} \Tr(w^N_{i,t}), A\in \nu).$$
 
 \begin{lem}\label{EDP cut and join}For any $f\in [q]$ and any partitioned word $x=(S,\nu)\in\Pc W_q,$ 
 $$\frac{d}{dt_f}K_t(x)= -\frac{\overline{n}_w(f)}{2}K_t(x)-\frac{1}{N}\left( \sum  K_t(\Tc^{+}_{i,j}(x)) -    \sum K_t(\Tc^{-}_{i,j}(x))\right),$$
 where the sums are  over $(i,j)$ belonging respectively to  $\Nc_{w(S)}^+(f)$  and $\Nc_{w(S)}^-(f)$.
\end{lem} 
Note that this Lemma could be  proved  differently  than below, using the formulas of  \cite[Sec. 4.]{Rains}. 
\begin{rmk} It is possible to adopt a more functional approach, closer to \cite{CollinsGuionnetSegala}, describing  the above operations in terms of the square of the  formal operator of derivation with respect to one letter, given in \cite{CollinsGuionnetSegala}, followed by a contraction.  A possible framework  could be the algebra of trace polynomials  used in  \cite{CK}. We leave here this question open.
\end{rmk}
\begin{proof} Let us fix a tuple $(w_1,\ldots,w_m)$ of  words, set $w=w_1\ldots w_m=x_{i_1}^{\epsilon_1}\ldots x_{i_n}^{\epsilon_n}$, with $\epsilon_1,\ldots,\ep_n\in \{-1,1\}$ and $\iota:[n]\to[m]$ the map induced by the decomposition of $w$ into $w_1,\ldots, w_m$.    For any partition $\nu \in \Pc_m,$ we denote by  $\overline{\nu}^{0}\in \Pc_{n}$ the biggest partition such that $\iota(\overline{\nu}^0)=\nu$. For any $\sigma\in \Sy_n$, let   $S_\sigma$ be  the tuple   of words   $x_{i_{c_1}}^{\ep_{c_1}}\ldots x_{i_{c_k}}^{\ep_{c_k}}$, such that $ c_1<\ldots< c_k $   forms a cycle of $\sigma $, ordered by the position of their first letter in $w.$  We denote  $\mathcal{W}_q(w)=\{S_\sigma:\sigma\in \Sy_n\}$  and   fix an arbitrary injective map  

\begin{equation}
S\in \mathcal{W}_q(w)\mapsto \sigma_S\in\Sy_n, \text{  such that  }S_{\sigma_S}=S.\label{word permutation}
\end{equation}
  
\vspace{0,1 cm}

Let us denote $\theta: M_N(\C)\to M_N(\C), M \mapsto M^t$, set for any integer $i\le \ell(w),$ $\delta_i=\frac{1-\ep_i}{2},$ and define  for any pair of distinct integers $i,j$ two  operators $T^+_{i,j}$ and $T^-_{i,j}$, acting on  $\End((\C^N)^{\ts \{i,j\}}),$ by setting for any $M\in \End((\C^N)^{\ts \{i,j\}}),$
\begin{align}\label{Transposition and Contraction}
T^+_{i,j}(M)&=M( i\hspace{0,1cm} j),\\
T^-_{i,j}(M)&=\theta_i^{\delta_i}\circ\theta^{\delta_j}_j\left(\theta^{\delta_i}_i\circ\theta^{\delta_j}_j(M)\langle i \hspace{0,1cm} j\rangle \right).
\end{align}
For any collection  of  words $S\in \mathcal{W}_q(w)$,  any pair $1\le i<j\le n$ and any $U_1,\ldots,U_n\in \Un(N)$ with $U^{\ep_i}_i=U^{\ep_j}_j,$ 
\begin{equation*}
\Tr_{(\C^N)^{\ts n}}( \sigma_S  \overline{T}_{i,j}^{\pm}\left(U_1\ts \ldots \ts U_n \right))=\Tr_{(\C^N)^{\ts n}}( \sigma_{\Tc^{\pm}_{i,j}(S)}  U_1\ts\ldots \ts U_n).
\end{equation*}
It follows that if  $U_1,\ldots, U_n $ are $\Un(N)$-valued random variables  with $U_i^{\ep_i}=U_j^{\ep_j}$, then for any partition $\pi\in\Pc_n $  such that $i$ and $j$ belong to the same block of $\pi,$ 
\begin{equation}\label{Cumulant relatif cut}
\Tr_{(\C^N)^{\ts n}}( \sigma_S  C_{\pi,1_n}(T_{i,j}^{\pm}(\Bts_{1\le i\le n}U_i )))=\Tr_{(\C^N)^{\ts n}}( \sigma_{\Tc^{\pm}_{i,j}(S)}  C_{\pi,1_n}(\Bts_{1\le i\le n}U_i  )).
\end{equation}
 We shall now apply  Lemma \ref{EDPTens}  to the  tensors
\begin{equation}
w_t^{\ts A} = \Bts_{k\in A} U_{i_k,t_{i_k}}^{\epsilon_k}, \label{not tenseur mots}
\end{equation}
for any $A\subset [n],$ that we   denote simply by  $w_t^{\ts},$ for $A=[n].$ For any $f\in [q]$, the latter yields
\begin{equation}
\frac{d}{dt_f}\esp[w_t^{\ts A}]= -\frac{\overline{n}_{w,A}(f)}{2}\esp[w_t^{\ts A}]- \frac{1}{N}\left(\sum  \overline{T^+_{i,j}}^A-\sum \overline{T^-_{i,j}}^A\right)\left(\esp[w_t^{\ts A}]\right),\label{diff tens mots}
\end{equation}
where the  first and the second sums are  over pairs $(i,j)\in A^2$ belonging respectively to $\Nc^+_{w}(f)$ and  $\Nc^-_{w}(f)$.
Consider now  a partitioned word $(S,\nu)$  with $S\in \mathcal{W}_q(w)$.   According to (\ref{Tensor Cumulants and usual Cumulants}),
$$K_t(S,\nu)= \Tr_{(\C^N)^{\ts n} } ( C_{\overline{\nu}^0, 1_n}[ \sigma_{S} w_t^{\ts }])= \Tr_{(\C^N)^{\ts n} } ( \sigma_{S}C_{\overline{\nu}^0, 1_n}[  w_t^{\ts }]).$$
For any subset $A\subset [n]$, with $\#A \ge 2$, let us set $T_A=0,$ if $\#A\ge 3 $ and 
\begin{equation} \label{transposition contraction signed}
T_A=\mp\frac{t_f}{N}T^{\pm}_{a,b},
\end{equation}
if $A=\{a,b\},$ with $(a,b)\in \Nc_{w}^\pm(f).$ 
Let us consider here elements of $M_N(\C)$  as linear operators of $M_N(\C)$ acting by left multiplication. The family of tensors 
\begin{equation}
\alpha_A=e^{\frac{1}{2}\sum_{1\le f\le q} t_f \overline{n}_{w,A}(f)}\esp[w_t^{\ts A}]\in \End(V^{\ts A}), \label{tenseurs partiels mots Browniens}
\end{equation}
for $A\subset [n],$    with  $V=M_N(\C)$, satisfies the condition of Lemma \ref{Lemma exp cumulants} with $\mu=\{\Nc_w^{\pm}(f):1\le f\le q\}.$ Hence, for any $f\in [q]$ and  $\pi\in \Pc_n,$ 
   \begin{align*}\tag{*}\label{EDS Cumulant Tens}
\frac{d}{dt_f}e^{\frac{1}{2}\sum_{1\le f\le q} t_f \overline{n}_w(f)}C_{\pi,1_n}[ w_t^{\ts }]=\sum_{(a,b)\in\Nc^{\pm}_{w}(f)} C_{\pi\vee \overline{\{a,b\}},1_n} \left(t_f^{-1}\overline{T}_{\{a,b\}}\left(w_t^{\ts }\right)\right).
\end{align*}
For any  pair $(a,b)\in\Nc_{w}$, if $p,q\in[m]$ are such that the $a^{\text{th}}$ and the $b^{\text{th}}$ letters of $w$ belong respectively to $w_p$ and $w_q$, then  according to (\ref{Cumulant relatif cut}), the trace $-\ep_a\ep_bN\Tr_{(\C^N)^{\ts n}}(\sigma_S C_{\overline{\nu}^{0}\vee \overline{\{a,b\}},1_n}\left( t_f^{-1} \overline{T}_{\{a,b\}}\left(w_t^{\ts}\right)\right))$ is equal to
\begin{align*}
Tr_{(\C^N)^{\ts n}}(\sigma_{\Tc_{a,b}^{\ep_i\ep_j}(S)} C_{\overline{\nu}^{0}\vee \overline{\{a,b\}},1_n}\left(  w_t^{\ts}\right)) &= K_t(\Tc^{\ep_i\ep_j}_{a,b}(S),\nu \vee \overline{\{p,q\}})\\
&=K_t(\Tc^{\ep_i\ep_j}_{a,b}(S,\nu)).\label{cut tensor cut cycle}\tag{**}
\end{align*}
The  two equations  (\ref{EDS Cumulant Tens}) and (\ref{cut tensor cut cycle})  then imply the announced formula. \hfill\qed\end{proof}

\subsubsection{Scaling and asymptotic expansion of the cumulants} Let us now introduce a scaling of the above functions that matches  the one of section \ref{Scaling cumulants} and that yields a differential system with converging coefficients and initial conditions,  as  $N$ goes to infinity. For any partitioned word  $(S,\nu)\in \Pc W_q,$ with $S=(w_1,\ldots,w_m),$  the following quantity  
\begin{equation}
\vp_{t,N}(S,\nu)= N^{2(\#\nu-1)-\ell(S)}C_{\#\nu}( \prod_{i\in A} \Tr(w_{i,t}), A\in \nu)\label{def fonction partitioned words}
\end{equation}
satisfies these two conditions. Let us define two operators on $\C^{\Pc W_q}$, by setting for any function $\vp\in \C^{\Pc W_q}$ and $(S,\nu)\in \Pc W_q,$ $L_f(\vp)(S,\nu)$ to be equal to
\begin{align}
 -\frac{\overline{n}_{w(S)}(f)}{2}\vp(S,\nu)&+\sum_{(i,j)\in  \Nc^{0,-}_{S,\pi}} \vp(\Tc_{i,j}^{-}(x))-\sum_{(i,j)\in  \Nc^{0,+}_{S,\pi}}  \vp(\Tc_{i,j}^{+}(x) )   \label{operation genre constant}
\end{align}
 and 
  \begin{equation}
  D_f(\vp)(x)= \sum_{(i,j)\in  \Nc^{2,-}_{S,\nu}} \vp(\Tc_{i,j}^{-}(x))-\sum_{(i,j)\in  \Nc^{2,+}_{S,\nu}}  \vp(\Tc_{i,j}^{+}(x)).\label{operation  anse}
  \end{equation}

\begin{prop}\label{PropoSystDiff Cut and Join}For any $t\in \R_+^q,$  $N\in \N^*$ and $(S,\pi)\in \Pc W_q$,
$$\frac{d}{dt_f} \vp_{t,N}(S,\pi) =  (L_f+\frac{1}{N^2}D_f).\vp_{t,N}(S,\pi)$$
and  $\vp_{0,N}(S,\pi)=1,$ if $\pi$ has one block and $0$ otherwise. As $N\to \infty,$ the sequence of functions $\vp_{t,N}$ converges pointwise  towards the unique function $\vp_t,$ such that for any $t\in \R_+^q$ and $ (S,\pi)\in \Pc W_q,$ 
$$\frac{d}{dt_f} \vp_{t}(S,\pi) =  L_f.\vp_{t}(S,\pi)$$
and $\vp_{0}=\vp_{0,1}.$
\end{prop}
\begin{proof}It is a direct consequence of Lemma \ref{EDP cut and join}.
\hfill\qed\end{proof}

\begin{coro} There exists a sequence of functions $(\psi_{t,g})_{g\ge 1}$ on $\Pc W_q$ such for any $(S,\pi)$,  the power series with coefficients  $(\psi_{t,g}(S,\pi))_{g\ge 1}$ has a positive radius of convergence and for $N$ large enough,
$$\vp_{t,N}(S,\pi)=\vp_{t}(S,\pi)+\sum_{g\ge 1}N^{-2g}\psi_{t,g}(S,\pi).$$\end{coro}


\begin{proof} For any fixed $n\in\N^*$, the operators $L_f,D_f$ preserve the finite dimensional space of functions 
supported on $\{x \in\Pc W_q  : \ell(x)\le n\}.$ The above expansion follows then easily from Proposition \ref
{PropoSystDiff Cut and Join}.
\hfill\qed\end{proof}

In particular, for any $n_1,\ldots,n_m\in  \Z,$  $N^{m-2}C_m(\Tr(U_t^{n_1}),\ldots,\Tr(U_t^{n_m}))$ admits a limit as $N\to\infty.$  Together with the property of independence, stationarity  and invariance by unitary adjunction of multiplicative increments of the process $(U_t)_{t\ge 0}$ and   the notion of higher order freeness developed in \cite{HOPS}, this fact alone implies that for any $(S,\pi)\in \Pc W_q,$  $\vp_{t,N}(S,\pi)$ admits a limit as $N\to\infty.$  Nonetheless, this result does not give  easily an  expansion  in $N$. 

\section{Two estimates on the cumulants}

The proof of Theorem \ref{Theom Energy Words} relies on  two estimates on the above cumulants. The first one  shall  allow to extend them to broader class of words as we will see in section \ref{Yang-Mills}.    It is nonetheless too loose to  obtain a positive radius of convergence as stated in Theorem \ref{Theom Energy Words}. The second type of estimates gives a much sharper bound and is the key ingredient to address this question.

\subsection{All-order bounds}  For any  $t\in \R_+^q$,  let us define a scalar product $\la\cdot,\cdot\ra_t$ on $\R^q$, by setting for any $a,b\in \R^q,$
$$\la a, b\ra_t= \sum_{f=1}^q a_fb_f t_t.$$
Let us recall the notations of section \ref{Section Partial words} and set for any word $w\in W_q$, 
$$A_t(w)=\la \overline{n}_w,\overline{n}_w\ra_t.$$
For any $(S,\pi)\in \Pc W_q$ and $N\in \N^*$,  we define inductively a sequence by setting 
$\psi_{t,0,N}(S,\pi)= \vp_{t,N}(S,\pi)$ and for any $g\in\N,$
$$\psi_{t,g+1,N}(S,\pi)=N^2\left(\psi_{t,g,N}(S,\pi)-\psi_{t,g}(S,\pi)\right).$$
\begin{lem}  \label{Exponential Bound Amperean area}For any  words $w_1,\ldots,w_m\in W_q$,  $N\in\N^*,$ $t\in \R_+^q$, $
\lambda\in [0,1]$  and  $k\in \N,$
$$|\frac{d^k}{d\lambda^k} \vp_{\lambda t,N}(w_1,\ldots,w_m)| \le  2^{-k}A_t(w_1\ldots w_m)^ke^{\frac{A_t(w_1\ldots w_m)}{2}} $$
and 
$$ |\psi_{ t,k,N}(w_1,\ldots,w_m)|\le   2^{-k}A_t(w_1\ldots w_m)^ke^{\frac{A_t(w_1\ldots w_m)}{2}}.$$
Besides, $A_t(w_1\ldots w_m)\le m\sum_{i=1}^mA_t(w_i).$
\end{lem}
\begin{proof}For every integer $p\in\N$ and  any matrix $M\in M_p(\C)$, let us set  $\|M\|=\max_{i\in [p]}\sum_{j=1}^p|M_{i,j}|.$ Recall 
that $\|\cdot\|$ is a sub-multiplicative norm on $M_p(\C)$, such that for any matrix $M\in M_p(\C)$ and $v\in \C^p$, $
\max_{i\in [p]}| (Mv)_i| \le \|M\|\max_{i\in [p]}|v_i| .$  Let us fix  a word $w\in W_q$ and denote by $B_
{w}$ the set of partitioned words $(S,\pi)\in \Pc   W_q$ with $n^{\pm}_{w(S)}=n^{\pm}_w$.   Note that $B_w$ is stable by the 
operations appearing in (\ref{operation genre constant}) and (\ref{operation  anse}), so that for any  $f\in [q]$,  the two operators $L_f$ and $D_f$ preserve  
the finite dimensional space $\Fc_{w}$ of functions on $\Pc   W_q,$ with support in $B_w.$ For any $F\in \End(\Fc_w),$ let us 
set $$\|F\|=\max_{y\in B_w}\sum_{x\in B_w} F(\delta_{x})(y).$$
For any $(S',\pi')\in B_w$, there are at most $\frac{\bar{n}_{w}(f)(\bar{n}_{w}(f)-1)}{2}$ elements $(S,\pi)$ of $B_w$ such that $(S',\pi')
$ is obtained from $(S,\pi)$ by a transformation of the form $\Tc_{i,j}^{\pm}$ with $(i,j)\in \Nc_{w(S)}$. It follows that the restriction of 
$L_f$ and $D_f$ to $\Fc_w$ satisfy the following inequality
  $$\max\{\|{L_f}_{|\Fc_w}\|, \|{D_f}_{|\Fc_w}\|,\|{L_f}_{|\Fc_w}+\frac{1}{N^2}{D_f}_{|\Fc_w}\|\}\le \frac{\overline{n}
  _w(f)^2}{2}.$$
Let us set  
\begin{equation}
L=\sum_{f=1}^qt_f{L_f}_{|\Fc_w} \,\, \text{  and   }\,\,D=\sum_{f=1}^qt_f{D_f}_{|\Fc_w}.  \label{cutjoin global}
\end{equation}
Then, according to Proposition \ref
{PropoSystDiff Cut and Join}, 
  $${\vp_{t,N}}_{|{B_w}}=e^{L+\frac{1}{N^2}D}({\vp_0}_{|{B_w}})$$
and 
$${\vp_t}_{|{B_w}}=e^{L}({\vp_0}_{|{B_w}}).$$
Let us recall that for any matrices $A,B\in M_p(\C),$ \begin{equation}
e^{A+B}-e^A=\int_0^1 e^{s (A+B)}B e^{(1-s) A} ds. \label{Duhamel}
\end{equation}
Using iteratively this formula  together with the former two equations yields that for any $g\ge 1,$
\begin{equation}
{\psi_{t,g,N}}_{|B_w}= \int_{0<s_1<s_2<\ldots<s_g<1}e^{s_1(L+\frac{1}{N^2}D)}De^{(s_2-s_1)L}D\ldots D e^{(1-s_g)L}ds 
({\vp_0}_{|B_w}).
\end{equation}
Therefore, 
\begin{align*}
\max_{(S,\pi)\in B_w}\{g!{|\psi_{t,g,N}}(S,\pi)| \}&\le  \|D\|^g e^{\max\{\|L+\frac{1}{N^2}D\|,\|L\|\}}\max_{(S,\pi)\in 
B_w}\{\vp_0(S,\pi) \}\\
&\le 2^{-g}(\sum_{f=1}^q t_f \overline{n}^2_w)^g e^{\frac{1}{2}\sum_{f=1}^q t_f \overline{n}^2_w}=2^{-g}A_t(w)^ge^{\frac{A_t
(w)}{2}}.
\end{align*}
Besides, for all $\lambda\in [0,1]$  and  $r\in \N,$
$$\frac{d^r}{d\lambda^r} {\vp_{\lambda t,N}}_{|B_w}= (L+\frac{1}{N^2}D)^re^{\lambda(L+\frac{1}{N^2}D)}({\vp_0}_{|{B_w}}),
$$
so that  the left-hand-side is uniformly bounded by $2^{-r}A_t(w)^re^{\frac{A_t(w)}{2}}$ on $B_w$.   Besides, for any  words $w_1\ldots w_m
\in W_q,$  if $w=w_1\ldots w_m,$ then $\overline{n}_w=\overline{n}_{w_1}+\ldots +\overline{n}_{w_m} $ and $A_t(w)= \|\overline{n}_w\|_t^2$, the last point follows.\hfill\qed\end{proof}
Note that these first estimates are very loose. For example,  for $w=x_1^n,$ the lemma shows that $|\vp_t(w)|\le e^{t_1n^2},$ when we
 have\footnote{The right decay of this moment sequence is  at least polynomial, as the measure $\mu_{t_1}$ is absolutely continuous with respect to 
 the Lebesgue measure, with a density that is H\"older continuous. The regularity of the density does depend on $t_1$, there are three regimes:  
 $t_1<4,$ $t_1=4$ and $t_1>4$, see remark 6.8. of \cite{ThierrySW}.  } the simple  bound $|\vp_t(w)|= |\int_\U \omega^n \mu_
 {t_1}(d\omega)|\le 1 $. Furthermore, for this same word, it yields for any $m\in\N^*$, $|N^{m-2}C_m(\Tr(U_{t_1}^n))|\le e^{t_1 
 n^2m^2},$ so that the exponential power series of the sequence on the right-hand-side diverges.

 \subsection{Sharper bounds for the first and second orders\label{Sharper Bounds}} For any  positive integer $m,$ let us denote by $
 \Cc_m$  the set of Cayley trees on $m$ vertices. For any $w_1,\ldots,w_m\in W_q$, $\mathfrak{T}\in \Cc_m$,   we set 
 \begin{equation}
 \mathfrak{T}_t(w_1,\ldots,w_m)=  
 \prod_{\{i,j\} \text{ edge of }\mathfrak{T}}  \la \overline{n}_{w_i}, \overline{n}_{w_j}\ra_t  \label{complex arbre}
 \end{equation}
 and $$\tilde{\mathfrak{T}}_t(w_1,\ldots,w_m)= \prod_{\{i,j\} \text{ edge of }\mathfrak{T}}  \la  n_{w_i}, n_{w_j}\ra_t,  $$
if $m\ge 2,$ and $\mathfrak{T}_t(w_1)=\tilde{\mathfrak{T}}_t(w_1)=1$,  otherwise.

\begin{prop} For any   words $w_1,\ldots,w_m\in W_q$ and any $N\in\N^*,$
$$|\vp_{t,N}(w_1,\ldots,w_m)|\le  \sum_{\mathfrak{T}\in \Cc_m} \mathfrak{T}_t(w_1,\ldots, w_m) $$\label{Cumulant's control}
and for $m\ge 2,$
$$ \vp_{ t,N}(w_1,\ldots,w_m)=\sum_{\mathfrak{T}\in \Cc_m}\tilde{\mathfrak{T}}_t(w_1,\ldots,w_m)+R_{t,N}(w_1,\ldots,w_m),$$
with $|R_{t,N}(w_1,\ldots,w_m)|\le (\frac{m}{2})^m(A_t(w_1)+ \cdots + A_t(w_m))^{m}e^{\frac{m}{2} (A_t(w_1)+\cdots +A_t
(w_m))},$ whereas 
$$\vp_{t,N}(w_1)= 1- \frac{1}{2}\|n_w\|_t^2+R_{t,N}(w_1),$$
with $|R_{t,N}(w_1)|\le A_t(w_1)^2e^{\frac{A_t(w_1)}{2}},$
\end{prop}

The second estimate shows that the first one is optimal for tuples of words $w\in W_q,$ such that  $\overline{n}_w= n_w.$   The idea of the following proof is to get formulas for the considered cumulants not in terms of partitioned words, as we did in Lemma  \ref{Exponential Bound Amperean area}, but rather in terms of their "Schur-Weyl  dual", that is here, tensors of unitary matrices.   
 
\begin{proof} Let us consider $q$ independent Brownian motions $(U_{1,t})_{t\ge 0},\ldots,(U_{q,t})_{t\ge 0}$, on $\Un(N),$ fix a tuple  of 
$m$ words $S_0=(w_1,\ldots,w_m)\in W^m_q$,  set  $w=w_1\ldots w_m=x^{\ep_1}_{i_1}\ldots x^{\ep_n}_{i_n}$,   $\sigma_0=(1\ldots \ell(w_1))\times \cdots \times (1\ldots \ell(w_m)) $  and $ \pi_0\in\Pc_n$ the set of its orbits ordered by their smallest element and recall definition (\ref{not tenseur mots}).  Let us further use the same notations as in the first paragraph of proof of Lemma \ref{EDP cut and join}.   According our choice of scaling and (\ref{Tensor Cumulants and usual Cumulants}), 
$$\vp_{t,N}(w_1,\ldots,w_m)=N^{m-2} \Tr_{(\C^N)^{\ts n}}\left( \sigma_0 C_{\pi_0,1_m}(w_t^{\ts })\right).$$
Let us remind that the family of tensors (\ref{tenseurs partiels mots Browniens}) satisfies the condition of Lemma  \ref{Lemma exp cumulants}, with $(T_A)_{A\in \mathfrak{B}_n}$  given in (\ref{transposition contraction signed}). We shall  denote here  $T_{\{p,q \}}$ by $T_{p,q}$.   Let $\Psi_{m}$ be the 
set of strictly increasing sequences $(\nu_i)_{i=1}^m\in\Pc_m^m$, with $\nu_1=0_m$ and $\nu_m=1_m.$  Each sequence of  pairs $(p_l,q_l)_{1\le l\le m-1}\in \Nc_{w}^k,$  with $\vee_{l=1}^{m-1}\overline{\{p_l,q_l\}}\vee \pi_0=1_n,$ induces an element of $\Psi_m$ and any element of $\Psi_m$ can be obtained in this way. For any $\nu\in \Psi_m$, we set   $\Upsilon(\nu)=\{(p_k,q_k)_{1\le k\le m-1}:\forall k\in[m-1],p_k<q_k, \vee_{i=1}^k \{p_i,q_i\}\vee\pi_0=\nu_k\}.$  
According to (\ref{DuhamelTens}),   $e^{\frac{1}{2}\sum_{1\le f\le q} t_f \overline{n}_w(f)} C_{\pi_0,1_m}(w_t^{\ts })$ equals
\begin{equation}
\beta_{\pi_0,1_m}(\alpha)= \sum_{\nu\in\Psi_{m}}\sum_{(p_k,q_k)_k\in\Upsilon(\nu)} \int_{\Delta^{m-1}}   \alpha_{\pi}^{s_0} \overline{T}_{p_1,q_1}\alpha_{\nu_1}^{s_1}\ldots \overline{T}_{p_{m-1},q_{m-1}}
\alpha^{s_m}_{1_m}ds.\label{Duhamel  plus Tens Brownien}
\end{equation}
Let us now rewrite the right-hand-side in terms of unitary matrices thanks to the definitions (\ref{partitioned tensor}) and (\ref{tenseurs partiels mots Browniens}).
For any partition $\mu\in\Pc_m$, let us denote by $\underline{\mu}:[n]\to \mu$ the natural quotient map and introduce  a collection  $((U_{t,b,f})_{t\ge 0})_{b\in \mu, 1\le f\le q}$  of $q\#\mu $ independent 
$\Un(N)$-Brownian motions. Then, let  $(U^{\mu})_{\mu\in\Pc_n}$ be a collection of independent random variables such that for each $\mu\in \Pc_n$, 
 $$\left(U^\mu_{t_k,k}\right)_{t\in \R_+^n,  1\le k\le n}\overset{(law)}{=}\left(U_{t_k,\underline{\mu}(k),i_k}^{\ep_k}\right)_{t\in \R_+^n,  1\le k\le n}.$$ 
 For any $v\in 
\End(M_N(\C)^{\ts n}),$ let us denote by $\mathcal{L}(v)$ the endomorphism of $\End(M_N(\C)^{\ts n}),$ of left multiplication by $v$.  With these notations, (\ref{Duhamel  plus Tens Brownien})  yields that  the cumulant $ C_{\pi_0,1_m}(w_t^{\ts })$  equals
\begin{equation}
 \sum_{\nu\in\Psi_{m}}\sum_{(p_k,q_k)_k\in\Upsilon(\nu)} \int_{\Delta^{m-1}} \esp\left[
\mathcal{L}({\Bts_{k=1}^n
 {U^{\nu_{1}}_{s_{1} t_{i_k},k}}}) \overline{T}_{p_1,q_1}\ldots \overline{T}_{p_m,q_m}
\mathcal{L}(\Bts_{k=1}^n{U^{\nu_{m}}_{s_{m} t_{i_k},k}})\right]ds.\label{Duhamel Rep unitaire}
\end{equation}
Notice that for every matrices $U_1,\ldots,U_n\in \Un(N)$ and  any pair $(a,b)\in \Nc_{w}^{\pm},$ such that $U_a=U_b$,  $\mathcal{L}(\Bts_{k\in 
[n]} {U_{k}}) $ commutes with $\overline{T}_{a,b}.$  Therefore, if we set for any sequence $\nu\in \Psi_m,$  
$s\in \Delta^{m-1}$ and $1\le k\le n,$ 
 $$V_{s,k}^\nu=U_{s_{1}t_{i_k},k}^{\nu_1}\ldots U_{s_mt_{i_k},k}^{\nu_m}\in\Un(N),$$
   then,    $\Tr_{(\C^N)^{\ts n}}\left( \sigma_0 C_{\pi,1_m}(w_t^{\ts })\right)$ equals  
\begin{align*}
\sum_{\nu\in\Psi_{m}}\sum_{(p_k,q_k)_k\in
 \Upsilon(\nu)}\int_{\Delta^{m-1}} \esp[\Tr(\sigma_0 \overline{T}_{p_{m-1},q_
 {m-1}}\ldots \overline{T}_{p_1,q_1} ( \ts_{k\in[n]}V^\nu_{s,k} ))]ds.
\end{align*} 
For each  $\nu\in \Psi_{m},$ $\gamma=(p_k,q_k)_k\in \Upsilon(\nu),$  there exists a word $w_\gamma$ such that   $\Tc^{\ep_{m-1}}_{\tilde{p}_{m-1},\tilde{q}_{m-1}}\circ\ldots\circ \Tc^{\ep_2}_{\tilde{p}_2,\tilde{q}_2}\circ \Tc^{\ep_{1}}_{p_1,q_1}(S_0,0_m)=(w_\gamma,1_1),$ where  for any $k\in [m-1],$  $\ep_k={\ep_{p_{k}}\ep_{q_{k}}}$  and  for $2\le k <m,$ 
 $\tilde{p}_k,\tilde{q}_k$ are the new positions of the $p_k^{\text{th}}$ and $q_k^{\text{th}}$ letters  of $S_0$ in  $\Tc^{\ep_{k-1}}_{\tilde{p}_{k-1},\tilde{q}_{m-1}}\circ\ldots\circ \Tc^{\ep_{1}}_{p_1,q_1}(S_0).$  If  $t_\gamma=\prod_{k=1}^{m-1} (-\ep_{i_
{p_k}}\ep_{i_{q_k}})t_{i_{p_k}},  $ then, according to (\ref{Cumulant relatif cut}), for any tensor  $v\in M_N(\C)^{\ts n},$
   $$\Tr(\sigma_0 \overline{T}_{p_{m-1},q_
 {m-1}}\ldots \overline{T}_{p_1,q_1} (v))=t_\gamma N^{1-m}  \Tr(\sigma_{w_\gamma}   v),$$
where $\sigma_{w_\gamma}$ is defined as (\ref{word permutation}).
It remains now to unfold the above notations to get 
\begin{align}
\label{Formula Cumulants one Trace}\vp_{t,N}(w_1,\ldots, w_m)=\sum_{\nu\in\Psi_{m}}\sum_{\gamma\in\Upsilon(\nu)}t_\gamma \int_{\Delta^{m-1}} \esp[
N^{-1}\Tr( w_\gamma\left(V^\nu_{s,1},\ldots,V^\nu_{s,n}\right) ))]ds.
\end{align}
To conclude, note that the normalized trace in the integrand of the right-hand-side are bounded by $1$. Let $\Gamma:\Psi_m\to \Cc_m$  be the  
$(m-1)!$-to-one map  that sends every sequence $\nu= (\vee_{l=1}^{k}\overline{\{i_l,j_l\}})_{1\le k<m}$ to the Cayley tree with edges 
$(\{i_k,j_k\})_{1\le k< m}$.   Then, for every $\nu\in \Psi_m,$ $ \sum_{\gamma\in\Upsilon(\nu)}|t_\gamma|=\Gamma(\nu)
(w_1,\ldots,w_m)$ and the first bound of the statement follows.  Another consequence of (\ref{Formula Cumulants one Trace}), is that, in the Taylor expansion of $\vp_{\cdot,N}(w_1,\ldots,w_m)$ around $0\in \R_+^q,$ the terms  of degree less than $m-2$  vanish, whereas, the sum of terms of degree $m-1$ is exactly $\sum_{\nu\in \Psi_m} \sum_{\gamma\in\Upsilon(\nu)}\frac{t_\gamma}{(m-1)!}. $ But, for any $\mathfrak{T}\in\Cc_m$ and $\nu \in \Gamma^{-1}(\mathfrak{T}),$ $\sum_{\gamma\in\Upsilon(\nu)}t_\gamma=\tilde{\mathfrak{T}}_t(w_1,\ldots,w_m).$ Hence, applying Proposition \ref{Exponential Bound Amperean area} yields the second estimate. A similar argument applies for $\vp_{t,N}(w_1).$
\hfill\qed\end{proof}
\begin{rmk} When $N=1,$ $(w_{k,t})_{1\le k\le m}$  has the same law as $(exp(i Z_k))_{1\le k\le m},$ where $(Z_k)_{1\le k\le m}$ is a Gaussian vector of covariance matrix $(\langle{n_{w_p},n_{w_q}}\rangle_t)_{1\le p,q\le m}$. It follows from the definition that $$\vp_{t,1}(w)=e^{-\frac{1}{2} \|n_{w}\|_t^2},$$
whereas for $w\in W_q^m,$ the Leonov Shiryaev formula (\ref{LeoS}),   shows that  for any $m\ge 2,$
$$\vp_{t,1}(w_1,\ldots,w_m)= \sum_{\mathfrak{T}\in \Cc_m} \tilde{\mathfrak{T}}(w_1,\ldots,w_m)e^{-\frac{1}{2}\|n_{w_1}+\ldots+n_{w_m}   \|^2_t}. $$\end{rmk}
 \begin{rmk} Formula  (\ref{Formula Cumulants one Trace})  implies that any cumulant  $\vp_{t,N}(w_1,\ldots,w_m)$ is expressed   in terms  of  the non-commutative distribution of the unitary Brownian motion.
 \end{rmk} 
For any $t\in\R_+,$ let us set $$\lambda_t=  (t+1+\sqrt{t(t+2)})e^{\frac{1}{2}\sqrt{t(t+2)}}$$
 and for any word  $w\in W_q$ and $t\in \R_+^q,$ 
\begin{equation}
\lambda_t(w)= \prod_{f=1}^q \lambda_{t_i}^{\overline{n}_w(f)}.\label{lambda}
\end{equation}
\begin{lem}\label{Bound second order} For any   sequence $S=(w_1,\ldots,w_m)$ of  words in $W_q$  with $m\ge 2$  and any $N\in\N^*,$
$$|\psi_{t,1,N}(S)|\le 2^{m-1}m^2\prod_{i=1}^m\lambda_{t}(w_i) \max_{1\le i\le m}{\| \overline{n}_{w_i}\|_t ^2} \sum_{\mathfrak{T}\in \Cc_m} \mathfrak{T}(S).  $$  
\end{lem}
\begin{proof} Let  us fix $S_0\in W_q^m,$ $B_0$  the set of partitioned word obtained from $x_0=(S,0_m)$ by  a sequence of cut and join transformations and set for any $\eta\in \Pc_m,$   $B_\eta=\{(S',\eta):(S',\eta)\in B_0\}.$    For  any $N\in\N^*,$ according to Proposition \ref{PropoSystDiff Cut and Join} and Duhamel's formula  (\ref{Duhamel}), if $L$ and $D$ are defined as in (\ref{cutjoin global}),
$$\psi_{t,1,N}(S_0)=-\int_0^1 e^{(1-s) L}D e^{s(L+N^{-2} D)} (\vp_0)(x_0)ds.  $$
Setting for any $ s\in [0,1]$ and $x,y\in B_0,$ $ Q_{s}(x,y)=   e^{ s L}D (\delta_{y})(x),$ yields
\begin{align*}
-\psi_{t,1,N}(S_0)&=\sum_{y\in B_0}  \int_0^1  Q_{1-s}(x_0,y)  e^{s(L+N^{-2} D)} (\vp_0)(y)ds\\
&=\sum_{\eta\in\Pc_m}\sum_{y\in B_\eta} \int_0^1 \vp_{st,N}(y) Q_{1-s}(x_0,y)  ds.\tag{$\diamondsuit$}\label{Relative cumulants Bound}
\end{align*}
Let us fix  $\eta\in\Pc_m$ and consider the space of strictly increasing sequences $\Psi^\eta=\{(\vee_{i=1}^{k}\overline{\{p_i,q_i\}})_{0\le k\le m-\#\eta}\in\Pc_m^{m-\#\eta+1}: \vee_{i=1}^{ m-\#\eta}\overline{\{p_i,q_i\}}=\eta\}$ and $\Psi_\eta=\{(\vee_{i=1}^{k}\overline{\{p_i,q_i\}}\vee \eta)_{0\le k\le \#\eta-1}\in\Pc_m^{\#\eta}: \vee_{i=1}^{\#\eta}\overline{\{p_i,q_i\}}\vee \eta=1_m\}.$ For any $t\in \R_+^q$ and any increasing sequence $\nu\in\Pc_m^l$,   induced by a sequence of pairs of integers $(p_i,q_i)_{1\le i \le l}$ of $[m],$ let us set 
$$\nu_t(S_0)=\prod_{k=1}^l \la \overline{n}_{w_{p_k}}, \overline{n}_{w_{q_k}}\ra_t.$$
For any $\pi\in \Pc_m$ and any linear operator $A$ on $\C^{B_0},$ let us introduce another operator $A_{\pi}$ by setting for any $\vp\in \C^{B_0}$ and $x\in B_0,$  $A_\pi(\vp)(x)=\sum_{y \in B_{\pi}}\vp(y)A(\delta_y)(x).$  
On one hand, for any $y\in B_{\eta}$,   a slight modification of the proof of Proposition \ref{Cumulant's control} yields that for any $t\in\R_+^q,$
\begin{equation*}
|\vp_{t,N}(y)| \le  \frac{1}{(\#\eta-1)!}  \sum_{\nu\in \Psi_{\eta}}\nu_t(S_0).
\end{equation*}
On the other hand, the same argument  as in Lemma \ref{Lemma exp cumulants} yields that  $ \sum_{y\in B_\eta}  |Q_{1}(x_0,y) |$ is bounded by   
$$ \sum_{\substack{\nu\in \Psi^\eta \\ \overline{\{p,q\}}\lec \eta}} \la \overline{n}_{w_{p}},\overline{n}_{w_{q}}\ra_t \nu_t(S_0) \int_{\Delta^{\#\eta-1}} \sup_{x\in B_{\eta}} \sum_{y\in B_{\eta}}|  e^{s_{\#\eta}L_{\nu_{\#\eta}}+\ldots +s_{1}L_{\nu_{1}}}(\delta_{y})(x)|ds. $$
To conclude, we shall  expand the exponential in the right-hand-side and use then triangular inequality. For each $f\in [q],$ let us define an operator  $\tilde{L}_f$ on $\C^{\Pc W_q},$ by setting for all $\vp\in \C^{\Pc W_q}$ and $x=(S,\pi)\in B_0,$
\begin{equation}
\tilde{L}_f(\vp)(x)=\sum_{(a,b)\in \Nc^{+,0}_{S,\pi}(f)}\vp(\Tc^+_{a,b}(x))+\sum_{(a,b)\in \Nc^{-,0}_{S,\pi}(f)}\vp(\Tc^-_{a,b}(x)).\label{tilda L}
\end{equation}
For any $x\in B_0,$  $\nu\in \Psi^\eta$ and $s\in \Delta^{\#\eta-1},$ 
\begin{align*}
\sum_{y\in B_{\eta}}|  e^{s_{\#\eta}L_{\nu_{\#\eta}}+\ldots +s_{1}L_{\nu_{1}}}(\delta_{y})(x)|&\le e^{-\frac{1}{2}\sum_{f=1}^q\overline{n}_w(f)t_f}\sum_{y\in B_{\eta}}  e^{\sum_{f=1}^qt_f\tilde{L}_f }(\delta_{y})(x).\end{align*}
For any tuple $S$ of words in $W_q$, let us denote, for each $f\in[q],$ by $w_f(S)\in W_q,$ the word obtained from $w(S)$ by deleting the letters $x_{f'}$ and $x_{f'}^{-1},$ for $f'\not=f.$ Recall that for any $x=(S,\pi)\in \Pc   W_q$,  $\vp_0(x)=1$, if $\#\pi=1,$  and $0$, otherwise.  Then,  for any $x=(S,\pi)\in B_\nu,$ 
$$\sum_{y\in B_{\eta}}  e^{\sum_{f=1}^qt_f\tilde{L}_f }(\delta_{y})(x)\le \prod_{f=1}^q   e^{t_f\tilde{L}_f }(\vp_0)(w_f(S),1_{\overline{n}_w(f)}).  $$
Let us denote \begin{equation}
\Lambda_t(w)= e^{-\frac{1}{2}\sum_{f=1}^qt_f\overline{n}_w(f)}\sup_{(S,\pi)\in B_0} \prod_{f=1}^q  e^{t_f\tilde{L}_f }(\vp_0)(w_f(S),1_{\overline{n}_w(f)}).  \label{Big lambda}
\end{equation}
Gathering the equality (\ref{Relative cumulants Bound})  with the last four inequalities yields that  the quantity  $ e^{ \sum_{f=1}^qt_f\overline{n}_w(f)}\psi_{t,1,N}(S_0,0_m)$ is bounded by 
\begin{align*}
\sum_{\substack{\eta \in \Pc_m\\\nu^1\in \Psi^\eta, \nu^2\in \Psi_\eta}}&\frac{\nu^1_t(S_0)\nu^2_t(S_0)}{(\#\eta -1)! (m-\#\eta)!}\Lambda_t(w)\sup_{1\le i\le m}\|\overline{n}_{w_i}\|^2_t  \\
&\le \Lambda_t(w)\sup_{1\le i\le m}\|\overline{n}_{w_i}\|^2_t \sum_{\nu\in \Psi_{m-1}}\sum_{k=1}^{m} \frac{\nu_t(S_0)}{(k-1) ! (m-k)!} \\
&\le  2^{m-1}  \Lambda_t(w)\sup_{1\le i\le m}\|\overline{n}_{w_i}\|^2_t \sum_{\mathfrak{T}\in \Cc_m}\mathfrak{T}(S_0).
\end{align*}
Then, the following lemma  implies the announced bound.
\hfill\qed\end{proof}

\begin{lem}  For any word $w\in W_q$,  $t\in \R_+^q,$ $$\Lambda_t(w)\le \lambda_t(w),$$
where  both terms are given  by (\ref{Big lambda}) and (\ref{lambda}).
\end{lem}
\begin{proof}  For any $w=x^{\ep_1}\ldots x^{\ep_n}\in W_1$, let us define $\mathcal{I}_w=\{k\in [\ell(w)]: \ep_k=-\ep_{k+1} \},$ where  $\ep$ is  indexed over $\Z\slash [l(w)]\Z$, and an operator $\hat{L}$ on $\C^{\Pc W_1},$ setting for any $\vp\in \C^{\Pc W_1}$ and  $x\in  \Pc W_1,$
$$\hat{L}(\vp)(x)=\sum_{(a,b)\in  \Nc^{0,\pm}_{w(S),1_n}(1)\setminus \mathcal{I}_w}\vp(\Tc_{a,b}^{\pm}(S)).$$
Let us consider the function $\vp_0$ on $\Pc W_1$ with $\vp_0(S,\pi)= 1$  if $\#\pi=1,$ and $0$ otherwise. As $2\#\mathcal{I}_w\le n$, the operator defined in (\ref{tilda L}) satisfies for any $t\in \R_+,$
$$  e^{t \tilde{L}_1}(\vp_0) (w,1_n)\le  e^{t(\frac{n}{2}+ \hat{L})}(\vp_0)(w,1_n).$$
What is more,  $$e^{t \hat{L}}(\vp_0)(w,1_n)\le e^{t \tilde{L}_1}(\vp_0)(x_1^n,1_n). $$
For any $n\in \N^*$, let us denote here by $\vp_0\in \C^{\Cc W_1},$ the constant function equal to $1$.  Therefore, setting for any $s\ge 0, n\in \N$,  $\rho_s(n)= e^{  s  \tilde{L}_1}(\vp_0)(x_1^n,1_n),$   for any $t\in \R_+^q,$ 
\begin{equation*}
\Lambda_t(w)\le \prod_{f=1}^q\rho_{t_f}(\overline{n}_{w}(f)). \tag{*}\label{Upperbound Free Mult HBM}
\end{equation*}
According to the definition (\ref{tilda L}) of the operator $\tilde{L}_1,$ the family of functions $(\rho_\cdot(n))_{n\ge 0  }$ satisfies the following differential system: for any $n\in \N^*$ and $s\ge 0,$
$$\frac{d}{ds}\rho_s(n)=\frac{n}{2}\sum_{p=1}^{n-1} \rho_s(p)\rho_s(n-p)$$
and $\rho_0(n)=1.$  Let us define a formal power series by setting for any $z$,  $$ \rho_s(z)=\sum_{n\ge 1}\rho_s(n) z^n.$$ 
Then, for any $s\ge 0,$ $$\frac{d}{ds}\rho_s(z)=z\rho_s(z) \frac{d}{dz} \rho_s(z). $$
According to the Lemma 13 of \cite{BianeSegalBarg},  $(\rho_s(n))_{n\ge 0}$ is the sequence of moments of a Hermitian operator (therein, denoted by  $\Lambda_{s/2}\Lambda_{s/2}^*$) acting on a separable Hilbert space and,  according to Proposition 11 of the same article, with spectrum $[\lambda^-_{s/2},\lambda_{s/2}^+]$, where
$$[\lambda^-_s,\lambda_s^+]=[(2s+1-2\sqrt{s(1+s)})e^{-\sqrt{s(s+1)}},(2s+1+2\sqrt{s(s+1)})e^{\sqrt{s(s+1)}}].$$
It implies that for all $n\in\N^*,$ $\rho_s(n)\le \left(\lambda^+_{s/2}\right)^n $ and the result then follows from (\ref{Upperbound Free Mult HBM}).
\hfill\qed\end{proof}


\section{Applications}

\subsection{Asymptotic behavior of the free energies\label{Section Free energies BM}}

We give here a proof of the Theorems \ref{Theom Energy Words} and \ref{MF Potential}. For any function $V\in \C^{W_q}$, let us set  $\|V\|_1=\sum_{w\in W_q}|V(w)|$ and $\|V\|_{\infty}=\sup_{w\in W_q} |V(w)|.$ We  define $\Fc_{1,q}=\{V\in \C^{W_q}: \|V\|_1<\infty\}$ and $\Fc_{0,q}$ the set of functions $V\in \C^{W_q},$ with $ \#\{w\in W_q: V(w)\not=0\}<\infty.$  For any  $N\in\N^*,$  $U_1,\ldots, U_q\in \Un(N)$ and $V\in \Fc_{1,q}$  the following sum converges almost surely and defines a random matrix 
$$V(U_1,\ldots, U_q)= \sum_{w\in W_q} V(w) w(U_{f},1\le f\le q)\in M_N(\C).$$ 
Let $(U_{1,t})_{t\ge 0},\ldots, (U_{q,t})_{t\ge 0}$ be $q$ independent $\Un(N)$ Brownian motions.
\begin{thm}\label{Theom Energy WordsNV} For $t\in\R_+^q $ and $V\in \Fc_{0,q}$, there exists 
$r_V>0$ and analytic functions $\vp_{t,V}, (\psi_{t,V,N})_{N\ge 1}$ and $\psi_{t,V}$ on $D_{r_V}=\{z\in\C: |z|<r_V\},$ such that 
$$ e^{\psi_{t,V,N}(z)}=\esp[e^{zN\Tr(V(U_{1,t_1},\ldots, U_{q,t_q}))- N^2\vp_{t,V}(z)}]\longrightarrow e^{\psi_{t,V}(z)},$$
as $N\to\infty,$ where the convergence is uniform on compact subset of $D_{r_V}$.\end{thm}

\begin{proof}For any function $V\in \Fc_{1,q}$,  let us define  $V_{t,N}= V(U_{1,t_1},\ldots, U_{q,t_q})$ and  $I_{t,V,N}(z)=  N^{-2} \log \esp[e^{ z  N\Tr(V_{t,N})}].$ The latter analytic function admits the following Taylor expansion on a neighborhood of $0$, 
$$ \sum_{w\in W_q} V(w)\vp_{t,N}(w) z + \sum_{\substack{m\ge 2\\ w_1,\ldots, w_m\in W_q }}\frac{\prod_{i=1}^mV(w_i)}{m!}\vp_{t,N}(w_1,\ldots,w_m ) z^{m}. $$
According to Proposition \ref{PropoSystDiff Cut and Join}, the summands of the two sums converge pointwise as $N\to\infty$.  The summand of the first sum is bounded by $|z V(w)|,$  so that this sum converges absolutely towards $\sum_{w\in W_q} V(w)\vp_t(w) z.$ Each coefficient of the second power series is bounded by
 $$\sum_{\substack{w_1,\ldots, w_m\in W_q\\\mathfrak{T}\in \Cc_m}}\frac{\prod_{i=1}^m |V(w_i)|}{m!}\mathfrak{T}(w_1,\ldots,w_m)\le \max_{a\not=b\in \supp(V)}\la \overline{n}_a,\overline{n}_b\ra_t^{m-1} \frac{m^{m-2}}{m!} \|V\|_1^m.$$
It follows that $$I_{t,V,N}(z)=\sum_{m\ge 1, w_1,\ldots, w_m\in W_q}\frac{z^m\prod_{i=1}^m V(w_i) }{m!}\vp_{t,N}(w_1,\ldots,w_m)$$ is well defined  on $\Dd_{r_V}$ with $$\frac{1}{r_V}=\max_{a,b\in \supp(V)}\{\la \overline{n}_a,\overline{n}_b\ra_t\} e\|V\|_1$$ and converges uniformly as $N\to \infty,$ on its compact subset towards a limit that we denote by $\vp_{t,V}(z)$.    Let us set for any $V\in \C^{W_q},$ 
$$\eta_V=\sup_{a\in \supp(V)}  \|\overline{n}_a\|_t^{2 } \sum_{w\in W_q} \lambda_t(w_i)|V(w_i)|,$$
where $\lambda_t$ is defined in (\ref{lambda}). For any $m\ge 1$ and $V$ with $\eta_V<\infty,$ according to Lemma \ref{Bound second order},  the sum 
 $$\psi_{N,m}(V)=\sum_{w_1,\ldots, w_m\in W_q}\frac{\prod_{i=1}^m V(w_i) }{m!}\psi_{t,N}(w_1,\ldots,w_m)$$ 
 is well defined and satisfies
 $$|\psi_{N,m}(V)|\le    2^{m-1}\frac{m^{m}}{m!}\eta_V^{m}.$$
Thanks to Proposition \ref{PropoSystDiff Cut and Join},   by dominated convergence,  
if $\psi_{t,V,N}(z)$ denotes the power series with coefficients   $(\psi_{N,m}(V))_{m\ge 1},$ then $\psi_{t,V,N}$ is well defined on 
 $D_{r'_V}$ and converges uniformly on its compact subsets towards a function $\psi_{t,V},$ with $$\frac{1}{r'_V}= 2 e \eta_V.$$
To conclude, note  that   $r'_V<r_V,$ so that if $|z|<r'_V<\infty,$ $\vp_{t,V}(z)$ is well defined  and the analytic function 
$$e^{\psi_{t,V,N}(z)}= \esp[e^{zN\Tr(V_{t,N})- N^2\vp_{t,V}(z)}]$$  
converges uniformly towards $e^{\psi_{t,V}(z)}$ on  $D_{r'_V}.$
\hfill\qed\end{proof}

For any function $V\in \C^{W_q}$ and any word $w=x_{i_1}^{\ep_1}\ldots x_{i_n}^{\ep_n}\in W_q$, with $\ep_1,\ldots, \ep_n\in \{-1,1\},$ let us set  $V^*(w)= \overline{V(x_{i_n}^{-\ep_n}\ldots x_{i_1}^{-\ep_1})}.$ We say that $V\in \C^{W_q}$ is \emph{symmetric} if $V^*=V$.  For any $N\in \N^*,$ let $(U_{1,t})_{t\ge 0},\ldots, (U_{q,t})_{t\ge 0}$ be $q$ independent $\Un(N)$-Brownian motions. For any symmetric function $V\in\Fc_{1,q}$ and $t\in \R_+^*,$ the random matrix 
$$V_{t,N}=V(U_{1,t_1},\ldots, U_{q,t_q}) $$  
is Hermitian and its operator norm is bounded by $\|V\|_1.$ In particular, it satisfies $0<\esp[e^{N\Tr(V_{t,N})}]<\infty$. Let   $\mu_{t,V}$ be the probability measure on $\Un(N)^q$, absolutely continuous with respect to the law of $(U_{1,t_1},\ldots,U_{q,t_q})$, whose Radon-Nikodym derivative  is $\esp[e^{N\Tr(V_{t,N})}]^{-1}e^{N\Tr(V(U_1,\ldots,U_q))}.$  We shall  denote by $(U^V_{1,t_1},\ldots,U^V_{q,t_q})$  a random variable distributed as $\mu_{t,V}$ on  $\Un(N)^q.$ 

For any $V,W\in  \C^{W_q},$ let us define  for any $t\in \R_+^q,$
$$\Jc_{t,V}(W)= \sum_{m\ge 1, a,w_1,\ldots, w_m\in W_q,\mathfrak{T}\in \Cc_m} |W(a)|\frac{\prod_{i=1}^m|V(w_i)|}{m!}\mathfrak{T}_t(a,w_1,\ldots, w_m)$$
and 
\begin{align*}
\Jc&_{t,2,V}(W)=\sum_{a,b\in W_q} \la \overline{n}_a,\overline{n}_b\ra_t |W(a)||W(b)|\\& +\sum_{m\ge 1, a,b,w_1,\ldots, w_m\in W_q,\mathfrak{T}\in \Cc_m} |W(a)||W(b)|\frac{\prod_{i=1}^m|V(w_i)|}{m!}\mathfrak{T}_t(a,b,w_1,\ldots, w_m).
\end{align*}
For a fixed $V\in \C^{W_q}, t\in \R_+^q,$  we consider $\Fc_{t,V}=\{ W\in \Fc_{1,q}: \Jc_{t,V},\Jc_{t,2,V}<\infty    \}.$

\begin{thm}\label{MF PotentialNV}   For any   $V\in \Fc_{1,q}$, the random variables $\frac{1}{N}\Tr(W(U^V_{1,t_1},\ldots, U^V_{q,t_q}))$, with  $W\in \Fc_{t,V}$ converge jointly  in probability towards  constants $\Phi_{t,V}(W),$ $W\in \Fc_{t,V}.$
\end{thm}

\begin{rmk} For any Cayley tree $\mathfrak{T}\in \Cc_m,$ let us write $(d_\mathfrak{T}(i))_{1\le i\le m }$ for the degree distribution of $\mathfrak{T}$. For $m\ge 2,$ Cauchy-Schwarz inequality yields that for $w_1,\ldots, w_m\in W_q,$
\begin{align*}
\frac{\prod_{i=1}^m|V(w_i)|}{m!}\mathfrak{T}_t(w_1,\ldots, w_m) \le\sum_{\substack{\mathfrak{T}\in \Cc_m}} \frac{1}{m!}  \prod_{i=1}^m|V(w_i)|\|\overline{n}_{w_i}\|_t^{d_\mathfrak{T}(i)}.
\end{align*}
For any $V\in \C^{W_q},$  according to formula (1) of \cite{BM},  
$$\sum_{w\in W^m_q}\prod_{i=1}^m|V(w_i)|\mathfrak{T}_t(w_1,\ldots, w_m) \le \sum_{w\in W_q^m} \prod_{i=1}^m\|\overline{n}_{w_i}\|_t|V(w_i)|  (\sum_{i=1}^m\|\overline{n}_{w_i}\|_t)^{m-2}.$$
In particular,  $\Jc_{t,V}(W),\Jc_{t,2,V}(W)<\infty,$  as soon as  $V,W\in \Fc_{0,q},$ with $$\|V\|_{1} .\sup_{\substack{a\in W_q\\ W(a) \text{ or }V(a)\not=0}} \|\overline{n}_a\|^2_t< \frac{1}{e}.$$   \end{rmk}

\begin{proof} Let us consider  $V\in \Fc_{1,q}$  and   $W\in \Fc_{t,V}.$  According to (\ref{Mean Pot}),  (\ref{VarPot}) and Proposition \ref{Cumulant's control}, the mean and variance of the random variable $\frac{1}{N}\Tr(W(U_{1,t_1}^V,\ldots,U_{q,t_q}^V)$   are respectively equal to the following  absolutely convergent sums, 
$$   \sum_{\substack{m\ge 0,w,w_1, \\ \ldots, w_m\in W_q}}\frac{1}{m!}W(w) \prod_{k=1}^mV(w_k) \vp_{t,N}(w,w_1,\ldots, w_m)$$
and
\begin{align*}
\sum_{\substack{m\ge 0,a,b,w_1, \\ \ldots, w_m\in W_q}}\frac{1}{m!}W(a)W(b) \prod_{k=1}^mV(w_k) \vp_{t,N}(a,b,w_1,\ldots, w_m).
\end{align*}
According to Proposition \ref{Cumulant's control}, dominated convergence implies  that these two sequences have a limit as $N\to\infty.$
\hfill\qed\end{proof}

\subsection{Central limit theorem}  As a consequence of Proposition \ref{PropoSystDiff Cut and Join}, we get the
 
\begin{prop} For any $t\in\R_+^q,$ using the notation (\ref{mots Browniens ind}),  the family $(\Tr(w_t^N)- N\vp_{t}(w))_{w\in W_q},$ converges towards the centered Gaussian field $(\phi_w)_{w\in W_q}, $ such that  for any $a,b\in W_q,$ $\cov(\phi_w,\phi_{w'})= \vp_{t}(w,w').$
\end{prop}
\begin{proof} For any word $w\in W_q,$ $\esp[N^{-1}\Tr(w_t^N)]= \vp_{t}(w)+O(N^{-2}),$  whereas for any $m\ge 2$ and $w\in W_q^m,$ $C_m(\Tr(w_{1,t}^N),\ldots,\Tr(w_{m,t}^N))= N^{2-m}\vp_{t,N}(w_1,\ldots,w_m) $ converges,  as $N\to\infty,$ towards $\vp_{t}(w_1,w_2),$  if $m=2$, and $0,$ if $m\ge 3.$ 
\hfill\qed\end{proof}
Let us remark that  in the proof of Proposition \ref{Cumulant's control},  we obtained in formula (\ref{Formula Cumulants one Trace}) an expression of the function $\vp_{t,N}$ for any $N\in\N^*,$   in terms of its restriction to single words. Specialized to partitioned words with two blocks, this gives an expression of  the covariance of the above field.  Let  us define a family of $3q$  independent  $\Un(N)$-Brownian motions $(U_{1,t})_{t\ge 0}, (V^{1}_{1,t})_{t\ge 0},(V^{2}_{1,t})_{t\ge 0},\ldots,$ $ (U_{q,t})_{t\ge 0}, (V^{1}_{q,t})_{t\ge 0},(V^{2}_{q,t})_{t\ge 0} $. For any words $w_1,w_2\in W_q$ and $(a,b)\in\Nc_{ w_1w_2}(f)\setminus \Sc_{(w_1,w_2),0_2}, $  let $\chi:[\ell(w_1w_2)]\to\{1,2\}$ be the function such that the $i^{\text{th}}$ letter of the word $\Tc_{a,b}((w_1,w_2))$ belongs to $w_{\chi(i)}.$ If $\Tc_{a,b}((w_1,w_2))= x_{i_1}^{\ep_1}\ldots x_{i_n}^{\ep_n}$, let us set for any $r,s\in \R_+^q,$ 
 $$w_{a,b}(r,s)=\left(U_{i_1,r_{i_1}}V^{\chi(1)}_{i_1,s_{i_1}}\right)^{\ep_1}\ldots \left(U_{i_n,r_{i_n}}V^{\chi(n)}_{i_n,s_{i_n}}\right)^{\ep_n}.$$
According to formula (\ref{Formula Cumulants one Trace}),
 $$ \cov(\Tr(w_{1,t,N}),\Tr(w_{2,t,N}))=\sum_{\substack{1\le f\le q \\\ep\in \{-1,1\}}}\ep t_f \int_{0}^1 \sum\frac{1}{N}\esp[\Tr(w_{a,b}(st,(1-s)t))]ds,$$
where the second sum is over pairs $(a,b)\in\Nc^\ep_{w_1w_2}(f)\setminus \Sc_{(w,w'),0_2}$.



\section{Planar Yang-Mills measure\label{Yang-Mills}}
We shall see in this section how the results of the previous one apply in the framework of the planar Yang-Mills measure.  We  first recall a construction  and some  properties of the planar Yang-Mills mesure following \cite{Champsmarkoholo,MF} and from  subsection \ref{Asympto Wilson loop} on, explain our results.  In the next section, we will then give analogues of Schwinger-Dyson's equations.

\subsection{Paths of finite length}   Let us call  \emph{parametrized path} any   Lipschitz function from $[0,1]$ to $\R^2$, that are either constant or with speed bounded by below.   We denote  by  $\Pd(\R^2)$  the set of parametrized paths up to  bi-Lipschitz increasing reparametrization and call its elements \emph{paths}.   For any path $\gamma\in \Pd(\R^2)$ with parametrization $c:[0,1]\to\R^2$, let us denote its endpoints $c(0)$ and $c(1)$ by $\underline{\gamma}$ and $\overline{\gamma}$, and by $\gamma^{-1}$ the \emph{inverse} path parametrized by $t\in[0,1]\mapsto c(1-t)$. 
For any $x\in\R^2$, we  denote  by $\Ld_x(\R^2)$ the set of paths $\gamma\in\Pd(\R^2)$ such that $\underline{\gamma}=x=\overline{\gamma}$ and call elements of $\Ld_x(\R^2)$ loops based at $x$.  We set $\Ld(\R^2)= \cup_{x\in \R^2}L_x(\R^2)$. For any loop $l$ based at some point $x\in\R^2$ and parametrized by the Lipschitz-continuous map $\tilde{l}: [0,1]\to \R^2$, we call \emph{non-based loop} the induced map $\U\to\R^2$, up to  bi-Lipschitz, order preserving, one-to-one mappings  of $\U$. If $a$ and $b$ are two paths such that $\overline{a}=\underline{b}$, we denote by $ab$ the path of $\Pd(\R^2)$ obtained by concatenation.

\subsection{Embedded graphs\label{section embedded graphs}} We call here \emph{ embedded graph in the plane}  the data of  a triple  of finite sets $\mathbb{G}=(\Vbb,\Ebb,\Fbb)$, where faces $\Fbb$ are domains of the plane with disjoint interior, simply connected  in  the Riemann sphere $\hat{\C}$ and  which boundary is the image of a non-based loops, edges  $\Ebb$ are   paths of $\Pd(\R^2)$ stable by the inversion map, such that the union of their image  is the union of boundaries of elements of $\Fbb$,  vertices $\Vbb$ are  the  endpoints of $\Ebb$ and  the graph induced by $\Ebb$ on $\Vbb$ is connected. With this convention, any edge $e\in \Ebb$ is either a simple loop or an injective path  of finite length. As any edge  is bounded, $\Gbb$ has a unique unbounded  face that we denote by $F_{\infty,\Gbb}$ (or simply $F_\infty$) and set $\mathbb{F}_b= \mathbb{F}\setminus{F_\infty}$.  We  write  $|F|$ the area of any element of   $F\in\Fbb_b$ and by  $\partial F$ the non-based loop, whose  image is the boundary of $F$ with counterclockwise orientation.     We shall write $\Pd(\Gbb)$ for  the set of  paths  that are concatenation of elements of $\Ebb$ and $\Ld(\Gbb)$ (and respectively for any $v\in \Vbb$, $\Ld_v(\Gbb)$) for the set of loops (respectively loops based at $v$)  in $\Pd(\Gbb).$  The dual graph of $\Gbb$ is the combinatorial graph $\hat{\mathbb{G}}=(\Fbb, \hat{\mathbb{E}})$ with vertices $\Fbb$, where two faces  are neighbors if their closures intersect. For any path $\gamma\in \hat{\mathbb{G}}$, we denote   its number of edges by $|\gamma|$.

\subsection{A free group: reduced loops of an embedded graph}

Let us fix an embedded graph $\Gbb$. For any  pair of  paths $\gamma_1$ and $\gamma_2$ of $\Pd(\Gbb)$,  let us write
 $\gamma_1\sim \gamma_2$ and say  that  $\gamma_1$ and $\gamma_2$  are equivalent,  if one can get $\gamma_1$ from 
 $\gamma_2,$ or vice-versa, by adding or erasing  paths of the form $e.e^{-1}$, with $e\in\Ebb.$ For any path $\gamma$, there is a unique 
 element of minimal length   in its equivalence class,  that we call the \emph{reduction} of $\gamma$. The set of reduced paths endowed 
 with the operation of concatenation and reduction forms a groupoid that we denote by $\RP(\Gbb).$  For any $v\in \Vbb$, we denote by 
 $\RL_v(\Gbb)$ the set of reduced paths that are loops based at $v$. Endowed with the above multiplication, $\RL_v(\Gbb)$ is a free group 
 of rank $\#\Fbb_b$   (we shall highlight specific free basis in section \ref{Basis RL}).

 \subsubsection{Multiplicative functions}

For any $N\in\N^*$ and  any subset $\mathfrak{P}$ of $\Pd(\R^2),$ stable by concatenation, we call a   function $h:\mathfrak{P}\to \Un(N),$  \emph{multiplicative} if for any paths  $a,b\in \mathfrak{P}$, with $\overline{a}=\underline{b},$  $$h(a b  )=h(a)h(b). $$
We denote the space of multiplicative functions by $\Mca_N(\mathfrak{P})$ and by $\Cc_\mathfrak{P}$ (or simply $\Cc$, when $\mathfrak{P}=\Pc(\R^2)$) the smallest $\sigma$-fields such that for any $\gamma\in \mathfrak{P}$, $h\in \Mca_N(\mathfrak{P})\longmapsto h(\gamma) \in U(N)$ is measurable, where $U(N)$ is endowed with its Borel $\sigma$-fields.   For any  embedded graph $\mathbb{G}$, $v\in \mathbb{V}$ and any choice of  basis $\Lambda$ of $\RL_v(\Gbb)$, there is a bijection  
\begin{align*}
\Theta_\Lambda:\Mca_N(\Ld_v(\Gbb))\to\Un(N)^{\#\Fbb_b}\\
h\mapsto (h(\lambda))_{\lambda\in\Lambda}.
\end{align*}

\subsubsection{Lassos basis and discrete Yang-Mills measure\label{Section Discrete YM}}
For any loop $v\in \Vbb$ and $l\in \Ld_v(\Gbb)$, we say that $l$ is a lasso based at $v$, if  $l= a\partial_{\overline{a}}F a^{-1}$ where $a\in\Pd(\R^2)$, with $\underline{a}=v$, $\overline{a}$ is a vertex in the image of $\partial F$  and     $\partial_{\overline{a}}F $ is the rooting of $\partial F$ at $\overline{a}.$ We shall see in the next section that there exists basis of $\RL_v(\Gbb)$ formed with lassos.  Let  $\Lambda=(\lambda_F)_{F\in \Fbb_b}$ be a  free basis of $\RL_v(\Gbb)$, composed with lassos, and  let  $YM^\Lambda_\Gbb$  be the law of $\Theta_\Lambda^{-1}( (U_{F,|F|})_{F\in\Fbb_b})$   on $(\Mca_N(\Ld_v(\Gbb)),\Cc_\Gbb),$ where $((U_{F,t})_{t\ge 0})_{F\in \Fbb_b}$ is a family of independent Brownian motions on $\Un(N).$
\begin{lem}[\cite{MF}] \label{Discrete YM}i) For any lassos free basis $\Lambda,\Lambda'$ of $\RL_v(\Gbb),$ $YM^\Lambda_\Gbb=YM^{\Lambda'}_{\Gbb}.$ We denote this law by $YM_{\mathbb{G}}.$

ii) If $\Gbb'$ is  embedded graph with $\Pd(\Gbb')\subset \Pd(\Gbb)$ and   $\Rc^{\Gbb}_{\Gbb'}:\Mca_N(\Ld_v(\Gbb))\to \Mca_N(\Ld_v(\Gbb'))$ denotes the restriction map, then $${\Rc^{\Gbb}_{\Gbb'}}_*YM_\Gbb=YM_{\Gbb'}.$$ 
 \end{lem} 
We denote by $(H_l)_{l\in \Ld_v(\Gbb)}$ the canonical process on $\Mca_N(\Ld_v(\Gbb))$ with law $\YM_\Gbb.$ The first point follows from the invariance of the law of the $\Un(N)$-Brownian motion by adjunction and the following result.
\begin{thm}[\cite{HUMPHRIES}]\label{Humphries} Let  $X=(x_1,\ldots, x_n)$ and $Y=(y_1,\ldots ,y_n)$ be two free basis for the free group $F_n$, such that $x_i$ is conjugated to $y_i$ for all $i\in [n]$.  Then  $X$ can be obtained from  $Y$ by a sequence of transformations of the kind   $(u_1,\ldots , u_n)\mapsto (u'_1,\ldots, u'_n)$ where, for some $i,j$,   $u'_i=u_ju_i u_j^{-1}$ or $u_j^{-1}u_i u_j$
and $u'_k=u_k$ for $k\not=i$. 
\end{thm}

The second point of Lemma \ref{Discrete YM} requires a proof (see \cite[Prop. 4.3.4]{Champsmarkoholo}) that we won't reproduce here; an argument goes as follows. Let  $\Gbb'=(\Vbb',\Ebb',\Fbb_b' )$ be an embedded graph with $v\in \Vbb'$ and $\Pd(\Gbb')\subset\Pd(\Gbb)$. Assume that there exists $F\in \Fbb_b'$ and $F_1,F_2\in \Fbb_b$ with $\overline{F}=\overline{F_1\cup F_2}.$    If  $\lambda$ and  $\lambda'_{1},\lambda'_{2}$  are lassos respectively in $\Ld_v(\Gbb)$ and $\Ld_v(\Gbb')$, with faces $F,F_1,F_2$ such that $\lambda= \lambda_1\lambda_2$, then under $YM_\Gbb$, $H_{\lambda}$  has the same law as $U_{1,|F_1|}U_{2,|F_2|},$ where $(U_{1,t})_{t\ge 0}$ and $(U_{2,t})_{t\ge 0}$ are two independent $\Un(N)$-Brownian motions. Hence, it has the same law as $U_{1,|F|}$, that is the law of $H_\lambda$ under $YM_{\Gbb'}.$
\subsection{Yang-Mills measure \label{section continuous YM}}
 Let $d_1$ and $d_l$  be the two distances on $\Pd(\R^2)$ defined in the following way: for any pair of paths $\gamma_1,\gamma_2\in \Pd(\R^2)$, parametrized by $c_1,c_2:[0,1]\to \R_2,$ with $|c'_1|=|c'_2|=1,$

$$d_1(\gamma_1,\gamma_2)=|\underline{\gamma_1}-\underline{\gamma_2}|+ \int_0^1 |c'_1(t)-c'_2(t)|dt$$
and
$$d_\ell(\gamma_1,\gamma_2)=\inf_{\phi,\psi }\sup_{r,s\in [0,1]}\{ |c_1\circ\phi(r)-c_2\circ\psi(s)|\}+|\ell(c_1)-\ell(c_2)|,$$
where we have denoted by $\ell(c)$ the \emph{length} of a path $\gamma\in\Pd(\R^2)$ and the infimum is taken over all increasing bijections of $[0,1]$. It has been proved in  \cite{Champsmarkoholo} that $d_1$ and $d_\ell$ induce the same topology on $\Pd(\R^2)$, though $(\Pd(\R^2),d_1)$ is complete and $(\Pd(\R^2),d_\ell)$ is not. In the following, we shall only use this topology and say that a sequence of paths $(l_n)_{n\ge 0}$ converges to $ l$ if   $d_\ell(\gamma_n,\gamma)\to 0$ and $\underline{\gamma_n}=\underline{\gamma}, \overline{\gamma_n}=\overline{\gamma},$ for every $n\in\N^*$.  For any embedded graph $\Gbb,$ with $v\in\mathbb{V}$,  let us denote by $\Rc^v_\Gbb:\Mca_N(\Pd(\R^2))\to \Mca_N(\Ld_v(\Gbb))$  the restriction mapping. This application is measurable with respect to the $\sigma$-fields $\Cc$ and $\Cc_{\Gbb}$. It is shown in \cite{Champsmarkoholo} that the family of measures $YM_\Gbb$, with $\Gbb$ ranging over embedded graphs, can be extended in the following way.
\begin{thm} \label{continuiteYM} There exists a probability measure $\YM_N$ on $(\Mca_N(\Pd(\R^2)), \Cc )$ such that, for any embedded graph $\Gbb$, $v\in\R^2,$ $${\Rc^v_{\Gbb}}_{*}\YM_N  =YM_\Gbb.$$
Let $(H_\gamma)_{\gamma\in \Pd(\R^2)}$ be a random multiplicative function   with law  $\YM_N$. If $(\gamma_n)_{n\ge 0}$ is a sequence of paths in $\Pd(\R^2)$ that converges to $\gamma$, then, under  $\YM$, $H_{\gamma_n}$ converges in probability towards $H_\gamma$. If $h$ is an area preserving diffeomorphism of $\R^2,$ then the process $(H_{h(\gamma)})_{\gamma\in \Pd(\R^2)}$ and $(H_\gamma)_{\gamma\in \Pd(\R^2)}$ have the same law.
\end{thm}
For any $l\in \Ld(\R^2)$ and $N\in \N^*$, the random variable $\frac{1}{N}\Tr(H_l)$ is called a \emph{Wilson loop}.  The above approach to define the Yang-Mills measure on embedded graphs by defining a random morphism of the free group  has   been considered  by  Franck Gabriel  to give a different construction of the Yang-Mills measure and to solve problems of characterization of Markovian holonomy fields  in \cite{FranckMarko}. Therein,  one of the key feature is the study of  group valued sequences satisfying  properties of invariances   by an action  of the  braid group  (analogues  to the operations introduced in Theorem \ref{Humphries}).  In \cite{FranckMarko,CDG}, it is also the starting point of an alternative construction of the Yang-Mills measure and of master fields in the plane.

\subsection{\texorpdfstring{$U(1)$}{Lg}-Yang-Mills measure}

Let us consider the commutative case, $N=1$. Let  $\Gbb$ be an  embedded graph in the plane, $v$ a vertex of $\Gbb$. For any loop $l\in \RL_v(\Gbb)$,  its winding number function defines a compactly supported function  $n_l\in \Ld^2(\R^2)$. Let us fix a family of lassos $(\lambda_F)_{F\in \Fbb_b}$ of $\Gbb$. Under $\YM_1$ measure,  $(H_{\lambda_F})_{F\in\Fbb_b} $ has the same law as  $\#\Fbb_b$ independent marginals of $\Un(1)$-Brownian motion $(U_{F, |F|})_{F\in \Fbb_b}$. Let $W$ be a white noise on the  plane, with intensity given by the  Lebesgue measure. The random family   $(H_l)_{l\in \RL_v(\Gbb)}$ is equal to $(\prod_{F\in \Fbb_b}H_{\lambda_F}^{n_F(l)})_{l\in \RL_v(\Gbb)}$ and has the same law as $(\exp\left(iW(n_l)\right))_{l\in \RL_v(\Gbb)}$.  For any loop $l\in \Ld(\R^2)$, according to Banchoff-Pohl inequality (see Lemma \ref{Banchoff-Pohl}), its winding number function defines an element  $n_l\in \Ld(\R^2).$ Moreover, according to Theorem 3.3.1. of \cite{Champsmarkoholo}, the map $l\in \Ld_0(\R^2)\to \Ld^2(\R^2)$ is continuous, so that, if $(l^1_n)_{n\ge 0},\ldots, (l^m_n)_{n\ge 0}$   are  sequences of $\Ld_0(\R^2)$  that converge for the $d_1$ topology to a  family of loops  $(l_k)_{1\le l\le m}\in \Ld_0(\R^2)^m$,  the sequences of  random variables $ \exp\left(i W(n_{l^1_n})\right),\ldots, \exp\left(i W(n_{l^m_n})\right)$   converge jointly  to $\left(\exp\left(i W(n_{l^k})\right)\right)_{1\le k\le m}$ in distribution. Hence, the process $(H_l)_{l\in \Ld(\R^2)}$ introduced in Theorem \ref{continuiteYM} has the same law as $(\exp\left(i W(n_{l})\right))_{l\in\Ld(\R^2)}.$ Moreover,  the same argument and Lemma \ref{det} yield the following lemma.
\begin{lem} For any integer $N\in\N^*$, under $\YM_N$,  the law of $(\det(H_l))_{l\in \Ld_0(\R^2)}$  and $(\exp\left(i W(n_{l})\right))_{l\in \Ld_0(\R^2)}$ is $YM_1$. 
\end{lem}
\subsection{Two free basis of the  group of reduced loops\label{Basis RL}}
 We shall present two families of free basis of $\RL_{v}(\Gbb)$.  Let $\Ebb^+$  be an orientation of $\Gbb$, that is a subset of $\Ebb$ such that for any $e\in\Ebb$, $e$ or $e^{-1}\in\Ebb^+$. Let us also fix  a spanning tree $T$ of the graph $\Gbb$ and set $T^+$ the collection of positively oriented edges of $T.$ We denote by $e:\Fbb_b\to \Ebb^+\setminus T^+$ the unique bijection such that for any face $F\in \Fbb_b$, $e(F)$ is bounding the face  $F$.  For any $e\in \Ebb$, bounding a face $F$,  we denote by $∂_eF$ the loop starting with $e$ and bounding $F.$ For any  $x,y\in \Vbb$, we denote by $[x,y]_T$ the unique path in $T$ going from $x$ to $y$.   Let us now define two families of loops by setting for any  edge $e\in\Ebb,$

$$\beta_e=  [v,\underline{e}]_T e [\overline{e},v]_T$$
and for any face $F\in \Fbb_b$,
 $$ \lambda_F= [v,\underline{e(F)}]_T∂_{e(F)}F [\underline{e(F)},v]_T.$$ 
 
 \noindent It is easy to see that $\RL_{v}(\Gbb)$ is a free group of rank $\#\Fbb_b$    with free basis $(\beta_e)_{e\in \Ebb^+\setminus T^+}$.  For any loop $l\in \Ld(\Gbb)$, 
 \begin{equation}\label{decompositionenbeta}
l \sim\beta_{e_1}\beta_{e_2}\cdots \beta_{e_n},
\end{equation}
 where $e_1, \ldots, e_n$ are the edges in $\Ebb\setminus T$,  used by the loop $l$ in this order. In \cite{MF}, it is proved that the second family of loops  is another free basis of $\RL_{v}(\Gbb)$.
  \begin{lem}[\cite{MF}] The family  $(\lambda)_{F\in\Fbb_b}$ is a free basis of $\RL_{v_0}(\Gbb).$  \label{Lemme  lassos base}\end{lem} 
For any edge $e\in \Ebb$,  we denote by  $F_L(e)$ and $F_R(e)$ the edges on the left and on the right  of $e$ and denote by $\hat{e}$ the edge  $(F_L(e),F_R(e))\in \hat{\Ebb}$ in the dual graph.   Let $\hat{T}=\Ebb\setminus T$ be the dual spanning tree of $T$, considered as rooted at the infinite face $F_\infty.$   We fix  an  orientation $\Ebb^+$ of $\Gbb$, such that for any edge  $e\in\Ebb^+\setminus T$, the distance in  $\hat{T}$ to the root $F_\infty$  decreases along $\hat{e}$. Note that with this orientation, for any bounded face $F$,   $F_L(e(F))=F$. For any  face $F$, we denote by $\hat{T}_F$ the subtree of $\hat{T}$, with root $F$ and  vertices  the set of descendants of $F$ in $\hat{T}$.  We denote by $C_F$ the set of children of $F$.  For any edge $e\in \Ebb^+\setminus T^+,$ $\hat{T}_{F_L(e)}$ is endowed with      the order   $\lec_e$   induced by the  time of the first visit by the \emph{clockwise contour process} boarding the dual tree $\hat{T},$ starting along the left of $\hat{e}^{-1}$, as is displayed with an example in figure \ref{DecompositionLassos}.  Then, for any edge $e\in \Ebb^+\setminus T$,
$$\lambda_{F_L(e)}=\beta_{e}\left( \overset{\longrightarrow}{\prod}_{F\in C_{F_L(e)}} \beta_{e(F)}\right)^{-1}$$
and

\begin{equation}
  \beta_e= \overset{\longrightarrow}{\prod}_{ F\in  \hat{T}_{F_L(e)} } \lambda_F, \label{lassobeta}
\end{equation}
where $\overset{\longrightarrow}{\prod}$ denotes the product of terms increasing for $\lec_e$, from the left to the right.  For any loop $l\in \Ld_v(\Gbb)$, we  denote by $w_l^T$ the word with letters  $(\lambda_F)_{F\in \Fbb_b}$  and their inverse, such that $l\sim w$, given by the decomposition  (\ref{decompositionenbeta})   and the inversion formula (\ref{lassobeta}). Using notation (\ref{index word}), for any face $F\in\Fbb_b$ and any complex number $z\in F$,  the winding number satisfies 
\begin{equation}
n_l(z)=n_{w_l}(F).\label{winding word loop}
\end{equation}

\begin{figure}\centering \includegraphics[height=3in]{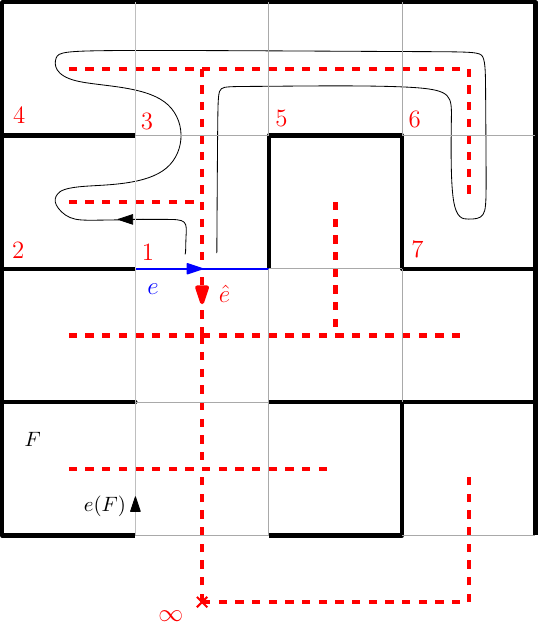}
\caption{\label{DecompositionLassos}We represent with black lines a spanning tree of a square grid together with its dual with dashed red lines. We also display an   edge $e$ of $\Ebb^+\setminus T$ in blue together with the order $\lec_e$ on $\hat{T}_{F_L(e)}$,  by numbering its elements and drawing in black the clockwise contour process around $\hat{T}_{F_L(e)}$. Here, elements of  $C_{F_L(e)}$ are labeled $2$ and $3$.}\end{figure}

\subsubsection{Complexity of lassos decompositions} We can now give an estimate on the complexity of  the above decomposition of a loop in $\Gbb$  in a word of lassos associated to a spanning tree $T$.  We display here results of \cite{MF} in a slightly different form, adapted to our purpose. Let us fix an embedded graph $\Gbb=(\Vbb,\Ebb,\Fbb)$, $v\in \Vbb$ and $t=(|F|)_{F\in \Fbb_b}.$ For any subset $E\subset \Ebb$ and any loop $l\in \Ld(\Gbb)$, denote by $\Lc_E(l)$ the number of times that $l$ uses the edges of  $E$ or $E^{-1}$. The two following lemmas are elementary.

\begin{lem} \label{indicemaxcombi}Let $l\in \Ld(\Gbb)$  be a loop of $\Gbb$. Then, for any face $F\in \Fbb_b$, $$\bar{n}_{w^T_l}(F)= \Lc_{[F,F_\infty]_{\hat{T}}}(l).$$
\end{lem}

\begin{lem} \label{TreeGeo}There exists a spanning tree $T$ of $\Gbb,$ such that for any face $F\in\Fbb_b$, $$d_{\hat{T}}(F,F_\infty)=d_{\hat{\mathbb{G}}}(F,F_\infty).$$
 \end{lem}
 For any loop $l\in\Ld_v(\Gbb)$ and $T$ a spanning tree of $\Gbb$, we want to control the maximal Amperean area $A_t(w^T_l)=\sum_{F\in\Fbb} |F|\bar{n}_{w^T_l}(F)^2,$ with the length of the loop $\ell(l)$.   The \emph{Amperean area } of $l$ is the integral   $$A(l)=\int_{\R^2}n_l(x)^2dx.$$
 \begin{lem}[Banchoff-Pohl inequality, \cite{BP}] \label{Banchoff-Pohl}For any loop of finite length $l\in \Ld(\R^2)$, $$A(l)\le \pi \ell(l)^2.$$
 \end{lem}
 Note that if   $\bar{n}_{w_l^T}=\pm n_l\in  \Z^\Fbb$, that is, if $l$ winds only  to the left or only to the right, then the Banchoff-Pohl inequality  gives  the  expected bound. To treat more general loops, we need the following lemma. 
  \begin{lem}[{{\cite[3. Lemma 5.9]{MF}}}]  \label{lacetmax}There exists a loop $\bar{l}\in \Ld(\Gbb)$, which  does not use any edge twice, such that for any face $F\in \Fbb$ and $z\in F,$  $$ n_{\bar{l}}(z)=d(F,F_\infty).$$    
 \end{lem}
 
 \begin{lem}\label{borneAmpcombi} Let $l\in\Ld(\Gbb)$ be a loop that uses each edge at most once. If $T$ is  a spanning tree chosen as in Lemma  \ref{TreeGeo} and $\mathbb{E}^+$ is any orientation of $\Gbb,$ then 
 $$A_t(w^T_l)\le  \pi (\sum_{e\in \Ebb^+}\ell(e))^2. $$
 \end{lem}
 \begin{proof} The   assumptions together with  Lemma \ref{indicemaxcombi} yield that for any face $F\in\Fbb_b$, 
 $$\bar{n}_{w^T_l}(F)\le d(F,F_\infty). $$
 Let us now choose a loop $\bar{l}$ as in Lemma \ref{lacetmax}. Then, $A_t(w^T_l)\le  A(\bar{l}) $ and Banchoff-Pohl inequality  applied to $\overline{l}$ yields  the expected bound. \hfill\qed\end{proof}
 
\subsection{Asymptotics of Wilson loops as \texorpdfstring{ $N\to\infty$}{Lg}\label{Asympto Wilson loop}}

We shall consider the following families of loops. 
\begin{dfn} A  \label{def skein} \emph{skein} is a finite  multiset  of loops  of $\Ld(\R^2)$. It is    \emph{regular}   if  its associated set is composed of  distinct smooth loops, forming transverse intersections of multiplicity at most $2$.  A skein is \emph{affine}  if the corresponding set is a regular skein composed of piecewise affine loops. The set of skeins, regular and affine skeins are respectively denoted by $\Sk(\R^2), \Sk_r(\R^2)$ and $\Ec_\Ac$.    For  any $S\in \Sk(\R^2)$, $\#S$ and $\ell(S)$ denote respectively the number of elements of $S$ counted with multiplicity and the sum of lengths counted \emph{without multiplicity}.
\end{dfn}  
We endow  $\Sk(\R^2)$     with the quotient topology for the  map $\coprod_{m\ge 1} \Ld(\R^2)^m\to  \Sk(\R^2)$, where for each $m\ge 1,$ $\Ld(\R^2)^m$ is endowed\footnote{with this convention, the function $\ell: \Sk(\R^2)\to \R_+$ is not continuous but upper continuous.} with the product topology.  It is elementary to show that the spaces $\Sk_r(\R^2)$ and $\Ec_\Ac$ are  dense in $\Sk(\R^2).$ If   the multiset $\Sc=\{l_1,\ldots, l_m\}$ is a skein, let us define  for any $N\in \N^*,$ 
\begin{equation}
\Phi_N(\Sc)=N^{m-2}C_m(\Tr(H_{l_1}),\ldots, \Tr(H_{l_m})),\label{Normalization Wilson skein}
\end{equation}
where the cumulants are with respect to the measure $\YM_N$. Observe that the law of the unitary Brownian motion is invariant under complex conjugation. Hence, for any skein $\Sc=\{l_1,\ldots, l_m\}\in \Sk(\R^2),$  denoting  $\Sc^*=\{l_1^{-1},\ldots, l_m^{-1}\},$  $$\Phi_N(\Sc)=\overline{\Phi_N(\Sc)}=\Phi_N(\Sc^*) $$  is real-valued. 

\begin{prop} For any affine skein  $\Sc\in \Ec_{\Ac}$,  the sequence $\Phi_N(\Sc)$ converges  as $N\to \infty$. We denote its  limit by $\Phi(\Sc).$ 
\end{prop}
\begin{proof} For any affine skein $\Sc$, there exists an embedded graph $\Gbb$ such that the elements of $\Sc$ belong to $\Ld(\Gbb). $ Choosing an arbitrary base point $v\in \Vbb$ and decomposing each loop in a lassos basis, yields that under $\YM_N$, the random family $(H_l)_{l\in\Sc}$ has the same law as a collection of words in marginals of independent $\Un(N)$ Brownian motions. Therefore, the Proposition  \ref{PropoSystDiff Cut and Join} implies the result.
\hfill\qed\end{proof}

\begin{prop}  \label{UNIFOr}Let us fix a constant $K>0$. For any skein $\Sc\in\Ec_\Ac$ of loops of length smaller than $\frac{K}{3}>0$ and taking their values in a  ball   of radius $\frac{K}{3}$, $$ |\Phi_N(\Sc)-\Phi(\Sc)|\le \frac{\pi \#\Sc^2K^2}{N^2}e^{\pi \#\Sc^2 K^2}.$$
\end{prop}

\begin{proof}   Let us assume that $\Sc=\{l_1,\ldots, l_m \}$ is a family of loops in $\Ec_\Ac$ all based at $0$.  Let  $\Gbb_\Sc=(\Vbb_\Sc,\Ebb_\Sc,\Fbb_\Sc)$ be the embedded graph with vertices the set of intersection points of the elements of $\Sc$ and with edges the restriction of elements of $\Sc$ between points of intersection.  The loop $l_1l_2\ldots l_n$ satisfies the conditions of Lemma \ref{borneAmpcombi}. Let us choose $T$ as in this Lemma and decompose each element of $\Sc$ in the corresponding lassos basis $\lambda^T$. Then, the second inequality of Lemma \ref{Exponential Bound Amperean area}, for $k=1$ and the bound of Lemma \ref{borneAmpcombi}  imply 

\begin{align} 
|\Phi_N(\Sc)-\Phi(\Sc)| &\le  N^{-2}A_t(w^T_{l_1 \ldots  l_m})  e^{A_t(w^T_{l_1\ldots  l_m})}  \nonumber\\
&\le  \frac{\pi }{N^2}(\ell(l_1)+ \cdots +\ell(l_m))^2e^{ \pi (\ell(l_1)+ \cdots +\ell(l_m))^2}.\label{vitconvbase}
\end{align}
Consider now $\{l_1,\ldots,l_m\}\in \Ec_\Ac$ satisfying the assumption of the Proposition. For any $\ep>0,$ let us choose piecewise affine paths $c_1,\gamma^\ep_1,\ldots,c_m, \gamma^\ep_m,$ such that $\underline{\gamma^\ep_i}=0=\overline{c_i}, \overline{\gamma^\ep_i}=\underline{l_i}=\underline{c_i}, \ell(\gamma^\ep_i), \ell(c_i)\le K(1+\ep),$ for any $i\in[m]$, $\Sc^\ep=\{\gamma^\ep_i l_ic_i,i\in [m]\}$ is an affine skein and $\gamma^\ep_i\to c_i^{-1},$ for any $i\in[m]$.  Theorem  \ref{continuiteYM} implies that $\Phi_N(\Sc^\ep)\to \Phi_N(\{c^{-1}_i l_ic_i,i\in [m]\})=\Phi_N (\Sc).$ The application of the bound (\ref{vitconvbase}) for $\Sc^\ep$, uniform in $\ep,$ implies that $\Phi_N(\Sc)$ admits a limit $\Phi(\Sc),$ as $N\to\infty,$ and that  the claimed bound holds true.   \hfill\qed\end{proof}

 This result allows then to extend the function $\Phi$ to all of $\Sk(\R^2).$ Surprisingly, an argument analogue to the proof of Theorem 5.14.  of \cite{MF} applies as well to the higher order case.
 \begin{thm}\label{ConvCMSUP} For any skein $\Sc\in\Sk(\R^2)$, the sequence $\Phi_N(\Sc)$ converges as  $N\to \infty$. We denote its limit by $\Phi(\Sc)$. The function $\Phi$ is a real-valued   continuous function  on $\Sk(\R^2)$. If $h$ is an area-preserving diffeomorphism of $\R^2$, for any $\Sc\in \Sk(\R^2),$ $\Phi(h(\Sc))= \Phi(\Sc).$
 \end{thm}
 We shall also call the function $\Phi: \Sk(\R^2)\longrightarrow \R $  \emph{planar master field}. 
 
 \begin{proof}For any $K>0,$ let $\Sk_K$ (respectively $\Ec_K$) be the set of  skeins $\Sc$ (respectively affine skeins with distinct loops) with elements included in the ball of radius $r$ around $0$ and with length less than $r,$ with $r(\#\Sc)^2 =K$.  As $\cup_{K>0}\Sk_K=\Sk(\R^2),$ it is enough to prove the result on $\Sk_K.$ The set $\Ec_K$ is dense in $\Sk_K.$  Indeed, any loop of  $\Ld(\R^2)$ can be approximated by its linear interpolation, which itself can be approached by piecewise linear loops with simple intersections, without increasing its length. According to Theorem  \ref{continuiteYM}, for any $N\ge 1,$ the function $\Phi_N$ is  continuous on  $\Sk_K.$ Moreover, Proposition \ref{UNIFOr} shows that $\Phi_N$ converges uniformly towards $\Phi$ on $\Ec_K.$ Therefore, $\Phi_N$ converges uniformly on $\Sk_K$ towards the unique continuous extension $\tilde{\Phi}$ of $\Phi$ to $\Sk_K$.    \hfill\qed\end{proof}
 
 A consequence of this Theorem is that for any $m\ge3,$ and any loops $l_1,\ldots , l_m$ in $\Ld(\R^2)$,   under $\YM_N,$   $$C_m(\Tr(H_{l_1}),\ldots,\Tr(H_{l_m}))= N^{2-m}\Phi_N(l_1,\ldots,l_m)\to 0,$$ 
 as $N\to \infty.$
 The following theorem follows.
 \begin{thm}  Under $\YM_N$,  the random family $(\Tr(H_l)-\esp[\Tr(H_l)])_{l\in\Ld(\R^2)}$  converges  weakly   as $N\to \infty,$ towards a Gaussian field  $(\phi_l)_{l\in\Ld(\R^2)}$, such that for any $a,b\in \Ld(\R^2),$ $\cov(\phi_a,\phi_b)=\Phi(\{a,b\})$.  If $(l_n)_{n\ge 0}$ is a fixed sequence of loops in $\Ld(\R^2) $ that converges towards $l\in\Ld(\R^2)$, then  $\phi_{l_n}\to\phi_l$ in distribution. If $h$ is an area preserving diffeomorphism of $\R^2,$  the process $(\phi_{h(l)})_{l\in\Ld(\R^2)}$ has the same law as $(\phi_l)_{l\in\Ld(\R^2)}$.
 \end{thm}

\subsection{Yang-Mills measure with a polynomial potential}
When $V$ is a function on $\Ld(\R^2),$   let us set $\Ld_V=\{l\in \Ld(\R^2): V(l)\not=0\} $  and $\Sk_V=\{\Sc\in \Sk(\R^2): \forall l\in \Sc, V(l)\not=0,\}.$  We shall consider the space  $\Fc_1$  of functions $V$  such that $\Ld_V$ is countable and
$$\|V\|_1=\sum_{l\in \Ld_V} |V(l)|<\infty.  $$
For any $V\in \Fc_1$ and $N\in\N^*,$  under $\YM_N$, almost surely the  following sum converges absolutely in operator norm and defines a random variable 
$$V_N=\sum_{l\in \Ld_V}V(l) H_l. $$
For any function $V\in \C^{\Ld(\R^2)},$ let us say that $V$ is \emph{symmetric} if for any $l\in \Ld(\R^2),$ $V(l^{-1})=\overline{V(l)}.$ For any symmetric function $V\in\Fc_{1}$, the random matrix $V_N$ is Hermitian and its operator norm is bounded by $\|V\|_1.$ In particular,  $0<\esp[e^{N\Tr(V_N)}]<\infty$. Let   $\YM_{N,V}$ be the probability measure on $(\Mca_N(\Pd(\R^2)),\Cc)$, whose density  with respect to $\YM_N$ is $\esp[e^{N\Tr(V_{t,N})}]^{-1}e^{N\Tr(V_N)}.$  We shall  denote by $(H^V_l)_{l\in \Ld(\R^2)}$  the canonical process on $\Mca_N(\Pd(\R^2))$ with law $\YM_{N,V}.$  When $\Sc\in \Sk(\R^2)$ has multiplicities $m_1,\ldots,m_k,$ we set 
$$m_\Sc=m_1!\ldots m_k!.$$For any $V\in  \C^{W_q}$ and $a\in\Ld(\R^2),$ let us define
$$\Jc_V(a)= \sum_{\Sc\in \Sk_V} m_\Sc^{-1}\left(\pi \ell( \{a\}\cup\Sc)^{2}\#\Sc\right)^{\#\Sc} \prod_{l\in \Sc}|V(l)| $$
and for $a,b\in \Ld(\R^2),$
$$\Jc_V(a,b)=\sum_{\Sc\in \Sk_V}  m_\Sc^{-1}\left(\pi \ell( \{a,b\}\cup\Sc)^{2}\#\Sc\right)^{\#\Sc}  \prod_{l\in \Sc}|V(l)|. $$
\begin{thm} For any symmetric function $V\in \Fc_1$  and $l\in \Ld(\R^2)$ such that $\Jc_V(l)<\infty,$  $\esp_{\YM_{N,V}}(\frac{1}{N}\Tr(H_l))\to \Phi_V(l),$ as $N\to\infty$, where  
$$\Phi_V(l)=\sum_{\Sc\in \Sk_V} \Phi(\{l\}\cup\Sc)m_\Sc^{-1}\prod_{l'\in \Sc}V(l') $$
is absolutely converging.  Moreover,
$$\Var_{\YM_{N,V}}(\frac{1}{N}\Tr(H_l))\le  \frac{1}{N^2}\Jc_V(l,l).$$
\end{thm}
\begin{proof}   For any  affine skein $\Sc\in \Ec_\Ac,$  with $m$ elements $l_1,\ldots,l_m,$    let us choose an embedded graph $\Gbb_\Sc$ with area's vector $t$, as in the proof of Proposition \ref{UNIFOr}, for the affine skein without multiplicities associated to $\Sc$. Then, if   $T$  is spanning tree of $\mathbb{G}_\Sc$ as in Lemma \ref{TreeGeo},  Lemma \ref{borneAmpcombi} implies that for any  $l\in \Sc,$ $\|\overline{n}_{w_l^T}\|^2 \le \pi \ell(\Sc)^2$.  Then, for any  Cayley tree $\mathfrak{T}_t\in\Cc_m$, $\mathfrak{T}_t(w_{l_1}^T,\ldots,w_{l_m}^T)\le \pi^{m-1}\ell(\Sc)^{2m-2}$ (recall the definition  (\ref{complex arbre}) of the left hand side).  Combining this inequality with  Proposition \ref{Cumulant's control},  yields for every $N\in\N^*,$
 $$|\Phi_N(\Sc)|\le m^{m-2} \pi^{m-1}\ell(\Sc)^{2m-2}. $$
 By continuity of $\Phi_N$, it follows that this inequality holds true for all skein $\Sc\in\Sk(\R^2).$
 Let us fix a symmetric function  $V\in \Fc_1$  on $\Ld(\R^2)$ and $l\in \Ld(\R^2)$ such that $\Jc_V(l)<\infty$. For every $N\in \N^*,$  the function $x\in\R\mapsto\esp(\exp(x N\Tr(V_{N})))$ is analytic. According to the previous inequality and to the assumption on $V$, its Taylor expansion at $0$ has a radius bigger than one and equals $$ \sum_{\substack{m\ge 0 \\l_1,\ldots, l_m\in\Ld(\R^2)}}\frac{1}{m!}x^m\prod_{i=1}^m V(l_i) \Phi_N(l,l_1,\ldots, l_m)=\sum_{\Sc\in \Sk_V}\frac{x^{\#\Sc}}{m_\Sc}\prod_{l\in \Sc}V(l)\Phi_{N}(\{l\}\cup\Sc).$$
 Each term of the sum  over $m$  being uniformly bounded in $N,$ by  dominated convergence, Theorem \ref{ConvCMSUP} yields the first result. The bound on the variance follows by a similar argument, using equality (\ref{VarPot}). \hfill\qed\end{proof}
\begin{rmk} For any $V\in \Fc_1$, let us write $\|V\|_\infty=\sup_{l\in \Ld(\R^2)} |V(l)|$ and $\ell(V)=\sup_{\Sc\in\Sk_V}\ell(\Sc). $ For any $V\in \Fc_1$ and  $l\in \Ld(\R^2)$, if $(\ell(V)+\ell(l))^2<\frac{1}{e\pi\|V\|_\infty},$ then  $\Jc_V(l)<\infty.$
\end{rmk}

\subsection{Small area limit \label{Section Small area limit}}

For any $\alpha>0$ and any loop $l\in \Ld(\R^2)$, denote by $\alpha.l$ the image of $l$ by the  dilatation of rate $\alpha,$  centered at $0$. If $\Sc=\{l_1,\ldots,l_m\}$ is a skein, $\alpha.\Sc= \{\alpha.l_1,\ldots,\alpha.l_m\}$. The following proposition shows that, as $\alpha\to 0$, all the quantities defined above have the same behavior as $\alpha\to 0$.

\begin{prop}
The following Taylor expansions are true for any $N\in\N^*$.   As $\alpha\to0, $ for any   loop $l\in\Ld(\R^2)$,
 $$\Phi_N(\sqrt{\alpha}. l)=1-\frac{\alpha}{2}\int_{\R^2} n_l^2(x)dx+O(\alpha^2)= \Phi(\sqrt{\alpha}. l)+ O(\alpha^2)$$
and for any  skein $\Sc$  with at least two loops,
\begin{align*}
\Phi_N(\sqrt{\alpha}.\Sc)&= (-\alpha)^{\#\Sc-1} \sum_{\mathfrak{T}_\Sc} \prod_{\{l_1,l_2\}\in \mathfrak{T}_\Sc}\int_{\R^2}n_{l_1}(x)n_{l_2}(x)dx+O(\alpha^{\#\Sc})\\
&=\Phi(\sqrt{\alpha}.\Sc)+O(\alpha^{\#\Sc}),
\end{align*}
where the sum is over connected graph with vertices $\Sc$ and $\#\Sc-1$ edges. In both cases, there exists a  positive  continuous function $b$, independent of $N$, such that $O(\alpha^{|\Sc|})\le \alpha^{|\Sc|}b(\sum_{l\in \Sc} \ell(l))$.
\end{prop}
\begin{proof} If $\Sc\in \Ec_\Ac,$ the assertion is a direct consequence of  Proposition \ref{Cumulant's control} and (\ref{winding word loop}). Continuity of the functions $\Phi_N,\Phi$ and $b$  allows then to conclude.\hfill\qed\end{proof}

A direct consequence is the following
\begin{coro}  Let $W$ be a white noise on $\R^2$, with intensity given by the Lebesgue measure. As $t\to 0,$ the Gaussian field $(\frac{1}{t}\phi_{t. l})_{l\in \Ld(\R^2)}$, as well as, for any $N\in\N^*$, the random family $(t^{-1}(\Tr(H_{t.l})-N\Phi(t.l) ))_{l\in \Ld(\R^2)}$, under $\YM_N,$  converge in distribution towards  the Gaussian field $(i W(n_l))_{l\in \Ld(\R^2)}.$
\end{coro}

\section{Makeenko-Migdal equations}

We shall now address the problem of the  computation and characterization of the master field.  Let  $\overline{\Sk_r}(\R^2)$ be the  quotient  of $\Sk_r(\R^2)$ under the action of diffeomorphisms of the plane.  For any integer $n$, the set of equivalence classes of skeins  with less than $n$ intersections is finite. Thanks to its invariance property under area-preserving diffeomorphisms and to its continuity, the master field  is  characterized by its value on $\Sk_r(\R^2)$ and yields functions  indexed by $\overline{\Sk_r}(\R^2)$ that can be expressed inductively.

\subsection{Makeenko-Migdal equations for the master field on skeins}

For any skein $\Sc$, let us denote by $W_\Sc$ the expectation $\esp_{\YM_N}\left[\prod_{l\in \Sc}\Tr(H_l)\right]$,  we call this function a \emph{Wilson skein}\footnote{We warn the Reader that these functions are not normalized as they can be in the literature, so that, with this conventions, $W_{\text{cst}}=N$.} and say it is regular whenever the associated skein is. In view of the definition of discrete Yang-Mills measure, one may try to compute  the master field of higher order  of a  regular skein $\Sc$ using Itô formula to yield a first order differential system for the family $(W_\Sc)_{\Sc\in \Sk_r(\R^2)}$, with  areas of the faces of a graph $\Gbb$ containing $\Sc$ as variables.  However, this differential system yields at first sight non-regular Wilson skeins $W_\Sc$ as features  the  example \ref{exampleInstability}.   \begin{example}  \label{exampleInstability}Consider a loop  $l$ that winds three times around the origin.  Let us name the faces $A,B$ and $C$ and choose a lassos basis $(l_A,l_B,l_C)$ according to a spanning tree as illustrated in figure \ref{SystemeInstableD3} in dashed lines.  In this basis, the loop is decomposed as $l=l_Cl_Bl_A^2l_Bl_A.$

\begin{figure}[!h] 
  \centering
 \includegraphics[height=1.5in]{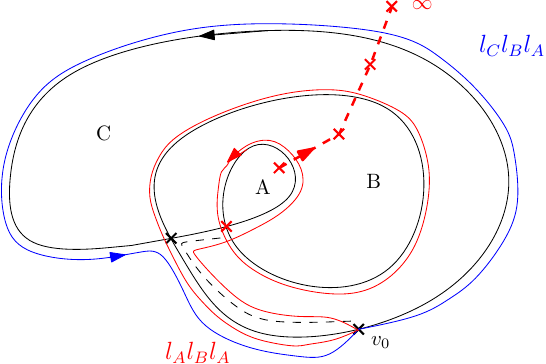}\caption{ \label{SystemeInstableD3}}
\end{figure} 
\noindent Using Itô formula as described in Lemma \ref{EDPTens} and differentiating with respect to the area of the faces $C$ and $B$ yields $\frac{d}{d|C|}W_l= -\frac{W_l}{2}$ and 
$$N\left(\frac{d}{d|B|}(W_l)+W_l\right)= -W_{\{l_Bl_A^2,l_Cl_Bl_A\}}= -W_{\{l_A l_Bl_A,l_Cl_Bl_A\}}.$$
These first two derivatives can be expressed in terms of regular Wilson skeins. However, the derivative with respect to the face of index 3 yields terms that do not seem to be polynomials of regular Wilson skeins: 
$$N\left(\frac{d}{d|A|}\left(W_{l}\right)+\frac{3}{2}W_l\right)= -W_{l_Bl_A,l_Cl_Bl_A^2}-W_{\{l_Cl_Bl_A,l_Al_Bl_A\}}-W_{\{l_A,l_Bl_Al_Cl_Bl_A\}}. $$

\end{example}
For any regular skein $\Sc$, one must therefore face the problem of finding a closed system of Wilson skeins containing $W_\Sc$. The system of equations given by Lemma \ref{EDPTens} gives such a system, but its size happens to grow exponentially with the number of faces of the original skein (see section 6.8 of \cite{MF}, where the smallest closed system obtained is made of what is called therein Wilson garlands). The Makeenko-Migdal equation solves this problem and  gives linear combinations of  area derivatives  operators that preserve the set of function indexed by  skeins, so that the size of the system grows as a polynomial in the number of faces.   Let $\Sc$ be a regular skein  and $x$ be a point of intersection of its elements (between themselves or each other).   Let us denote by $\Sc_x$ the skein composed with the same loops as $\Sc $ except for the loop or the pair of loops  containing $x$ that is  replaced respectively  by the pair of loops or the  loop based at $x$, which instead of going straight along the same strand of $\Sc$, turns  at the point $x$ using the other outgoing strand  (see figure \ref{lateraldodge}).

 \begin{figure} \centering \includegraphics[height=2 in]{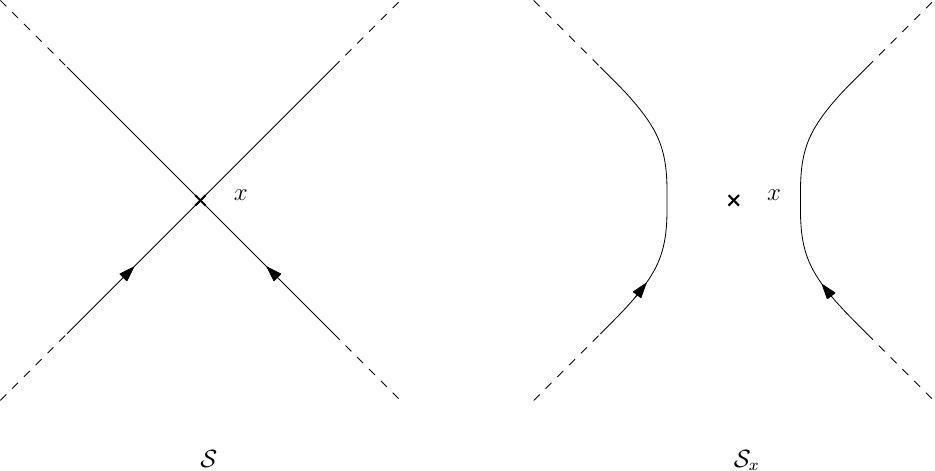}\caption{ \label{lateraldodge} Local transformation at an intersection point $x$ of a skein $\Sc$.}\end{figure}
The following proposition is proved in \cite{MF}, in a more  general framework \footnote{for intersections of higher degree and for classical compact Lie groups} and relies on integration by  parts applied to the product of  a function on $\Mca_N(\Pd(\Gbb))$ with the density  of the discrete Yang-Mills measure. We provide here another proof relying on the decomposition in lassos described in section \ref{Basis RL}  and on the invariance of the Brownian motion by adjunction. Let us fix a regular skein $\Sc$ and an embedded graph $\Gbb,$ such that elements of $\Sc$ belong to $\Pd(\Gbb)$.
\begin{prop}[Makeenko-Migdal equation] \label{Makeenko} Let $F_1,\ldots, F_4$ be the four faces of $\Gbb$ around a point of intersection  $x\in\Vbb$ of $\Sc$  in a cyclic order and such that $F_1$ is the face bounded by  the two incoming edges of $\Sc$ at $x$.  If their area are respectively parametrized by $t_1,\ldots,t_4,$ then, 

\begin{equation}
\left(\frac{d}{dt_1}-\frac{d}{dt_2}+\frac{d}{dt_3}-\frac{d}{dt_4}\right)\esp_{\YM_N}\left[\prod_{l\in \Sc}\Tr(H_l)\right]=\frac{1}{N}\esp_{\YM_N}\left[\prod_{l\in \Sc_x}\Tr(H_l)\right].\label{Makequation}
\end{equation}
\end{prop}
In the latter left hand side, when two faces $F_i,F_j$ agree or if  $F_k$  is unbounded, then by convention $\frac{d}{dt_i}=\frac{d}{dt_j}$ or $\frac{d}{dt_k}=0.$ Using  this convention, we denote by $\mu_x$ the operator  $\frac{d}{dt_1}-\frac{d}{dt_2}+\frac{d}{dt_3}-\frac{d}{dt_4},$  where the faces are numbered as in the Proposition. For any skein $\Sc$, let us set $n_\Sc=\sum_{l\in\Sc}n_l$.  Notice that for $N=1$, the equality (\ref{Makequation}) is equivalent to the fact $$n_\Sc(F_1)^2-n_\Sc(F_2)^2+n_\Sc(F_3)^2-n_\Sc(F_4)^2=-2.$$ 
The strategy of our proof is to choose  an embedded graph $\mathbb{G}$ containing $\Sc$ and an appropriate basis of $\RL_x(\Gbb)$, so that  using properties of  section \ref{Basis RL} and  Lemma \ref{EDP cut and join},  terms on the left-hand-side of (\ref{Makequation}) cancel themselves leaving a single cut and join transformation.    Such a cancellation appears in the following situation. Let us recall  notation of section \ref{Section differential system} and set  for any words  $w\in W_q^m,$ and $t\in \R_+^q,$  $ E_t(w)  = K_t(w,1_m).$

\begin{lem} \label{Combi MM}For $q\ge 4,$ let  $w_1,\ldots,w_m$ be  $m$ words such that 
$$w_i=\mathfrak{m}_i(x_1x_2x_3x_4, x_2x_3, x_3x_4,x_5,\ldots, x_q),$$
with $\mathfrak{m}_i\in W_{q-1}.$  Let us assume that   
\begin{equation}
\overline{n}_\mathfrak{m}(2)=\overline{n}_\mathfrak{m}(3)=1,\label{linear MakeenkoMigdal}
\end{equation}
where $\mathfrak{m}=\mathfrak{m}_1\ldots\mathfrak{m}_m.$   If    $y^{\ep_a},y^{\ep_b}$   occur  in position $a$ and $b$   in $\tilde{w}_1\ldots \tilde{w}_m,$  where   $\tilde{w}_i=\mathfrak{m}_i(x_1x_2x_3x_4, x_2y, yx_4,x_5,\ldots, x_q),$ for $i\in [m],$ then
$$\left(\frac{d}{dt_1}-\frac{d}{dt_2}+\frac{d}{dt_3}-\frac{d}{dt_4}\right) E_t( w_1,\ldots, w_m)= -\frac{\ep_a\ep_b}{N}E_t(\Tc_{a,b}(w_1,\ldots,w_m)).  $$
\end{lem}

\begin{proof}

 Let $A$ be a subset  of $[\ell(w)]$ such that the restriction of $w=w_1\ldots w_m$ to $A$  is of the form $v_1(x_ix_j)\ldots v_m(x_ix_j),$ with  for any  $p\in[m],$  $v_p\in W_1$ and $v_p(x_ix_j)$ is a restriction of $w_p$. Then,  an inspection of  the definition of section \ref{Section CutJoin} yields that  for any $\ep\in \{-1,1\},$    $\Nc_w^\ep(j)\cap A^2=(\ep,\ep)+\Nc_w^\ep(i)\cap A^2$ and 
\begin{equation}
\Tc_{p,q}(w_1,\ldots,w_m) =\Tc_{p+\ep,q+\ep}(w_1,\ldots,w_m). \label{invariance combi}
\end{equation}
Let $A,B,C$ be  the set of occurrences of  $x,y,z$ in  $\hat{w}_1\ldots \hat{w}_m,$  where   for each $i\in [m],$ $\hat{w}_i=\mathfrak{m}_i(x^{4}, y^2,z^2,x_5,\ldots, x_{q-1}).$ According to Lemma \ref{EDP cut and join} and (\ref{invariance combi}), 
\begin{align*}
\left(\frac{d}{dt_1}-\frac{d}{dt_2}+\frac{d}{dt_3}-\frac{d}{dt_4} \right)E_t(\mathbf{w})&=\frac{\overline{n}_w(4)-\overline{n}_w(3)+\overline{n}_w(2)-\overline{n}_w(1)}{2}E_t(\mathbf{w})\\
& - \frac{1}{N} \sum \ep E_t(\Tc_{p,q} \mathbf{w}),\end{align*}
where the sum is over $\ep\in \{-1,1\}$ and $(p,q)\in \Nc^{\ep}_w(3)\cap (B\times C)$.   By assumption the alternated sum of the right hand side vanishes and $\Nc_w^{\ep}(3)\cap(B\times C)=\{(a,b)\}$ if $\ep=\ep_a\ep_b$ and the empty set otherwise.\hfill\qed\end{proof}

\begin{dfn} \label{def conj skeins} We say that two  skeins $\Sc= \{l_1,\ldots, l_m \},\Sc'=\{l'_1,\ldots, l'_m\}$ are conjugated and write $\Sc \equiv \Sc',$  if there exist $\gamma_1,\ldots,\gamma_m \in \Pd(\R^2)$ with $\underline{\gamma}_i=\underline{l}'_i,$  $\overline{\gamma}_i=\underline{l}_i$  and $l'_i=\gamma_il_i\gamma^{-1}_i$, for any $i\in [m].$ 
\end{dfn}

\begin{lem} \label{choice basis}Let $x$ be a point of intersection of a regular skein $\Sc$, whose loops are all based at $y\not=x$.  There exists an embedded graph $\Gbb',$  with $\Pd(\mathbb{G})\subset \Pd(\Gbb'),$ a  lassos basis\footnote{as defined in section \ref{Section Discrete YM}}  $\Lambda$ of $\RL_y(\Gbb')$ and a labeling of faces of $\Gbb'$ matching the condition of Proposition \ref{Makeenko}, such that a decomposition of elements of $\Sc,$ into  words $w_1,\ldots,w_m$  in $\Lambda,$  satisfies the condition of Lemma  \ref{Combi MM}, with moreover  $\ep_a\ep_b=-1$ and 
\begin{equation}
\{\tilde{w}_i(\lambda, \lambda\in \Lambda): 1\le i\le m' \} \equiv \Sc_x,\label{mot MM}
\end{equation}
where $ (\tilde{w}_1,\ldots, \tilde{w}_{m'})=\Tc_{a,b} (w_1,\ldots,w_m).$ 
\end{lem}

\begin{proof}[Proposition \ref{Makeenko}] Let us choose a  graph $\mathbb{G}'$ and a lassos basis  $\Lambda=(\lambda_F)_{F\in \mathbb{F}'_b}$ according to Lemma \ref{choice basis}.   Then, according to Lemma \ref{Discrete YM} and Theorem \ref{continuiteYM}, under $\YM_N$,  $(H_\lambda)_{F\in \mathbb{F}'_b}$  has the same law as $(U_{F,  |F| })_{F\in \mathbb{F}'_b},$ where $(U_F)_{F\in \in \mathbb{F}'_b}$ are $\# \mathbb{F}'_b$ independent $\Un(N)$-Brownian motions.  Let $t\in \R_+^{\mathbb{F}'_b}$ be the vector of faces area of $\Gbb'.$ For any skein $\tilde{\Sc},$  with elements in  $\Ld( \Gbb'),$ let $\{\gamma_l,l\in \tilde{\Sc}\}$  be  a family of paths of $\Pd(\mathbb{G}'),$ with $\underline{\gamma}_l=y$ and $\overline{\gamma}_l= \underline{l}$  and denote by   $(w_1,\ldots,w_m)$ a decomposition of $\tilde{\Sc}'=\{\gamma_l l \gamma^{-1}_l,l\in \tilde{\Sc}\}$ in the basis  $\Lambda.$ Then,  as $H\in \mathcal{M}_N(\Pd(\R^2)),$
\begin{align*}
\esp[\prod_{l\in \tilde{\Sc}} \Tr(H_l)]=\esp[\prod_{l\in \tilde{\Sc}'} \Tr(H_{l})]=E_t(w_1,\ldots, w_m)
\end{align*}
and $ (\frac{d}{dt_1}-\frac{d}{dt_2}+\frac{d}{dt_3}-\frac{d}{dt_4})E_t(w_1,\ldots, w_m)=\mu_x  \esp[\prod_{l\in \Sc} \Tr(H_l)].$  Applying Lemma \ref{Combi MM}  and  the  former equality to $\Sc_x$ implies the claim.\hfill\qed\end{proof}
\begin{proof}[Lemma \ref{choice basis}] Let us consider an embedded graph $\Gbb'$ with $\Pd(\Gbb')\supset\Pd(\Gbb),$  with a vertex $y\in \Vbb'\setminus \{x\}$ and  labeled faces, such that the faces $F_1,\ldots,F_4$ are neighboring $x$ in $\Gbb'$, bounded, distinct,  in clockwise order, with $F_1$ bounding the two ingoing edges and such that the graph    $\displaystyle(\hat{\Vbb}'\setminus \{F_1,\ldots,F_4\},\Ebb'\setminus \cup_{i=1}^4\{e, F_i\in e\})$ is connected. 
 \begin{figure}\centering\includegraphics[height=2 in]{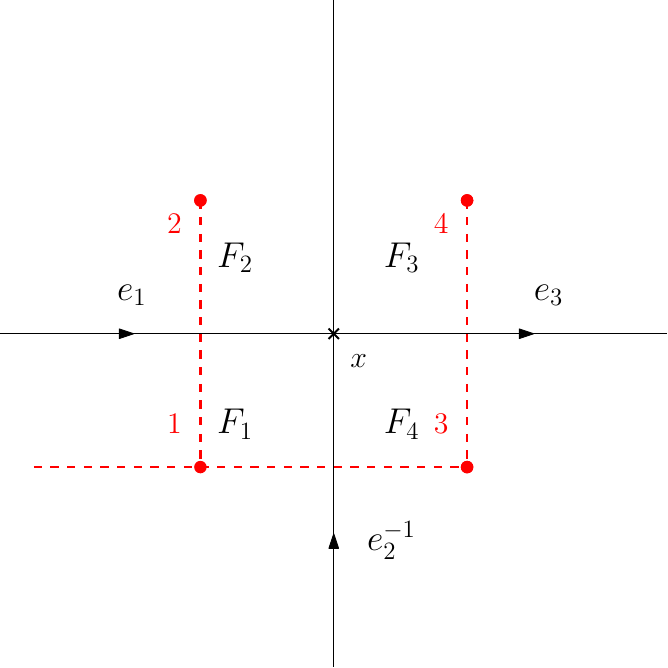}\caption{A spanning tree $\hat{T}$ of  $\hat{\Gbb}',$  such that  $\hat{T}_{F_1}$ is a tree with vertices $\{F_i,i\in[4]\},$ with $F_2$ and $F_3$ as leaves. The order induced by the clockwise contour is denoted in red. \label{PictureChoiceSpanningTree}}\end{figure}
  Such a graph can be obtained by splitting successively the faces of $\mathbb{G}.$  Let us choose a spanning tree\footnote{recall the notation below Lemma \ref{Lemme  lassos base}} $\hat{T}$ of $\hat{\Gbb}',$ such that  the tree  $\hat{T}_{F_1} $ has  vertices $\{F_i,i\in[4]\}$ and edges $\{(F_2,F_1),(F_3,F_4),(F_4,F_1)\},$ see figure  \ref{PictureChoiceSpanningTree}. We denote by $T$ the spanning tree of $\Gbb'$ dual to $\hat{T}.$  Let us consider the basis of lassos $\Lambda^T$  rooted at $y$, as defined in  section \ref{Basis RL}.  Let us fix an ordering  $(l_1,\ldots,l_m)$ of $\Sc$.  According  to (\ref{decompositionenbeta}) and (\ref{lassobeta}),  each element $l_i$  is decomposed into   $\tilde{w}_i(\lambda_{F_i}, i\in [q])$, with $q=\# \Fbb_b',$  
$$\tilde{w}_i= \tilde{\mathfrak{m}}_i (x_1x_2x_4x_3, x_2x_3,x_4x_3, x_5,\ldots,x_q)$$ and $\tilde{\mathfrak{m}}_l\in W_{q-1}$. Moreover,  if $\tilde{\mathfrak{m}}=\tilde{\mathfrak{m}}_1\ldots\tilde{\mathfrak{m}}_{m},$ as $F_2$ and $F_3$  are leaves of $\hat{T},$ $\overline{n}_{\tilde{\mathfrak{m}}}(2)=\overline{n}_{\tilde{\mathfrak{m}}}(3)=1.$  Let us emphasize that as $y\not = x,$  $x_2$ and $x_3$  occur consecutively in $\tilde{w}_{i}$  only  once, for $i_0$  such that $l_{i_0}$ goes through the edges dual to $(F_2,F_1)$ and $(F_3,F_4).$  Let us denote by $\Theta_{3,4}$ the automorphism of the free group  $\RL_x(\Gbb'),$ that maps $\lambda_{F_4}$ to $\lambda_{F_3}^{-1}\lambda_{F_4}\lambda_{F_3}$ and fixes $\lambda_F$  for $F\not=F_4$. Then, the family $\Sc$ admits a decomposition $(w_1,\ldots,w_m)$ into the lassos basis $\Theta_{3,4}(\Lambda^T)$ satisfying the condition of Lemma \ref{Combi MM}.  As the skein crosses the edge  $(F_3,F_4)$ (directed towards $F_{\infty,\mathbb{G}'}$) from right to left,  whereas  it crosses $(F_4,F_1)$ from left to right, $\ep_a\ep_b=-1.$ To conclude, it remains to identify the right hand side of Lemma  \ref{Combi MM} as a decomposition of a skein in $\Theta_{3,4}(\Lambda^T).$  We shall  only detail the case when $x$ is a point of self-intersection of $\Sc$, the case of the intersection of two different loops being similar.  Let  $l_{i_0}$ be the loop of $\Sc$ with an intersection point at $x$ and consider  the three edges $e_1,e_2,e_3$ of $\Gbb'$  around $x,$ crossing $\hat{T}$ from right to left, with the counterclockwise order (see figure \ref{PictureChoiceSpanningTree}).   The decomposition of $l_{i_0}$ in the basis $\beta$ (recall (\ref{decompositionenbeta})) is  
$$l_{i_0}=     X \beta_{e_1}  \beta_{e_3}  Y \beta_{e_2}^{-1} Z \,\,\text{   or   }\,\,  X \beta_{e_2}^{-1}  Y \beta_{e_1}  \beta_{e_3} Z, $$
where $X,Y,Z$ are words in $\{\beta_e: e \in \Ebb\setminus \left(T\cup \{e_1,e_2,e_3\}\right)\},$ so that   
$$w_{i_0}= W_X x_2 x_3 W_Y  x_4^{-1}x_3^{-1} W_Z  \,\,\text{   or   }\,\, W_X x_4^{-1}x_3^{-1} W_Y   x_2 x_3  W_Z ,$$
where  $W_X,W_Y,W_Z$ are words of the form $m(x_1x_2x_4x_3,x_5,\ldots,x_q)$ for some $m\in W_{q-3}$. Then, $\Tc^-_{a,b} \mathbf{w}$ is a permutation of 
$ (w_1,\ldots, \hat{w}_{i_0},  \ldots, w_m, w_l, w_r ),$ with 
$$w_l= W_Xx_2 W_Z,\, w_r= x_3W_Yx_4^{-1}x_3^{-1}\,\,\text{   or   }\,\, w_l=  W_Yx_2, \,w_r= W_X x_4^{-1}x_3^{-1}x_3 W_Z.$$ 
In the first case, $w_l$ is    the decomposition  into $\Theta_{3,4}(\Lambda^T)$ of  the loop rooted at $y$  following the strands of $l_{i_0}$ but  at $x$, where it turns left,  whereas $w_r$ is the decomposition of  $[y,\underline{e_3}]_T  l_x^R [\underline{e_3},y]_T,$ where $l_x^R$ is the loop based at $x,$ using the strands of $l_{i_0}$ starting with $e_3$, until it uses $e_2^{-1}.$ The second case is similar.
\hfill\qed\end{proof}
 For any regular skein $\Sc\in\Sk_r(\R^2)$, let us denote respectively  by $\Vbb_s(\Sc)$ and $\Vbb_f(\Sc)$, or simply $\Vbb_s,\Vbb_f$,  the set of self-intersection points of each loops  and the set of intersection points of pair of distinct loops  of $\Sc$.  If $x\in\Vbb_s$ is a point of intersection   of one loop $l\in \Sc$, we denote  respectively  by  $l_x^L$  and $l_x^R$ the loops  based at $x,$ that use respectively the  left and right outgoing edges, following the strand of $l$ until their first return to $x$. If $x\in \Vbb_f$ is the intersection point of two different loops $l_1$ and $l_2$ of $\Sc$, we denote by $l_1\circ_xl_2 $ the concatenation of the loops  obtained by rooting $l_1$ and $l_2$  at the point $x$. 


\begin{thm}\label{coromak}Let $\Sc=\{l_1,\ldots, l_m\}$ be a regular skein, $x \in\Vbb_\Sc$ a point of intersection and  $F_1,\ldots, F_4$ faces around $x$, with areas parametrized by $t_1,\ldots, t_4, $ as in Proposition \ref{Makeenko}. If $x$ is the intersection point of two different loops $l_1$ and $l_2$, 

\begin{equation*}
\left(\frac{d}{dt_1}-\frac{d}{dt_2}+\frac{d}{dt_3}-\frac{d}{dt_4}\right)\Phi_N( \Sc)=\Phi_N(l_1\circ_xl_2,l_3,\ldots, l_m).\tag{*}\label{Makun}
\end{equation*}
If $x$ is an intersection point of the loop $l_1$, then $\left(\frac{d}{dt_1}-\frac{d}{dt_2}+\frac{d}{dt_3}-\frac{d}{dt_4}\right)\Phi_N( \Sc)$ equals 
\begin{align*}\tag{**}\label{Makdeux}
\sum_{\substack{\Sc_x^L\sqcup\Sc_x^R=\Sc_x\\ l_{1,x}^L\in \Sc_x^L \text{ and } l_{1,x}^R\in \Sc_x^R}}\Phi_N(\Sc_x^L)\Phi_N(\Sc_x^R)+\frac{1}{N^2}\Phi_N(l_x^L,l_x^R,l_2,\ldots, l_m).
\end{align*}
Moreover, if $F\in \Fbb_b$ is a neighbor face of the unbounded face, then 
\begin{equation*}
\frac{d}{d|F|}\Phi_N(\Sc)=-\frac{1}{2} \Phi_N(\Sc).\tag{***} \label{Makbord}
\end{equation*}

\end{thm}
\begin{proof} For any loops $l_1,\ldots, l_m\in \Ld(\R^2)$ and  $\pi,\Nu\in\Pc_m$ with $\pi\le \Nu$, let us  set  $\esp_\pi[l_1,\ldots,l_m]= \prod_{B\in\pi}\esp[\prod_{i\in B}\Tr(H_{l_i})]$, $C_\Nu[l_1,\ldots, l_m]= \prod_{B\in\Nu} C_{\#B}(\Tr(H_{l_i}),$ $i\in B)$ and $C_{\pi,\Nu}[l_1,\ldots,l_m]=C_{\pi,\Nu}(\Tr(H_{l_i}),i \in [m])$.   Let us recall (\ref{Normalization Wilson skein}) and  consider the normalized cumulant
$$\Phi_\pi(l_1,\ldots,l_m)= \prod_{B\in \pi}\Phi(l_i, i\in B) = N^{m-2\#\pi}C_\pi(l_1,\ldots,l_m).$$ 
Assume that $x$ is an intersection point of $l_1$ and denote by $\overline{\{1,2\}}$ the smallest partition of $[m+1]$ containing $\{1,2\}.$ For any partition $\Nu$ of $[m+1]$  connecting $1$ with $2$,  denoting by $\Nu'$ the partition of $[m]$ obtained by identifying $2$ with $1$, 

\begin{align*}
\sum_{\overline{\{1,2\}}\le \pi\le \Nu} N\mu_x C_{\pi'}(l_1,l_2,&\ldots,l_m)=  \sum_{\pi\le \Nu'} N\mu_xC_\pi(l_1,l_2,\ldots,l_m)\\
&= N\mu_x \esp_{\Nu'}[l_1,l_2,\ldots,l_m]= \esp_{\Nu'}[l^L_{1,x}l^R_{1,x}, l_2,\ldots,l_m]\\
&=\esp_\Nu[l^L_{1,x},l^R_{1,x}, l_2,\ldots,l_m].
\end{align*}
Therefore,  for any partition $\pi\in \Pc_{m+1}$ connecting $1$ and $2$, 
$$C_{\overline{\{1,2\}},\pi}(l^L_{1,x},l^R_{1,x},\ldots,l_m)= N\mu_x C_{\pi'}(l_1,l_2,\ldots,l_m).$$
In particular, according to the Leonov Schiryaev formula (\ref{LeoS}),

\begin{align*}
\mu_x\Phi_N(l_1,\ldots,l_m)&= N^{m-3}C_{\overline{\{1,2\}},1_{m+1}}(l^L_{1,x},l^R_{1,x},\ldots,l_m)\\
&=\sum_{\pi\in\Pc_{m+1}:\, \pi \vee \overline{\{1,2\}}=1_{m+1}}N^{m-3}C_\pi(l^L_{1,x},l^R_{1,x},\ldots,l_m)\\
&=\sum_{\pi\in\Pc_{m+1}:\, \pi \vee \overline{\{1,2\}}=1_{m+1}}N^{2\#\pi-4}\Phi_\pi(l^L_{1,x},l^R_{1,x},\ldots,l_m)
\end{align*}
If $\pi\in\Pc_{m+1}$ satisfies $\pi\vee \overline{\{1,2\}}= 1_{m+1}$, then, whether $\#\pi=1$ or $\#\pi=2$ and the equation (\ref{Makdeux}) follows. Assume now that $x$ is an  intersection point of $l_1$ with $l_2$. Then, for any partition $\pi\in\Pc_{m}$,  $\mu_x\esp_\pi(l_1,l_2,\ldots,l_m)= \frac{1}{N} \esp_\pi(l_1\circ_xl_2,\ldots,l_m),$ if $1,2$ are  in the same block of $\pi$, and $0$ otherwise, by Leibniz rule.  This implies that for any partition $\pi\in\Pc_m$, such that $1$ and $2$ are not in the same block of $\pi$, $\mu _xC_\pi(l_1,l_2,\ldots,l_m) =0. $
Therefore, for any partition $\nu\in\Pc_{m-1},$ denoting by $\tilde{\nu}\in \Pc_{m}$ the partition  obtained by shifting $\pi$ by $1$ and adding $1$ to the block containing $2$,
\begin{align*}
\sum_{\pi \le \nu}N\mu_x  C_{\tilde{\pi}} (l_1, l_2,l_3,\ldots, l_m)&= \sum_{\Wc\le \tilde{\nu}}N\mu_x  C_{\Wc} (l_1, l_2,l_3,\ldots, l_m)\\
&=N\mu_x\esp_{\tilde{\nu}}[l_1,l_2,\ldots,l_m]= \esp_\nu[l_1\circ_xl_2,l_3,\ldots,l_m].
\end{align*}
It follows that for any  $\pi\in\Pc_{m-1}$, $\mu_x C_{\tilde{\pi}}(l_1,l_2,\ldots,l_m)=\frac{1}{N}C_\pi(l_1\circ_xl_2,l_3,\ldots,l_m). $ For $\pi=1_{m-1}$, the latter  equality yields  $$\mu_x \Phi_N(l_1,l_2,\ldots,l_m)= \Phi_N(l_1\circ_xl_2,l_3,\ldots,l_m).$$ \hfill\qed\end{proof}

Observe that equations (\ref{Makun}) and (\ref{Makbord})  on $\Phi_N$ do not depend on  $N$.

\subsection{Uniqueness for Makeenko-Migdal equations}  We  want to tackle this question  by  asking whether  area-derivative operators can be obtained by linear combinations of   the operators appearing on the left-hand-side of Theorem \ref{coromak}.

\vspace{0,2 cm}
Therefor, we shall consider  a slightly different notion of embedded graph. We call a \emph{multi-embedded graph in the plane} a triplet $\mathbb{G}=(\Vbb,\Ebb,\Fbb)$  satisfying the same conditions as an embedded graph  defined at the beginning of section \ref{section embedded graphs}, without  the condition  of simple connectivity on the faces  $\Fbb$.  The dual graph $\hat{\Gbb}$ is  then defined in the same way.  Let us fix a regular skein $\Sc$ (definition \ref{def skein} of section \ref{Asympto Wilson loop}). We let    $\Gbb_\Sc=(\Vbb_\Sc,\Ebb_\Sc, \Fbb_\Sc)$ be  the finest multi-connected embedded graph such that $\Sc\subset\Pd(\Gbb_\Sc).$   Let  $\Ebb^+$ and $\lambda$  be respectively  the orientation and the permutation of the edges $\Ebb$ induced by $\Sc$.
 For any $N\ge 1,$ the function $\Phi_N$ restricted to the class of $\Sc$ in $\overline{Sk}_r(\R^2)$ can be considered  as a smooth function on  $\R_+^{\Fbb_\Sc}$ constant along the coordinate  indexed by $F_\infty.$  Let us set 
\begin{align*}
\mu:  \R^{\Fbb}&\longrightarrow \R^{\Ebb^+} \\
u&\longmapsto \left(e \mapsto u(F_L(e))-u(F_R(e))-u(F_L(\lambda^{-1}(e)))+u(F_R(\lambda^{-1}(e)))\right)
\end{align*}
and denote by $\mathfrak{m}$ its transpose. Identifying the vector space of first order differential operators on $C^{\infty}(\R_+^{\Fbb})$ with  $(\R^*)^{\Fbb}$, any operator  (\ref{Makun}) and (\ref{Makdeux}) defined in  Theorem \ref{coromak} is of the form   $\mathfrak{m}( \partial_e),$  $e\in \Ebb^+$, where  $(\partial_e)_{e\in \mathbb{E}^+}$ denotes the canonical basis of $(\R^*)^{\Ebb^+}.$ To answer our question we need to identify the range of $\mathfrak{m}.$ For any loop $l$ respectively in  $\Ld(\Gbb)$ and $ \Sc,$ we denote   by $n_l\in \C^{\Fbb}$ and $\delta_l\in \C^{\Ebb^+}$ the winding number function of $l$ and the function $\displaystyle\sum_{e\in \mathbb{E}^+:l \text{ traverses } e}\delta_e$.  For any $v\in \Vbb,$ set $$*_v= \sum _{e\in \text{Out}(v)}\delta_e,$$
where $\text{Out}(v)$  denotes the set of oriented edges of $\Gbb,$ outgoing from $v$.  The following lemma is proved in \cite{MF}(Lemma 6.28.).
\begin{lem}\label{maknoyau}  

i) The kernel of $\mu$ is spanned by $\{n_{l_1},n_{l_2},\ldots, n_{l_m},1_{\Fbb}\}.$

ii)   If $\mathbb{G}_\Sc$ is an embedded graph, then the image of $\mu$ is the orthogonal space to $\{*_v,v\in\Vbb\}\cup \{\delta_l,l\in\Sc\}.$

\end{lem}
Let us denote by $(\frac{d}{d |F|})_{F\in \Fbb}$ the canonical basis $(\R^*)^{\Fbb}$. Note that we have a partial negative answer to our question: $\dim(\ker(\mu))\ge2 $ and $\Im(\mathfrak{m})+\R\frac{d}{d|F_\infty|}= \ker(\mu)^{\bot}+\R\frac{d}{d|F_\infty|}\not=(\R^*)^\Fbb.$ We need to complete the left-hand-side with a space of operators whose action on $\Phi(\Sc)$  is known.      When $m=1,$ if $F_0$ is a face of $\Gbb_\Sc,$ neighbor of  the infinite face,  then, it is shown in   (\cite{MF})  that  $\R \frac{d}{d|F_0|}$ answers this question.      In general, the following holds true.  

 \begin{lem}  \label{basediff} Suppose that $\mathbb{G}_\Sc$ is an embedded graph and that there exists $m$ distinct faces $F_1,\ldots,F_m$  of $\Gbb_\Sc$, neighbors of the unbounded face such that for any $i\in [m]$, $l_i$ is bounding $F_i$. Let $\Fbb_{\infty,1}=\{F_\infty,F_1,\ldots,F_m\}$. Then, 
 $$\Im(\mathfrak{m})\oplus \text{\emph{span}}\{\frac{d}{d|F|}: F\in \mathbb{F}_{\infty,1}\}=(\R^*)^\Fbb.$$
 \end{lem}

\begin{proof} By assumption, for any $i,j\in [m],$  $\frac{d}{d|F_i|} (n_{l_j})=\delta_{i,j}.$ As $\Im(\mathfrak{m})= \ker(\mu)^\bot,$ Lemma \ref{maknoyau}  implies  that  $\Im(\mathfrak{m})\cap\,\text{span}\left(\frac{d}{d|F_i|}:i\in[m]\right)=\{0\}$ and  $\dim(\Im(\mathfrak{m}))+m+1= \# \Fbb. $
\hfill\qed\end{proof}

\begin{dfn}  We call a  regular skein satisfying the condition of  Lemma \ref{basediff}   a   \emph{skein based at infinity}. \label{rooted skein}\end{dfn}
For any skein based at infinity, the differential operators appearing in Theorem  \ref{coromak} allow to express  any area derivative. Indeed according to the previous Lemma there exists\footnote{we shall give an explicit formula in the next section.} $(\mathcal{K}_{F,x})_{F\in \Fbb, x\in \Ebb^+\cup \Fbb_{\infty,1}},$ such that for any face $F\in \Fbb$,
\begin{align}
 \frac{d}{d|F|}=  \sum_{e\in  \Ebb_\Sc^+} \mathcal{K}_{F,e}  \mathfrak{m}(\partial_e)+\sum_{F'\in \Fbb_{\infty,1}}\mathcal{K}_{F,F'}  \frac{d}{d |F'|}. \label{InversMak}
\end{align}
Our task is now to show that  solutions to the differential problem of Theorem \ref{coromak}  are characterized by their value on skeins based at infinity.  \begin{rmk}  Note that if $\Sc$ satisfies the condition of Lemma  \ref{basediff} but  the condition of simple connectivity, then  one loop $l$ of $\Sc$ is disjoint from the others and under $\YM_N,$  $H_l$ is independent from $(H_{l'})_{l'\in\Sc\setminus \{l\}},$ so that $ \Phi_N(\Sc)=0. $\end{rmk}   A  problem is that the latter family of skeins is not stable by the operation (\ref{Makun}):   among the two skeins $\Sc_v^L$ and $\Sc_v^R$   obtained by splitting  $\Sc$ at $v$, one of them might not be based at infinity.   To solve this problem and compute the master field against  all skeins, we could enlarge the type of loops families. Instead, we shall use that any skein can be deformed into a skein based at infinity, without changing the conjugacy class of its elements in $\Pd(\R^2)$.

\vspace{0,2 cm}

For any $l\in \Pd(\Gbb_\Sc)$, we  consider   $$d_{\infty,\Sc}(l)= \inf\{d_{\hat{\Gbb}_{\Sc}}(F,F_\infty)-1: F\in \hat{\Gbb}_{\Sc}, F\cap F_{\infty,\Gbb_{\{l\}}}=\emptyset \}.$$   Let $\Vbb_s(\Sc)$ and $\Vbb_f(\Sc)$ be  respectively the  points  of self-intersection and of intersection  of two different loops and set   $I(\Sc)=\#\Vbb_s(\Sc)+ \#\Vbb_f(\Sc)$ (we shall drop the notation $\Sc$, when the context is not ambiguous). For any loop $l\in\Sc,$ denote by $I_\Sc(l)$ the number of intersections of $l$ with itself and other loops of $\Sc$. We define the \emph{complexity} of $\Sc$ to be the number

\begin{equation}
\Cc(\Sc)= I(\Sc) + 2 \sum_{l\in\Sc} d_{\infty,\Sc}(l). \label{complexity}
\end{equation}
 \begin{example} A skein $\Sc$ is based at infinity if and only if $\Cc(\Sc)=I(\Sc).$ 
 \end{example}
 \begin{example}  If $\Cc(\Sc)=0$, then $\Sc$ is an union of closed Jordan curved bounding disjoints domains. Therefore, $\Phi(\Sc)=0,$ if $\#\Sc\ge 2$ and $e^{-\frac{|D|}{2}}$, if $\Sc$ has a single loop bounding a simply connected domain $D.$ \end{example}
Recall that for any $v\in I_\Sc$, $\Sc_v$  denotes the transformed skein, whereas $l_v^L$ and $l_v^R$  stand for the two new loops of $\Sc_v$, when $v\in \Vbb_s$, as respectively defined  above Proposition \ref{Makeenko} and Theorem  \ref{coromak}.
 \begin{lem} \label{deccomplex}\begin{itemize}\item[i)]If $v\in \Vbb_f(\Sc)$, $$\Cc(\Sc_v)<\Cc(\Sc)$$
 and if $v\in\Vbb_s(\Sc)$, for any partition $\Sc_v^L\sqcup \Sc^R_v=\Sc_v$, with $l_v^L\in \Sc_v^L$ and $l_v^R\in \Sc_v^R,$
  $$\max\{\Cc(\Sc_v^L),\Cc(\Sc_v^R)\}<\Cc(\Sc).$$
 \item[ii)] For any regular skein $\Sc$, there exists a family  $(\Sc^\ep)_{\ep>0}$  of skeins based at infinity  with $\Cc(\Sc^\ep)=\Cc(\Sc),$ for any $\ep>0,$ that converges to  $\Sc'$, with $\Sc'\equiv \Sc.$ 
 \end{itemize}
 \end{lem}

 \begin{proof} i)  Assume that $v\in \Vbb_f$. Then,  for any loop $l\in \Sc$, that does not contain $v$, $l\in \Sc_v$ and $d_{\infty, \Sc_v}(l)\le d_{\infty,\Sc}(l)$. If $l_1$ and $l_2$ are the two loops crossing at $v$, then $d_{\infty, \Sc_v}(l_1\circ_vl_2)\le \min\{d_{\infty, \Sc}(l_1),d_{\infty, \Sc}(l_2) \}$.  Moreover, $\Vbb_s(\Sc_v)= \Vbb_s(\Sc)$ and for any  $w\in \Vbb_s(\Sc)$, $d_{\infty, \Sc_v}(l_v)\le d_{\infty, \Sc}(l_v)$. Therefore, the fact that $I(\Sc_v)=I(\Sc)-1$ yields  the expected inequality.

 \noindent Assume now that  $v\in \Vbb_s$. Let $l\in\Sc$ be the loop of  $\Sc$ crossing at $v$ and fix a partition $\Sc_v^L\sqcup \Sc_v^R$ of $\Sc_v$ separating the loops $l_v^L$ and $l_v^R$.  
 For any loop $l'\in\Sc_v\setminus \{l_v^L,l_v^R\}$, $d_{\infty,\Sc^L_v}(l'),d_{\infty,\Sc^R_v}(l')\le d_{\infty,\Sc}(l')$, whereas $\min_{a\in \{L,R\}}d_{\infty,\Sc } (l ^a_v)= d_{\infty,\Sc}(l).$  
 Let us suppose w.l.o.g.  that the latter minimum  is reached at $a=L,$ and  let $c\in \Pd(\hat{\Gbb}_{\Sc})$ be a path such that  $ F_{\infty,\Gbb_{\{l^L_v\}}}\cap\underline{c}=\emptyset$ 
 and $|c|-1= d_{\infty,\Sc}(l_v^L)=d_{\infty,\Sc}(l)$.  Then,  $I(\Sc_v^L)\le I(\Sc_v)-1$ and  as $d_{\infty,\Sc_v^L}(l_v^L)\le d_{\infty,\Sc}(l),$ $\Cc(\Sc^L_v)\le \Cc(\Sc)-1$.  The right 
 side needs more caution. Let us  consider  the two paths $c^{\pm}\in \Pd(\hat{\Gbb}_{\Sc_v^R})$   induced by   $l^L_v$ and $c$ in the following way:  $\underline{c}^{\pm}$ is the face on the left of the outgoing 
 edge of $l_v^R$ at $v$,   one path follows the orientation of $l_v^L$ and the other goes in the reverse direction, erasing loops chronologically, until they first hit  $c,$ when they both follow $c$ up to 
 $F_{\infty,\Gbb_{\Sc_v}}$.  Their combinatorial length satisfies

  $$|c^+|+|c^-|\le I_{\Sc_v}(l_v^L)+2  (|c|-1).$$
 Therefore, $$d_{\infty,\Sc^R}(l_v^R)\le  \min\{|c^+|,|c^-|\}\le \frac{I_{\Sc_v}(l_v^L)}{2}+ d_{\infty,\Sc}(l).$$ The number of intersections of  $\Sc_v^R$ is bounded by $I(\Sc\setminus\{l\})+I_{\Sc_v}(l_v^R)$. Moreover, for any loop $l'\in \Sc_v^R\setminus\{ l_v^R\}$, $d_{\infty,\Sc^R_v}(l')\le d_{\infty,\Sc}(l').$
 It follows that  $$\Cc(\Sc^R_v)\le  I(\Sc\setminus \{l\})+I_{\Sc_v}(l_v^R) +I_{\Sc_v}(l_v^L)+  2\sum_{l'\in \Sc} d_{\infty,\Sc}(l').$$
The equality   $I_{\Sc_v}(l_v^R) +I_{\Sc_v}(l_v^L)+ I(\Sc\setminus \{l\})=I(\Sc_v)= I(\Sc)-1$  implies the claim.

ii)  For each $l\in\Sc$, such that $d_{\infty,\Sc}(l)>0$, consider a self-avoiding path $c_l$ in $ \hat{\Gbb}_\Sc,$ such that $ F_{\infty, \Gbb_l}\cap \underline{c}_l=\emptyset, \overline{c}_l =F_{\infty,\Gbb_\Sc}$ and $|c_l|= d_{\infty,\Sc}(l)+1$.  Choose such a family $(c_l)_{l\in \Sc}$   of loops that do not cross each other but may be merged with one another. Deform each loop $l$ along $c_l$  into $\tilde{l}$  so that the deformation intersects exactly twice each dual edge  of $c_l$ and does not intersect the deformation of other loops. Denote by $\tilde{\Sc}$ the skein  $\{\tilde{l}: l\in \Sc, d_{\infty,\Sc}(l)>0\}\cup\{l: l\in \Sc, d_{\infty,\Sc}(l)=0\}$.  By construction, for any $l\in\tilde{\Sc}$, $d_{\infty,\tilde{\Sc}}(l)=0$ and $$\Ic(\tilde{\Sc})=I(\Sc)+2 \sum_{l\in\Sc}(|c_l|-1)=\Cc(\Sc).$$
 We can now choose a family of skeins $(\Sc^\ep)_{\ep>0}$ as above, that converges  towards a skein $\Sc'$, such that $\Sc'\equiv \Sc$, as $\ep\to0.$
 \hfill\qed\end{proof}

 We can now solve our differential system recursively ordering skeins by their complexity. Recall that if $x$ is a point of intersection of a skein $\Sc$, then $\mu_x=\frac{d}{d|F_1|}-\frac{d}{d|F_2|}+\frac{d}{d|F_3|}-\frac{d}{d|F_4|}$, where $F_1,F_2,F_3$ and $F_4$ are faces around the vertex $v$ in cyclic order and $F_1$ is the face bounded by the two outgoing edges of $x$. 
 \begin{thm} \label{Thm caracterisation par MM} There exists a unique function $\Phi$ on  $\Sk(\R^2)$ satisfying the following equations\footnote{ the set of skein is defined as in section \ref{Asympto Wilson loop} with the topology associated to $d_\ell$ (defined in section \ref{section continuous YM}). }: 
 \begin{enumerate}[1.]
 \item $\Phi(\{1\})=1.$
 \item If $\Sc^-$ and $\Sc^+$ are two skeins that are separated by a closed Jordan curve, $\Phi(\Sc^-\sqcup\Sc^+)=0$.
 \item $\Phi$ is continuous.
 \item If $\Sc\equiv \Sc'$ (definition on p. \pageref{def conj skeins}), then $\Phi(\Sc')=\Phi(\Sc).$ \item  For any area-preserving  diffeomorphism $g$ of the plane, $\Phi\circ g =\Phi.$
 \item  For any regular skein  $\Sc$, $\Phi$ is differentiable with respect to $(|F|)_{F\in \Fbb_\Sc}$ and satisfies the following differential equations. If $x$ is the intersection of two different loops,
 \begin{equation}
\mu_x \Phi(\Sc) =\Phi(\Sc_x) .
\end{equation}
  If $x$ is the intersection of a loop $l$ of $\Sc$ with itself,
  \begin{equation}
\mu_x \Phi(\Sc)= \sum_{\substack{\Sc_x^L\sqcup\Sc_x^R=\Sc_x\\ l^L_{1,x}\in \Sc_x^L \text{ and } l^R_{1,x}\in \Sc_x^R}}\Phi_N(\Sc_x^L)\Phi_N(\Sc_x^R).
\end{equation}
For any face $F\in\Fbb_\Sc$, neighbor of $F_\infty$,
\begin{equation}
\frac{d}{d|F|} \Phi(\Sc) =-\frac{1}{2}\Phi(\Sc). 
\end{equation}
\end{enumerate}
 \end{thm} 
 
 \begin{proof}  The function $\Phi_N$ satisfies by construction the point 1-5.  According to Theorem  \ref{coromak} and \ref{ConvCMSUP}, for any regular skein $\Sc$, $\Phi_N(\Sc)$ is analytic in $(|F|)_{F\in\Fbb_\Sc}$, satisfies $(*)$, $(**)$ and $(***)$ and converges uniformly on every compact set of $\R_+^{\Fbb_\Sc}$ to the function $\Phi(\Sc)$. Therefore, $\Phi(\Sc)$ is analytic and satisfies the equations of point 6. It remains to show uniqueness of the solutions of the latter problem. We wish to prove it by induction on the complexity (\ref{complexity}).  Let $\Psi$ be a function on finite skeins satisfying point 1 to 6.   Using point 3, it is enough to prove that $\Psi(\Sc)=\Phi(\Sc)$ for any regular skein $\Sc$.  For any integer $n$, set $\mathsf{Sk}_n=\{\Sc\in \Sk_r(\R^2): \Cc(\Sc)\le n\}$. Let us prove inductively that $\Psi_{|\Sk_n}=\Phi_{|\Sk_n}$. Thanks to points 1 and 2, the equality holds for $n=0$. Assume that it is true for $n\in\N$ and consider a regular skein $\Sc\in\Sk_{n+1}$. Suppose that $\Sc$ is based at  infinity. According to  the inversion formula (\ref{InversMak}) together with point 6 and Lemma \ref{deccomplex}, for any face $F\in\Fbb_\Sc,$ $\frac{d}{d|F|}\Phi$ and $\frac{d}{d|F|}\Psi$ are a linear combination of  terms of the form  $\Phi(\Sc'),\Psi(\Sc')$ or $\Phi(\Sc^L)\Phi(\Sc^R),\Psi(\Sc^L)\Psi(\Sc^R)$, with $\Cc(\Sc'),\Cc(\Sc^L),\Cc(\Sc^R)<n$. Hence, by induction hypothesis, $\Psi(\Sc)=\Phi(\Sc)$. Assume now that $\Sc$ is not based at infinity. Let $(\Sc^\ep)_{\ep>0}$  be given as in Lemma  \ref{deccomplex}. Then, for any $\ep>0$,  $\Sc^\ep\in\Sk_{n+1}$ is based at infinity and $\Psi(\Sc^\ep)=\Phi(\Sc^\ep)$.  The   points  3  and  4  yield that $\Psi(\Sc)=\Phi(\Sc)$.
 \hfill\qed\end{proof}

\subsection{Generalized Kazakov basis\label{General Kazakov}}   We consider a  skein $\Sc$  based at infinity\footnote{recall definition \ref{rooted skein} of the last section} and set  $\Gbb=\Gbb_\Sc$.   We  shall give here bases adapted to the decomposition  $\Im(\mathfrak{m})\oplus \text{span}\{\frac{d}{d|F|}: F\in \Fbb_{\infty, 1}\}$ and their dual, which leads to an explicit formula for  the matrix $\mathcal{K}$ appearing in (\ref{InversMak}). We  call them   \emph{Kazakov bases},  following   the works  \cite{KazakovMF,MF}, where they were introduced in the case $\#\Sc=1$.

For any $l\in\Sc$, let  $\tilde{l}$ and $e_l\in\Ebb$  be  the non-based loop associated to $l$  and    the edge dual  to $(F_{\infty}, F_l)$, where $\Fbb_{\infty,1}=\{F_l, l\in \Sc\}\cup \{F_\infty\}.$ For any pair  $l,l'\in \Sc$  of distinct  intersecting loops, let us fix  a vertex  $p_{l,l'}\in \Vbb_f$ at the intersection  of $l,l',$  and set $\Vbb_0$ the collection of these points. We denote by $\prec$ an arbitrary order on $\Sc$. A  Kazakov basis is described thanks to   families of loops in $\Ld(\Gbb)$.      For any $v\in \Vbb\setminus \Vbb_0,$ we define  a loop $l_v$  based at $v$ as follows:

\vspace{0,2 cm}



\noindent i)   If  $v\in \Vbb_s,$ $l_v$  is the restriction of $l$   between two visits of  $v$ that does not use $e_l$. 

\vspace{0,1 cm}

\noindent ii)  If $v\in \Vbb_f\setminus \Vbb_0$ is at the intersection  of  $l,l'\in \Sc$  with $l\prec l'$,   $l_v$ is the   concatenation of  the restrictions of  $\tilde{l}$ or $\tilde{l}^{-1}$  between the hitting time of  $v $ and  $p_{l,l'}$, with the restriction of  $\tilde{l'}$ between the hitting time of  $p_{l,l'}$ and $v,$ that does not use the edges $e_l$ and $e_{l'}.$ 
\vspace{0,1 cm}

Consider   the graph $\Gc_\Sc$ with vertices indexed by $\Sc$, such that two loops are connected in $\Gc_\Sc$ if and only if they intersect each other. For any $v\in \Vbb$, let  $e(v)$ be the left-outgoing edge at $v$. The winding number of a non-backtracking  loop $l\in \Ld(\Gbb)$ jumps by $1$  along any dual edge of $\Gbb$ that crosses $l$ from right to left.  Hence, for any  loop $l\in \Ld_v(\Gbb),$   using exactly two    edges  around $v,$ bounding the same face,
$$   \ep(l)=  \mathfrak{m}(\partial_{e(v)})(n_l)\in \{ -1,1 \}, $$
whereas  for any $l\in \Sc,$
$$\ep_l= \frac{d}{d|F_l|} (n_l) \in \{-1,1\}.$$



   
%
%

\begin{lem} If $\mathcal{G}_\Sc$ is a tree, 
$$\beta=\{\mathfrak{m}(\partial_{e(v)}) , e\in \Vbb\setminus \Vbb_0\} \cup \{\frac{d}{d|F|}-\frac{d}{d|F_\infty|}: F\in \Fbb_{\infty,1}\}\cup \{\frac{d}{d|F_\infty|}\} $$
is a basis of  $(\R^*)^\Fbb,$ with dual 
$$\alpha=\{\ep(l_v)n_{l_v}: v\in \Vbb \setminus  \Vbb_0 \}\cup \{\ep_ln_l:l\in \Sc\}\cup \{ 1_\Fbb\}.$$
\end{lem}
\begin{proof} For any loop $l$ in $\Ld(\Gbb)$ and $e\in \Ebb^+,$ such that $l$ uses  exactly  $e$ and $\lambda^{-1}(e)$ among the four edges adjacent to $\underline{e},$   $\mathfrak{m}(\partial_e)(n_l)=0$. Besides for any loop $l$ belonging to the families i) or ii), and any $a,b\in \Sc,$  $\frac{d}{d|F_a|} (n_l)=0,$ as $l$ does not use $e_a$, for any $v\in \Vbb\setminus \Vbb_0,$   $\mathfrak{m}(\delta_{e(v)})(n_l)= \ep_v(l)\delta_{v, \underline{l}} $,  $\frac{d}{d|F_a|} n_b= \ep_a \delta_{a,b},$  $\frac{d}{d|F_\infty|}(n_l)=0$ and $\frac{d}{d|F|} ( 1_\Fbb  )=1,$ for any face $F\in \Fbb.$  It follows that the two families $\alpha$ and $\beta$ are free and dual to each other, with rank $\#\Vbb-\#\Vbb_0+\#\Sc+1$. Any vertex of $\Gbb$ has degree $4$, hence by Euler's relation, $\#\Vbb= \#\Fbb -2.$    Besides, there are as many points in $\Vbb_0$  as there are couples of distinct intersecting loops, therefore $\Vbb_0$ is the number of edges of $\mathcal{G}_\Sc.$ If $\mathcal{G}_\Sc$ is a tree,    $\#\Vbb_0= \#\Sc-1$ and the rank of $\alpha$   is $\#\Fbb.$      \hfill\qed\end{proof}

If $\mathcal{G}_\Sc$ is not a tree the latter family is not a basis anymore and can be modified as follows. Let $\mathfrak{T}$ and $E^+_\Sc$ be a spanning tree of $\mathcal{G}_\Sc$  and an arbitrary orientation of $\mathcal{G}_\Sc$ and denote by  $\Vbb(\mathfrak{T}),$  the set of vertices at the  intersection of a pair of loops $(l,l')$ that is  an edge of $\mathfrak{T}.$
\vspace{0,2 cm}

\noindent i) and ii) For $v\in  \Vbb_s\cup (\Vbb_f\cap \Vbb(\mathfrak{T})),$ the definition of $l_v$  is not changed. 

\vspace{0,1 cm}

\noindent iii) Let $v\in \Vbb_f\setminus \Vbb(\mathfrak{T})$   be an intersection point of two loops $l,l',$ such that $(l,l')\in E_\Sc^+$ is not an edge of $\mathfrak{T}.$   Let us consider the path  $l_0,l_1,\ldots, l_m$  in $\mathfrak{T},$ with $l_0=l$ and $l_m=l',$ and set  $p_0=v=p_{m+1}$ and $p_k= p_{l_{k-1}, l_{k}},$ for $1\le k\le m.$ For any $0\le k\le m,$ fix  a path $\gamma_k$  from $p_k$ to $p_{k+1},$  restriction of $\tilde{l}_k$ or $\tilde{l}_k^{-1},$ that is not using $e_{l_k}.$   We define a loop  based at $v$  setting   $l_v=\gamma_0\gamma_1\ldots\gamma_m.$ See figure \ref{TroisCercles} for an example.  
 \begin{figure}[!h] 
  \centering
 \includegraphics[height=2in]{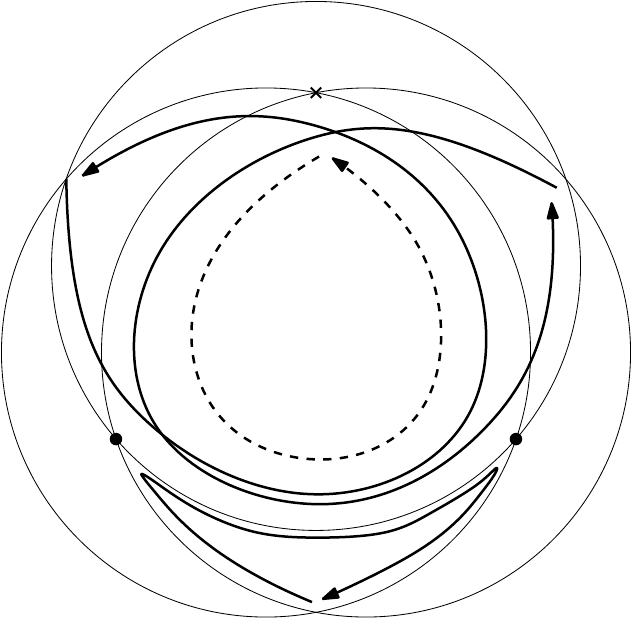}\caption{ \label{TroisCercles}  A Kazakov basis for three circles  with counterclockwise orientation.  The set $\Vbb_0$ is given by  vertices of the inner triangle and $\Vbb(\mathfrak{T})$ is drawn with black disks. There are no type i) loops and type ii) and iii) loops are drawn respectively with plain and  dashed lines.}
\end{figure}


\begin{lem} The family 
$$\beta'=\{\mathfrak{m}(\partial_{e(v)}) , v\in \Vbb\setminus (\Vbb_0\cap \Vbb(\mathfrak{T})) \} \cup \{\frac{d}{d|F|}-\frac{d}{d|F_\infty|}: F\in \Fbb_{\infty,1}\}\cup \{\frac{d}{d|F_\infty|}\} $$
is a basis of  $(\R^*)^\Fbb$ with dual 
$$\alpha'=\{\ep(l_v)n_{l_v}: v\in \Vbb \setminus  (\Vbb_0 \cap \Vbb(\mathfrak{T}))\}\cup \{\ep_ln_l:l\in \Sc\}\cup \{ 1_\Fbb\}.$$
\end{lem}
\begin{proof} If $a$ and $b$ are two distinct  loops of the  families i), ii) or iii),     whether $b$  uses the outgoing edge $e\in \text{Out}(\underline{a})$ and $\lambda^{-1}(e)$ and no other edge around $v$, or it does go through $v$, in both cases,   $\mathfrak{m}(\partial_{e(\underline{a})} )(n_{b})=0.$ As $b$ does not use any edge in $\{e_l,l\in\Sc\}$, for any $F\in \Fbb_{\infty,1},$  $\frac{d}{d|F|} (n_b)=0.$   It follows that $\alpha'$ and $\beta'$ are free and dual to each other.   Besides, $\# \Vbb_0\cap \Vbb(\mathfrak{T})$ is the number of edges of $\mathfrak{T},$ that is, $\#\Sc-1.$  Therefore,  $\alpha',\beta'$ have rank  $\#\Vbb- \#(\Vbb_0\cap \Vbb(\mathfrak{T}))+ \#\Sc +1=\#\Vbb+2=\#\Fbb$.  \hfill\qed\end{proof}
We have now an explicit expression for $\mathcal{K}$ in (\ref{InversMak}). For any  regular skein  $\Sc$ based  at infinity and  any bounded face $F\in \Fbb_\Sc,$ 
$$\frac{d}{d|F|}= \sum_{v\in \Vbb\setminus (\Vbb_0\cap \Vbb(\mathfrak{T}))} \ep(l_v)n_{l_v}(F) \mathfrak{m}(\partial_{e(v)}) +\sum_{l\in \Sc}\ep_l n_l(F)  \frac{d}{d|F_l|}.$$

\section{Acknowledgments.}The author wishes to thank his PhD advisor Thierry L\'evy, as well  as Franck Gabriel and Guillaume C\'ebron, for fruitful discussions about Yang-Mills measure and the planar master field.  Many thanks are  due to the anonymous referee for useful comments. This research work has been partly funded by the RTG 1845 and the EPSRC grant New Frontiers in Random Geometry.

\def\cprime{$'$}


\end{document}